\documentclass{article}

\usepackage{PRIMEarxiv}

\usepackage[utf8]{inputenc} 
\usepackage[T1]{fontenc}    
\usepackage{hyperref}       
\usepackage{url}            
\usepackage{booktabs}       
\usepackage{amsfonts}       
\usepackage{amsmath,amssymb}
\usepackage{mathtools}
\usepackage{nicefrac}       
\usepackage{microtype}      
\usepackage{lipsum}
\usepackage{fancyhdr}       
\usepackage{graphicx}       
\usepackage{amsthm}
\usepackage{times}
\usepackage{subfig}
\usepackage{xcolor}
\usepackage{array}
\usepackage{cite}
\usepackage{multirow}
\usepackage{adjustbox}
\usepackage[scr=dutchcal]{mathalfa}
\usepackage[font=small,labelfont=bf]{caption}
\usepackage{enumitem}
\usepackage{tikz}
\usetikzlibrary{arrows,matrix} 

\newtheorem{theorem}{Theorem}
\newtheorem{lemma}{Lemma}
\newtheorem{proposition}{Proposition}

\newcommand{\motifs}{motif}
\newcommand{\IFFM}{IFFM}

\usepackage{tikz-network}
\tikzset{
    block/.style = {draw, fill=white, rectangle, minimum height=.75cm, minimum width=1.25cm},
    Circle/.style = {draw, thick, fill=white, circle, minimum height=.75cm, minimum width=1.25cm},
    tmp/.style  = {coordinate}, 
    sum/.style= {draw=white, fill=white, circle, minimum height=.75cm, minimum width=1.25cm},
    input/.style = {coordinate},
    output/.style= {coordinate},
    pinstyle/.style = {pin edge={to-,thin,black}
    }
}

\title{
When is cumulative dose response monotonic?\\ Analysis of incoherent feedforward {\motifs}s
\thanks{This research was partially supported by
ONR grant N00014-21-1-2431 and AFOSR grant 792309.}
}

\author{
  Moh Kamalul Wafi$^1$, Arthur C. B. de Oliveira$^1$, and Eduardo D. Sontag$^{1,2}$ \\
  $^1$Department of Electrical and Computer Engineering. \\
  $^2$Department of BioEngineering and affiliated with the Departments of Chemical Engineering and Mathematics. \\
  Northeastern University \\
  Boston, USA\\
  \texttt{\{wafi.m, a.castello, e.sontag\}@northeastern.edu} \\
}

\begin{document}
\maketitle

\begin{abstract}
We study the monotonicity of the cumulative dose response (cDR) for a class of incoherent feedforward {\motifs} ({\IFFM}) systems with linear intermediate dynamics and nonlinear output dynamics. While the instantaneous dose response (DR) may be nonmonotone with respect to the input, the cDR can still be monotone. To analyze this phenomenon, we derive an integral representation of the sensitivity of cDR with respect to the input and establish general sufficient conditions for both monotonicity and nonmonotonicity. These results reduce the problem to verifying qualitative sign properties along system trajectories. The structural conditions identify eight canonical {\IFFM} models,
four of which are analyzed in detail within the proposed framework. In particular, {\IFFM}1 and {\IFFM}3 exhibit monotone cDR despite potentially nonmonotone DR, while {\IFFM}2 is monotone already at the level of DR, which implies monotonicity of cDR. In contrast, {\IFFM}4 exhibits nonmonotone cDR. Numerical simulations indicate that these properties persist beyond the structured initial conditions used in the analysis. Overall, our results provide a unified framework for understanding how system structure shapes monotonicity in cumulative input--output
responses.
\end{abstract}
\allowdisplaybreaks


\section{Introduction}\label{sec:intro}

Biological systems, from single cells to whole organisms, can adjust their responses to environmental signals. A key feature of this ability is maintaining internal stability under sustained inputs while still responding transiently to changes. This behavior, known as \emph{adaptation} or \emph{homeostasis}, arises from the interplay of signal transduction and gene regulation dynamics \cite{ref1,two-zone-journal}.

A system exhibits \emph{perfect adaptation} when its steady-state output is independent of a constant input~\cite{ref3}.
Although adaptation is an asymptotic property, many applications
concern finite-time behavior~\cite{2025_arxiv_dynamic_phenotypes}.
At a fixed observation time, the dependence of the output on the input
magnitude defines the \emph{dose response} (DR), while the
\emph{cumulative dose response} (cDR) measures the total output
accumulated over a time interval, often reported experimentally as the
``area under the curve.''
A key example of cDR arises in T-cell recognition, where immune cells must respond to antigens without overactivation \cite{2025_arxiv_dynamic_phenotypes}, and the response is quantified by measuring accumulated cytokines over a finite period. 

Incoherent feedforward {\motifs}s ({\IFFM}s) define a class of network
motifs known to generate adaptive responses. In these systems, the
input affects the output both directly and indirectly through an
intermediate regulatory pathway with an opposing
effect~\cite{2025_arxiv_dynamic_phenotypes}; see
Fig.~\ref{fig:1a}. This antagonistic structure enables such motifs
to exhibit adaptation and detect changes in the input; see
Fig.~\ref{fig:1b} and~\cite{Alon2006-sg,shoval10,shoval_alon_sontag_2011}
for studies using the simplest models.

\begin{figure}
    \centering

    \subfloat[\label{fig:1a}]{%
        \begin{minipage}[c]{0.4\linewidth}
            \centering
            \begin{minipage}[c]{0.45\linewidth}
                \centering
                \begin{tikzpicture}[>=stealth, thick]
                    \node (u) at (0,4) {$u$};
                    \node (x) at (0,2) {$x$};
                    \node (y) at (0,0) {$y$};

                    \draw[->] (u) -- (x);
                    \draw[red] (x) -- ($(y)+(0,0.3)$);
                    \draw[red] ($(y)+(-0.15,0.3)$) -- ($(y)+(0.15,0.3)$);
                    \draw[->] (u) -| ($(y)+(-.75,0)$) -- (y);
                \end{tikzpicture}
            \end{minipage}
            \hfill
            \begin{minipage}[c]{0.45\linewidth}
                \centering
                \begin{tikzpicture}[>=stealth, thick]
                    \node (u) at (0,4) {$u$};
                    \node (x) at (0,2) {$x$};
                    \node (y) at (0,0) {$y$};

                    \draw[->] (u) -- (x);
                    \draw[->] (x) -- (y);
                    \draw[red] (u) -- ++(-.75,0) |- ($(y)+(-0.18,0)$);
                    \draw[red] ($(y)+(-0.18,-0.15)$) -- ($(y)+(-0.18,0.15)$);
                \end{tikzpicture}
            \end{minipage}
        \end{minipage}
    }
    \subfloat[\label{fig:1b}]{%
        \begin{minipage}[c]{0.33\linewidth}
            \centering
            \includegraphics[width=\linewidth]{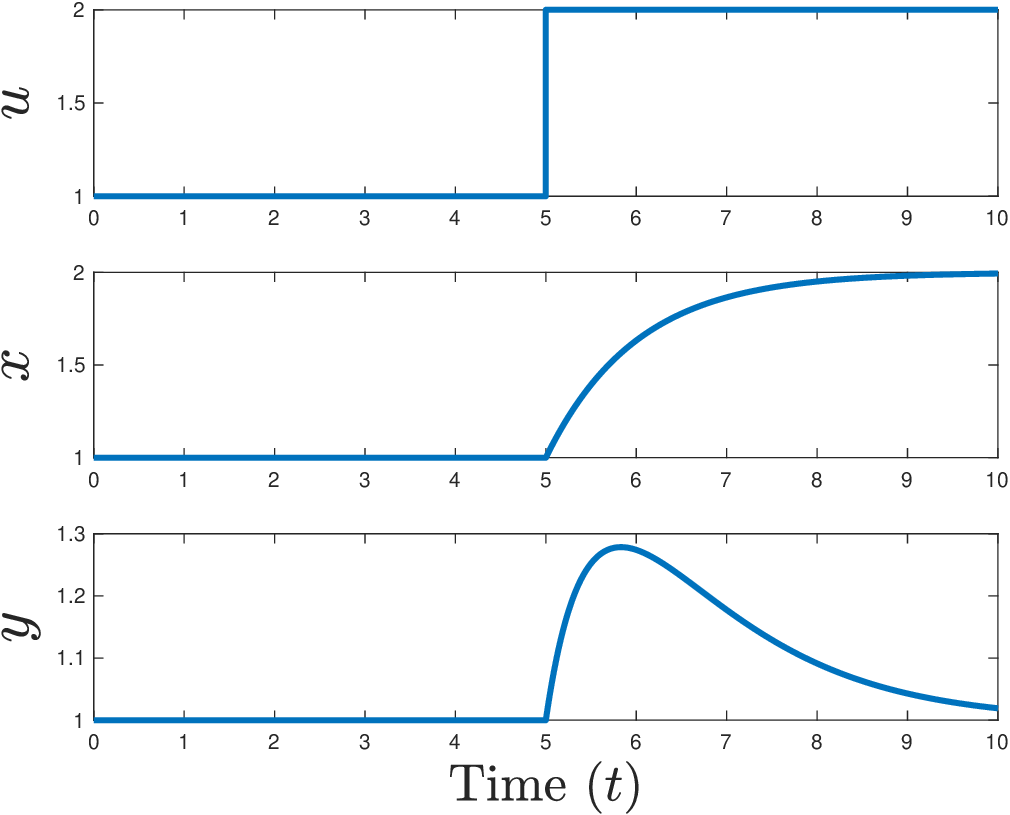}
        \end{minipage}
    }

    \caption{(a) {\IFFM}s with inhibition via $x$ (left) or directly from $u$ (right). (b) An example: suppose that the external input $u$ changes from a baseline value $u=1$ to a new constant value $u=2$ at time $5\,\mathrm{s}$, and $x(0)=1$, $y(0)=1$ (steady state for $u=1$ in system $\dot x = -x+u$, $\dot y = u/x-y$). This acts as a ``change detector'': $\Delta u(t)$ triggers an activity burst, followed by a return to the adapted value, $y(t) \rightarrow y(0)$.}
    \label{fig:iffl}
\end{figure}
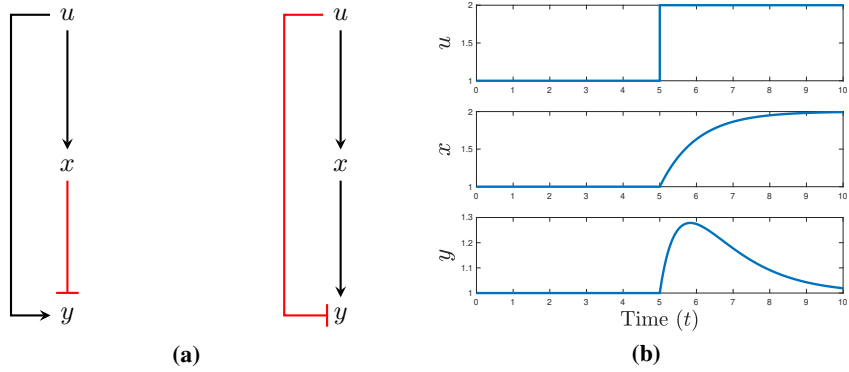

As explained in~\cite{Ankit_EDS,2025_arxiv_dynamic_phenotypes}, experimental work leads one to the following question: how do the DR and cDR, at a given finite time, depend on the magnitude of a constant stimulus input $u$? Specifically, when is the cDR monotonic as a function of the input magnitude $u$? Note that this is a \textit{very different question from asking whether
the output is monotonic as a function of time}. From now on, when we refer to monotonicity, we mean in this sense of dependence on input magnitude. An interesting phenomenon, discussed in the above papers, is that even if the DR happens to be nonmonotonic, the cDR could be monotonic, and this makes the analysis much harder. 

Building on cDR monotonicity results for two scalar {\IFFM} models~\cite{Ankit_EDS}, our contributions are:
\begin{itemize}[leftmargin=*]
    \item We introduce a general {\IFFM} class with $n$-dimensional linear intermediate dynamics and nonlinear output dynamics.
    
    \item We derive an integral representation of the cDR sensitivity with respect to the input and establish sufficient conditions for monotonicity and nonmonotonicity based on qualitative sign properties along system trajectories.
    
    \item We characterize the cDR monotonicity of the eight canonical {\IFFM} models, including cases with monotone cDR despite nonmonotone DR.
    
    \item We identify two novel {\IFFM} mechanisms producing nonmonotone cDR.
\end{itemize}

These results provide a unified perspective on monotonicity of
cumulative responses and clarify the role of system structure in
shaping input--output behavior.

\section{Theoretical Framework}\label{sec:Problem}

In this paper, we let $\mathbb{R}_+^{n\times m} := \{A\in\mathbb{R}^{n\times m}\mid A> 0\}$ be the set of $n\times m$ matrices with positive entries. The sets $\mathbb{R}_+^{n}$ and $\mathbb{R}_+$ are defined similarly. The orders $>,<,\geq$, and $\leq$ are interpreted elementwise when applied to vectors and matrices. The symbol $\mathbf{0}_p$ is a zero vector of dimension $p$.

\subsection{Incoherent feedforward {\motifs}s ({\IFFM}s) dynamics}
\label{subsec:IFFM}

By definition, an incoherent feedforward motif ({\IFFM}) is a system
in which the input $u$ affects the output $y$ both directly and
indirectly through an intermediate state $x$, with the two pathways
acting in \textit{opposite} directions. 
The intermediate state acts as a filtered or \textit{delayed}
representation of $u$.

We consider the linear intermediate dynamics
\begin{equation}
\label{eq:x_general}
    \dot{x}(t)=Ax(t)+bu,
\end{equation}
where $x\in\mathbb{R}^n$ denotes the state with
$x(0)\in\mathbb{R}_{\geq0}^n$, $u\in\mathbb{R}_+$ is constant,
$A\in\mathbb{R}^{n\times n}$ is Metzler and Hurwitz, and
$b\in\mathbb{R}_+^n$.
Therefore, the $x$--subsystem is positive\footnote{The restriction to
positive systems is natural in biological systems, where the
states typically represent molecular concentrations.} and
$x(t)$ converges exponentially to
$x_{\mathrm{ss}}(u)=-A^{-1}bu$.

The output dynamics are
\begin{equation}
\label{eq:y_general}
    \dot{y}(t)=F(x(t),y(t),u),
\end{equation}
where $y\in\mathbb{R}$ and
$F:\mathbb{R}^n\times\mathbb{R}\times\mathbb{R}_+
\rightarrow\mathbb{R}$ is continuously differentiable on the
domain of interest. The direct effect of $u$ on $y$ is characterized
by $\partial_uF$, whereas the effect of the indirect pathway
$u\to x\to y$ is characterized by $\partial_xF$, owing to positivity
of the $x$--subsystem. Thus, the system \eqref{eq:x_general}--\eqref{eq:y_general}
has an incoherent feedforward structure if
\begin{equation}
\label{eq:incoherent_condition}
    \partial_uF(x,y,u)\,
    \partial_xF(x,y,u)
    \leq\mathbf{0}_n^\top,
\end{equation}
where the product is not identically zero.

To obtain the canonical {\IFFM} models studied in this work, we
specialize the output dynamics to
\begin{equation}
\label{eq:y_IFFM}
    \dot{y}(t)
    =
    \underbrace{\gamma(c^\top x(t))^{\kappa_1}(\beta u)^{\ell_1}}_{\text{production}}
    -
    \underbrace{d(c^\top x(t))^{\kappa_2}(\beta u)^{\ell_2}y(t)}_{\text{degradation}},
\end{equation}
where $c\in\mathbb{R}_+^n$, $d,\beta,\gamma>0$, and
$\kappa_1,\kappa_2,\ell_1,\ell_2\in\mathbb{Z}$.
Since $c^\top(-A^{-1}b)>0$, the quantity
$c^\top x_{\mathrm{ss}}(u)$ is positive and proportional to $u$
for every $u\in\mathbb{R}_+$.
The following structural conditions are imposed:

\begin{enumerate}[leftmargin=*]
    \item We exclude homodimerization and other higher-order interactions
    by restricting the exponents to
    \begin{equation}
        \kappa_i,\ell_i\in\{-1,0,1\},
        \qquad i\in\{1,2\}.
    \end{equation}
    Consequently, quadratic and higher-order terms, such as
    $(c^\top x)^2$, $u^2$, and their higher powers, are not considered.

    \item Each of the intermediate signal $c^\top x$ and the input $u$
    has a single regulatory role, either as an activator or an inhibitor.
    Mathematically,
    \begin{equation}
    \label{eq:exponent_constraints}
        |\kappa_1|+|\kappa_2|=1,
        \qquad
        |\ell_1|+|\ell_2|=1.
    \end{equation}
    Hence, exactly one of $\kappa_1$ and $\kappa_2$ is nonzero, and
    exactly one of $\ell_1$ and $\ell_2$ is nonzero.

    \item The production and degradation terms have equal total degrees in $(c^\top x,u)$:
    \begin{equation}\label{eq:degree_balance}
        \kappa_1+\ell_1
        =
        \kappa_2+\ell_2.
    \end{equation}
    Setting $\dot y=0$ in \eqref{eq:y_IFFM} gives
    \begin{equation}
        y_{\mathrm{ss}}(u)
        =
        \frac{\gamma}{d}
        \bigl(c^\top x_{\mathrm{ss}}(u)\bigr)^{\kappa_1-\kappa_2}
        (\beta u)^{\ell_1-\ell_2},
    \end{equation}
    and hence
    $y_{\mathrm{ss}}(u)\propto u^{(\kappa_1+\ell_1)-(\kappa_2+\ell_2)}.$
    Thus, \eqref{eq:degree_balance} ensures that $y_{\mathrm{ss}}$ is
    independent of the input $u$ (perfect adaptation), while the transient
    response of $y$ depends on $u$.
\end{enumerate}

\begin{table}[h!]
    \centering
    \caption{Admissible {\IFFM} models, where
    $c\in\mathbb{R}_+^n$ and $d,\beta,\gamma>0$.}
    \renewcommand{\arraystretch}{1.35}
    \begin{tabular}{|c|cccc|c|}
    \toprule
    No.
    & $\kappa_1$
    & $\ell_1$
    & $\kappa_2$
    & $\ell_2$
    & Output dynamics \\
    \midrule
    1
    & $1$ & $-1$ & $0$ & $0$
    & $\displaystyle \dot y=\frac{\gamma c^\top x}{\beta u}-dy$
    \\

    2
    & $1$ & $0$ & $0$ & $1$
    & $\displaystyle \dot y=\gamma c^\top x-d\beta u y$
    \\

    3
    & $-1$ & $1$ & $0$ & $0$
    & $\displaystyle \dot y=\frac{\gamma\beta u}{c^\top x}-dy$
    \\

    4
    & $-1$ & $0$ & $0$ & $-1$
    & $\displaystyle \dot y=\frac{\gamma}{c^\top x}-\frac{dy}{\beta u}$
    \\

    5
    & $0$ & $1$ & $1$ & $0$
    & $\displaystyle \dot y=\gamma\beta u-d(c^\top x)y$
    \\

    6
    & $0$ & $0$ & $1$ & $-1$
    & $\displaystyle \dot y=\gamma-\frac{d(c^\top x)}{\beta u}y$
    \\

    7
    & $0$ & $-1$ & $-1$ & $0$
    & $\displaystyle \dot y=\frac{\gamma}{\beta u}-\frac{dy}{c^\top x}$
    \\

    8
    & $0$ & $0$ & $-1$ & $1$
    & $\displaystyle \dot y=\gamma-\frac{d\beta u}{c^\top x}y$
    \\
    \bottomrule
    \end{tabular}
    \label{tab:systems}
\end{table}

Combining the structural conditions above yields the eight canonical
{\IFFM} models listed in Table~\ref{tab:systems}.
In biochemical models, the effect of the intermediate signal
is commonly represented by a Michaelis--Menten-type factor
\begin{equation*}
    \frac{1}{K+c^\top x},
    \qquad K\geq0,
\end{equation*}
where $K$ is the Michaelis--Menten constant. The models in
Table~\ref{tab:systems} correspond to the idealized case $K=0$.
When $K>0$, the denominator remains positive at $c^\top x=0$,
but perfect adaptation is generally lost.

In this work, we add $K$ only to denominators involving $c^\top x$.
One could also replace $(\beta u)^{-1}$ with
$(K_u+\beta u)^{-1}$, where $K_u\geq0$. We retain the idealized form
$(\beta u)^{-1}$ and do not consider this change here.

For the canonical models, we consider initial conditions
$x_0\in\mathbb{R}_{\geq0}^n$ and
$y_0\in\mathbb{R}_{\geq0}$. When a model contains
$(c^\top x)^{-1}$ with $K=0$, we assume $c^\top x_0>0$. Positivity of the
$x$--subsystem then gives $c^\top x(t)>0$ for all $t\geq0$, and
\[
    c^\top x(t)\longrightarrow
    c^\top x_{\mathrm{ss}}(u)
    =
    c^\top(-A^{-1}b)u>0.
\]
Consequently, $c^\top x(t)$ is bounded away from zero on
$[0,\infty)$. Indeed, it has a positive lower bound for all
sufficiently large $t$, while continuity and strict positivity give a
positive minimum on the remaining compact time interval.

The output dynamics have the form
$\dot y=\phi_1(t)-\phi_2(t)y$, where
$\phi_1(t),\phi_2(t)\geq0$, and therefore preserve nonnegativity.
Boundedness of $x(t)$ and the positive lower bounds on the denominators
imply that these coefficients remain bounded. Moreover,
$\phi_2(t)$ is bounded away from zero for all sufficiently large $t$.
Hence, the positive orthant is forward invariant, and the solutions
are forward complete and bounded.

\subsection{Dose response and cumulative dose response}
\label{subsec:DR_cDR}

For $T>0$ and each $u\in\mathbb{R}_+$, let $(x_u(t),y_u(t))$\footnote{The dependence on the initial conditions will not be
indicated explicitly; they will be fixed at values to be discussed.}
denote the unique solution of
\eqref{eq:x_general}--\eqref{eq:y_general} on $[0,T]$ from fixed
initial conditions $(x_0,y_0)$ chosen independently of $u$.
Define the dose response (DR) and cumulative dose response (cDR) by
\begin{equation}
\label{eq:DR_cDR:def}
    \mathrm{DR}(u,T)\coloneqq y_u(T),
    \qquad
    \mathrm{cDR}(u,T)\coloneqq
    \int_0^T y_u(t)\,\mbox{d}t.
\end{equation}
The dose response represents the instantaneous output at time $T$,
whereas the cumulative dose response captures the total effect
accumulated over $[0,T]$.

In general, the map $u\mapsto\mathrm{DR}(u,T)$ need not be monotone
because the direct and indirect effects of $u$ on the output may
compete. However, a nonmonotone DR does not necessarily imply a
nonmonotone cDR: integration may preserve monotonicity even when the
instantaneous response does not. We therefore analyze the sensitivity
of $y_u(t)$ with respect to $u$ and derive an integral representation
of $\partial_u\mathrm{cDR}(u,T)$, from which sufficient conditions for
both monotonicity and nonmonotonicity are obtained.

\section{On the monotonicity of $\mathrm{cDR}(u,T)$}

In this section, we characterize the monotonicity of the map
$u\mapsto\mathrm{cDR}(u,T)$. These results apply to any continuously
differentiable output map $F$ in \eqref{eq:y_general}, independently
of \eqref{eq:incoherent_condition} and \eqref{eq:y_IFFM}.
Define the sensitivities:
\begin{equation*}
    p_u(t):=\partial_u x_u(t)\in\mathbb{R}^n
    \;\;\text{and}\;\;
    q_u(t):=\partial_u y_u(t)\in\mathbb{R}.
\end{equation*}
Differentiating \eqref{eq:y_general} with respect to $u$ gives
\begin{equation}\label{eq:q_general}
    \dot q_u(t)+a_u(t)q_u(t)=g_u(t),
    \qquad q_u(0)=0,
\end{equation}
where, evaluated along $(x_u(t),y_u(t),u)$,
\begin{equation*}
    a_u(t):=-\partial_y F(\cdot), \quad g_u(t):=\partial_x F(\cdot)p_u(t) +\partial_u F(\cdot).
\end{equation*}
Define further
\begin{equation}\label{eq:G_general}
    G_u(t):=\int_0^t g_u(s)\,\mbox{d}s,
\end{equation}
and
\begin{equation}\label{eq:lambda_general}
    \lambda_u(t):=\int_t^T \exp\!\left(-\int_t^s a_u(\tau)\,\mbox{d}\tau\right)\mbox{d}s,
    \quad t\in[0,T].
\end{equation}
By definition,
$\lambda_u(t)>0$ for $t<T$ and $\lambda_u(T)=0$, independently of
$u$ and of the sign of $a_u$.

For a fixed $T>0$, the map $u\mapsto\mathrm{cDR}(u,T)$ is monotonic if and only if $\partial_u\mathrm{cDR}(u,T)$ has a fixed sign with respect to $u$. The following theorem provides sufficient conditions for this property.

\begin{theorem}\label{thm:cDR_monotone}
    Consider \eqref{eq:x_general}--\eqref{eq:y_general}, with
    $\mathrm{cDR}(u,T)$, $g_u(t)$, $G_u(t)$, and $\lambda_u(t)$ as
    defined in \eqref{eq:DR_cDR:def}--\eqref{eq:lambda_general}.
    For $T>0$, if either
    \begin{itemize}
        \item[\emph{(i)}] $\lambda_u(t)$ and $g_u(t)$ each have a fixed sign
        on $[0,T]$, with these signs independent of $u$; or
        \item[\emph{(ii)}] $\dot{\lambda}_u(t)$ and $G_u(t)$ each have a fixed sign
        on $[0,T]$, with these signs independent of $u$,
    \end{itemize}
    then $\partial_u \mathrm{cDR}(u,T)$ has a fixed sign.
    In particular, the map $u\mapsto \mathrm{cDR}(u,T)$ is monotonic.
\end{theorem}

\begin{proof}
By the variation-of-parameters formula applied to
\eqref{eq:q_general},
\begin{equation}
\label{eq:q_explicit:general}
    q_u(t)
    =
    \int_0^t
    \exp\!\left(-\int_s^t a_u(\tau)\,\mbox{d}\tau\right)
    g_u(s)\,\mbox{d}s.
\end{equation}
Since $q_u(t)=\partial_u y_u(t)$, integrating over $[0,T]$
and applying Fubini's theorem, changing the order of integration over
$\{(s,t):0\leq s\leq t\leq T\}$, gives
\begin{align}
    \partial_u \mathrm{cDR}(u,T)
    &= \int_0^T \partial_u y_u(t)\,\mbox{d}t
     = \int_0^T q_u(t)\,\mbox{d}t \nonumber\\
    &= \int_0^T
       \int_s^T
       \exp\!\left(-\int_s^t a_u(\tau)\,\mbox{d}\tau\right)
       \mbox{d}t\, g_u(s)\,\mbox{d}s \nonumber\\
    &= \int_0^T \lambda_u(s)g_u(s)\,\mbox{d}s.
    \label{eq:int_by_part}
\end{align}

We now consider the two cases.

\begin{itemize}
    \item[(i)] Suppose that $\lambda_u(t)$ and $g_u(t)$ each have a
    fixed sign on $[0,T]$, with signs independent of $u$. Then
    $\lambda_u(t)g_u(t)$ has a fixed sign, also independent of $u$.
    It follows from \eqref{eq:int_by_part} that
    $\partial_u\mathrm{cDR}(u,T)$ has a fixed sign with respect to $u$.

    \item[(ii)] Suppose that $\dot\lambda_u(t)$ and $G_u(t)$ each have
    a fixed sign on $[0,T]$, with signs independent of $u$.
    Since $g_u(t)=\dot G_u(t)$, integration by parts in
    \eqref{eq:int_by_part} gives
    \begin{align}
        \partial_u\mathrm{cDR}(u,T)
        &=\int_0^T\lambda_u(t)\dot G_u(t)\,\mbox{d}t \nonumber\\
        &=
        \big[\lambda_u(t)G_u(t)\big]_0^T
        -\int_0^T\dot\lambda_u(t)G_u(t)\,\mbox{d}t \nonumber\\
        &=
        -\int_0^T\dot\lambda_u(t)G_u(t)\,\mbox{d}t,
        \label{eq:int_by_part_G}
    \end{align}
    where the boundary term vanishes because $G_u(0)=0$ and
    $\lambda_u(T)=0$. Hence, $-\dot\lambda_u(t)G_u(t)$ has a fixed
    sign, independent of $u$, and therefore
    $\partial_u\mathrm{cDR}(u,T)$ has a fixed sign with respect to $u$.
\end{itemize}

In either case, $\partial_u\mathrm{cDR}(u,T)$ has a fixed sign with
respect to $u$. Therefore, the map
$u\mapsto\mathrm{cDR}(u,T)$ is monotone.
\end{proof}

Theorem~\ref{thm:cDR_monotone} gives sufficient conditions for monotonicity of the map $u\mapsto \mathrm{cDR}(u,T)$, based on a sign structure that holds uniformly with respect to $u$. If this sign structure fails for some values of $u$, monotonicity is no longer guaranteed by that theorem. To establish nonmonotonicity, it is enough to show that $\partial_u \mathrm{cDR}(u,T)$ takes opposite signs at two input values. The following theorem formalizes this observation.

\begin{theorem}\label{thm:cDR_nonmonotone}
Consider \eqref{eq:x_general}--\eqref{eq:y_general}, with $\mathrm{cDR}(u,T)$, $g_u(t)$,
$G_u(t)$, and $\lambda_u(t)$ as defined in
\eqref{eq:DR_cDR:def}--\eqref{eq:lambda_general}.
For $T>0$, assume that there exist $u_-,u_+>0$ such that
\begin{itemize}
    \item[\emph{(i)}] $\lambda_{u_-}(t)g_{u_-}(t)\ge 0$ and
    $ \lambda_{u_-}(t)g_{u_-}(t)\not\equiv 0,$
    whereas
    $\lambda_{u_+}(t)g_{u_+}(t)\le 0$ and 
    $\lambda_{u_+}(t)g_{u_+}(t)\not\equiv 0$; or

    \item[\emph{(ii)}] $-\dot\lambda_{u_-}(t)G_{u_-}(t)\ge 0$ and $-\dot\lambda_{u_-}(t)\,G_{u_-}(t)\not\equiv 0,$ while
    $-\dot\lambda_{u_+}(t)G_{u_+}(t)\le 0$ and $-\dot\lambda_{u_+}(t)\,G_{u_+}(t)\not\equiv 0$,
\end{itemize}
for all $t\in[0,T]$. Then
\begin{equation}\label{eq:cDR_nonmonotone}
    \partial_u \mathrm{cDR}(u_-,T)> 0,
    \qquad \partial_u \mathrm{cDR}(u_+,T)< 0,
\end{equation}
and the map $u\mapsto \mathrm{cDR}(u,T)$ is not monotone on any interval containing both $u_-$ and $u_+$.
\end{theorem}

\begin{proof}
Under assumption~(i), \eqref{eq:int_by_part} shows that the integrand
$\lambda_u(t)g_u(t)$ is nonnegative for $u=u_-$ and nonpositive for
$u=u_+$, and is not identically zero in either case. Since the
integrands are continuous, their integrals are respectively strictly
positive and strictly negative. Hence,
\[
    \partial_u\mathrm{cDR}(u_-,T)>0,
    \qquad
    \partial_u\mathrm{cDR}(u_+,T)<0.
\]

Under assumption~(ii), \eqref{eq:int_by_part_G} gives the same
conclusion because $-\dot\lambda_u(t)G_u(t)$ is nonnegative for
$u=u_-$ and nonpositive for $u=u_+$, and is continuous and not
identically zero in either case. Thus, \eqref{eq:cDR_nonmonotone}
follows.

Since $\partial_u\mathrm{cDR}(u,T)$ takes opposite signs at $u_-$ and
$u_+$, the map $u\mapsto\mathrm{cDR}(u,T)$ cannot be monotone on any
interval containing both $u_-$ and $u_+$.
\end{proof}

Next, we apply Theorems~\ref{thm:cDR_monotone}
and~\ref{thm:cDR_nonmonotone} first to the scalar specializations of
the eight models in Table~\ref{tab:systems}, and then to their
generalizations with the intermediate dynamics in
\eqref{eq:x_general}. Four representative systems are analyzed in the
main text, while the remaining cases are provided in
Appendix~\ref{Apx:vector}.

\section{Applications to {\IFFM} models}

Theorems~\ref{thm:cDR_monotone}
and~\ref{thm:cDR_nonmonotone} reduce the monotonicity analysis to
sign properties of $g_u$, $G_u$, $\lambda_u$, and
$\dot\lambda_u$. For each {\IFFM} model, we proceed as
follows\footnote{Dependence on $t$ is omitted when no confusion arises,
and $0$ denotes a zero scalar, vector, or matrix of the appropriate
dimension.}:
\begin{enumerate}
    \item Compute $p_u=\partial_u x_u$ from
    $\dot p_u=Ap_u+b$.

    \item Differentiate the $y$--dynamics to identify
    $a_u$, $g_u$, and $G_u$.

    \item Analyze the signs of $g_u$, $G_u$, and
    $\dot\lambda_u$, and apply the appropriate theorem.
\end{enumerate}

We first consider the scalar specialization of
\eqref{eq:x_general} with $n=1$, $A=-1$, and $b=1$, giving $\dot x_u=-x_u+u.$

We likewise specialize the output models in
Table~\ref{tab:systems} by setting
$c=d=\beta=\gamma=1$ and $K=0$. This yields eight scalar
two-state systems. Applying the preceding steps gives
Table~\ref{tab:ag}; complete derivations are provided in
Appendix~\ref{Apx:scalar}. The following proposition collects the
models for which $g_u$ has a fixed sign.

\begin{table}[h!]
    \centering
    \caption{Summary of $a_u$, $g_u$, and $G_u$ for the scalar
    specializations of the {\IFFM} models in Table~\ref{tab:systems},
    with $\dot x_u=-x_u+u$, $c=d=\beta=\gamma=1$,
    $p_u=1-e^{-t}$, $x_u>0$, and $y_u\geq0$.}
    \begin{tabular}{|c|c|c|c|c|}
    \toprule
    No. & $\dot y_u$ & $a_u$ & $g_u$ & $G_u$ \\ 
    \midrule
    1 &
    $\frac{x_u}{u}-y_u$
    & $1$
    & $-\dfrac{x_0 e^{-t}}{u^2}$
    & $-\dfrac{x_0}{u^2}(1-e^{-t})$
    \\[0.8em]
    
    2 &
    $x_u-uy_u$
    & $u$
    & $p_u-y_u$
    & $\displaystyle \int_0^t (p_u-y_u)\,\mbox{d}s$
    \\[0.8em]
    
    3 &
    $\frac{u}{x_u}-y_u$
    & $1$
    & $\dfrac{x_0e^{-t}}{x_u^2}$
    & $\displaystyle \int_0^t \frac{x_0e^{-s}}{x_u^2}\,\mbox{d}s$
    \\[0.8em]
    
    4 &
    $\frac{1}{x_u}-\frac{1}{u}y_u$
    & $\frac{1}{u}$
    & $-\dfrac{p_u}{x_u^2}+\dfrac{y_u}{u^2}$
    & $\displaystyle \int_0^t
      \left(-\frac{p_u}{x_u^2}+\frac{y_u}{u^2}\right)\mbox{d}s$
    \\[0.8em]
    
    5 &
    $u-x_uy_u$
    & $x_u$
    & $1-y_up_u$
    & $\displaystyle t-\int_0^t y_up_u\,\mbox{d}s$
    \\[0.8em]
    
    6 &
    $1-\frac{x_u}{u}y_u$
    & $\frac{x_u}{u}$
    & $\dfrac{x_0e^{-t}}{u^2}y_u$
    & $\displaystyle \int_0^t
      \frac{x_0e^{-s}}{u^2}y_u\,\mbox{d}s$
    \\[0.8em]
    
    7 &
    $\frac{1}{u}-\frac{1}{x_u}y_u$
    & $\frac{1}{x_u}$
    & $\dfrac{y_up_u}{x_u^2}-\dfrac{1}{u^2}$
    & $\displaystyle \int_0^t
      \left(\frac{y_up_u}{x_u^2}-\frac{1}{u^2}\right)\mbox{d}s$
    \\[0.8em]
    
    8 &
    $1-\frac{u}{x_u}y_u$
    & $\frac{u}{x_u}$
    & $-\dfrac{x_0e^{-t}}{x_u^2}y_u$
    & $-\displaystyle \int_0^t
      \frac{x_0e^{-s}}{x_u^2}y_u\,\mbox{d}s$
    \\
    
    \bottomrule
    \end{tabular}
    \label{tab:ag}
\end{table}

\begin{proposition}
\label{prop:scalar_models}
Let $\dot x_u=-x_u+u$, where $u\in\mathbb{R}_+$,
$x_0\in\mathbb{R}_+$, and $y_0\in\mathbb{R}_{\geq0}$. For every
$T>0$, $\mathrm{cDR}(u,T)$ is nonincreasing in $u$ for models $1$
and $8$, and nondecreasing in $u$ for models $3$ and $6$ of
Table~\ref{tab:systems}.
\end{proposition}

\begin{proof}
Table~\ref{tab:ag} gives $g_u\leq0$ for models $1$ and $8$ and
$g_u\geq0$ for models $3$ and $6$. Since $\lambda_u(t)\geq0$ on
$[0,T]$, the result follows from
Theorem~\ref{thm:cDR_monotone}(i).
\end{proof}

Before presenting the detailed applications, we summarize in
Table~\ref{tab:general_summary} the DR and cDR behavior of the eight
models in Table~\ref{tab:systems}. Under the stated assumptions,
models $1$, $2$, $3$, $5$, $6$, and $8$ yield monotone cDR, whereas
models $4$ and $7$ allow nonmonotone cDR.

We next analyze models $5$, $3$, $2$, and $4$ of
Table~\ref{tab:systems} with the $n$-dimensional intermediate
dynamics in \eqref{eq:x_general}. These systems are denoted by
{\IFFM}1, {\IFFM}2, {\IFFM}3, and {\IFFM}4, respectively.
Models $3$ and $5$ were previously analyzed in~\cite{Ankit_EDS},
albeit only in the scalar case. Here, we extend their analysis to the
more general dynamics and provide streamlined proofs within the unified
sensitivity framework. The corresponding analyses of models $1$,
$6$, $7$, and $8$ are provided in Appendix~\ref{Apx:vector}.

\begin{table*}[h!]
    \centering
    \caption{Summary for the systems with intermediate dynamics
    $\dot x_u=Ax_u+bu$. The cDR conclusions hold under the initial-condition
    assumptions stated in the corresponding results. Let $z_u:=c^\top x_u$,
    $r_u:=c^\top p_u$, and $h(t):=c^\top e^{At}x_0$, so that $z_u=h+ur_u$.
    For models $3$, $4$, $7$, and $8$, the denominator $K+z_u$ is considered
    with $K>0$.}
    \renewcommand{\arraystretch}{1.5}
    \resizebox{\textwidth}{!}{%
    \begin{tabular}{c|c|c|c|c|c}
    \toprule
    Model & Output dynamics $\dot y_u$ & $a_u$ & $g_u$
    & DR & cDR\\
    \midrule
    1
    & $\displaystyle \frac{\gamma z_u}{\beta u}-dy_u$
    & $d$
    & $\displaystyle -\frac{\gamma h}{\beta u^2}$
    & Nonincreasing & Nonincreasing\\
    2
    & $\displaystyle \gamma z_u-d\beta u y_u$
    & $d\beta u$
    & $\displaystyle \gamma r_u-d\beta y_u$
    & Can be nonmonotone & Nonincreasing\\
    3
    & $\displaystyle \frac{\gamma\beta u}{K+z_u}-dy_u$
    & $d$
    & $\displaystyle \frac{\gamma\beta(K+h)}{(K+z_u)^2}$
    & Nondecreasing & Nondecreasing\\
    4
    & $\displaystyle \frac{\gamma}{K+z_u}-\frac{d}{\beta u}y_u$
    & $\displaystyle \frac{d}{\beta u}$
    & $\displaystyle -\frac{\gamma r_u}{(K+z_u)^2}
       +\frac{d y_u}{\beta u^2}$
    & Can be nonmonotone & Can be nonmonotone\\
    5
    & $\displaystyle \gamma\beta u-dz_u y_u$
    & $dz_u$
    & $\displaystyle \gamma\beta-d r_u y_u$
    & Can be nonmonotone & Nondecreasing\\
    6
    & $\displaystyle \gamma-\frac{d z_u}{\beta u}y_u$
    & $\displaystyle \frac{d z_u}{\beta u}$
    & $\displaystyle \frac{d h y_u}{\beta u^2}$
    & Nondecreasing & Nondecreasing\\
    7
    & $\displaystyle \frac{\gamma}{\beta u}
       -\frac{d}{K+z_u}y_u$
    & $\displaystyle \frac{d}{K+z_u}$
    & $\displaystyle \frac{d r_u y_u}{(K+z_u)^2}
       -\frac{\gamma}{\beta u^2}$
    & Can be nonmonotone & Can be nonmonotone\\
    8
    & $\displaystyle \gamma-\frac{d\beta u}{K+z_u}y_u$
    & $\displaystyle \frac{d\beta u}{K+z_u}$
    & $\displaystyle -\frac{d\beta(K+h)y_u}{(K+z_u)^2}$
    & Nonincreasing & Nonincreasing\\
    \bottomrule
    \end{tabular}%
    }
    \label{tab:general_summary}
\end{table*}

\subsection{{\IFFM}1: Monotonicity of $\mathrm{cDR}(u,T)$}
\label{subsec:{\IFFM}1}

We first consider the following {\IFFM}1 system:
\begin{equation}
\label{eq:IFFL1}
    \dot{x}_u=Ax_u+bu,\qquad
    \dot{y}_u=\beta u-dc^\top x_uy_u,
\end{equation}
where $b,c\in\mathbb{R}^n_+$, $\gamma = 1$, and $d,\beta\in\mathbb{R}_+$. Notice that all {\IFFM} models share the same $x$--subsystem. Thus, the sensitivity $p_u=\partial_u x_u$ is identical in all cases.
Recall $p_u=\partial_u x_u\in\mathbb{R}^n$. Differentiating with respect to $u$ yields 
$\dot p_u=A p_u+b$ with $p_u(0)=0$.
Hence, the solution is given by
\begin{equation}
\label{eq:p_explicit}
    p_u=\int_0^t e^{A(t-s)}b\,\mbox{d}s
    =A^{-1}\bigl(e^{At}-I\bigr)b.
\end{equation}
Since $A$ is Metzler, $e^{A(t-s)}\ge 0$ for all $t\ge s$, and since $b\in\mathbb{R}^n_+$, it follows that $p_u\ge 0$ for all $t\ge 0$.
Next, recall $q_u=\partial_u y_u\in\mathbb{R}$. Differentiating with respect to $u$ yields
\begin{equation*}
    \dot q_u=-dc^\top p_uy_u-dc^\top x_uq_u+\beta,
    \qquad q_u(0)=0.
\end{equation*}
Thus, \eqref{eq:q_general} holds with
$a_u=dc^\top x_u$ and $g_u=\beta-dc^\top p_uy_u$.

Regarding $\lambda_u$, since $c\in\mathbb{R}^n_+$ and $x_u\ge 0$ for all $t\ge 0$, we have $a_u=dc^\top x_u\ge 0$.
Therefore,
\begin{equation*}
    \lambda_u
    =
    \int_t^T
    \exp\!\left(-\int_t^s dc^\top x_u(\tau)\,\mbox{d}\tau\right)
    \mbox{d}s
    \ge 0,
    \qquad t\in[0,T].
\end{equation*}
Moreover, $\lambda_u(T)=0$, while $\lambda_u(t)>0$ for every $t<T$.
However, its detailed behavior depends on the trajectory $x_u$. 
In particular, the qualitative behavior of $x_u$ depends on the initial condition $x_0$ relative to the steady state $x_{\mathrm{ss}}(u)$.
Thus, although $\lambda_u\geq0$, the sign of $\dot\lambda_u$ cannot
yet be determined without further conditions on $x_u$.

Next, the following lemma\footnote{For tractability of the analysis, we assume that the initial
conditions satisfy
$x_0\in\{-A^{-1}bv\mid v\in[0,\infty)\}$ and
$y_0=y_{\mathrm{ss}}$. In the simulations, we use general initial
conditions not restricted to this ray.} establishes the required sign properties of $g_u$ and $G_u$.

\begin{lemma}
\label{lem:IFFL1}
Consider system \eqref{eq:IFFL1}. Fix $u\in\mathbb{R}_+$ and assume that
$x_0=-A^{-1}bv$ for some $v\in[0,\infty)$ and
$y_0=y_{\mathrm{ss}}$, where
$y_{\mathrm{ss}}=\beta/[dc^\top(-A^{-1}b)]$.
Define $g_u:=\beta-dc^\top p_uy_u$, where $p_u=\partial_u x_u$. Then, for all $t\in[0,T]$:
\begin{itemize}
    \item if $v\geq u$, then $g_u\geq0$;
    \item if $v<u$, then $G_u\geq0$.
\end{itemize}
\end{lemma}

When $v\ge u$, the first condition of
Theorem~\ref{thm:cDR_monotone} can be applied directly because
$\lambda_u\ge0$ and $g_u\ge0$. When $v<u$, the sign of $G_u$ is known
instead, and we therefore establish the required sign of
$\dot\lambda_u$.

\begin{proposition}
\label{prop:lambda_{\IFFM}1}
Under the assumptions of Lemma~\ref{lem:IFFL1}, suppose that $v<u$.
Let $\lambda_u$ be given by \eqref{eq:lambda_general} with
$a_u=dc^\top x_u$. Then $\dot\lambda_u\leq0$ for all
$t\in[0,T]$.
\end{proposition}

We now apply Theorem~\ref{thm:cDR_monotone} on the two input
intervals determined by $v$. For $u\in(0,v]$,
Lemma~\ref{lem:IFFL1} gives $g_u\geq0$ on $[0,T]$. Since
$\lambda_u\geq0$, Theorem~\ref{thm:cDR_monotone}(i) yields
$\partial_u\mathrm{cDR}(u,T)\geq0$.

For $u\in(v,\infty)$, Lemma~\ref{lem:IFFL1} gives
$G_u\geq0$ on $[0,T]$, while
Proposition~\ref{prop:lambda_{\IFFM}1} gives
$\dot\lambda_u\leq0$. Hence,
Theorem~\ref{thm:cDR_monotone}(ii) yields
\begin{equation*}
    \partial_u\mathrm{cDR}(u,T)
    =
    -\int_0^T\dot\lambda_u(t)G_u(t)\,\mbox{d}t
    \geq0.
\end{equation*}
Therefore, under the initial conditions of
Lemma~\ref{lem:IFFL1}, the map
$u\mapsto\mathrm{cDR}(u,T)$ is nondecreasing for system
\eqref{eq:IFFL1}.

\subsection{{\IFFM}2: Monotonicity of $\mathrm{cDR}(u,T)$}
\label{subsec:{\IFFM}2}

Next, we consider the following {\IFFM}2 system:
\begin{equation}
\label{eq:IFFL2}
    \dot{x}_u=Ax_u+bu,\qquad
    \dot y_u=\frac{\beta u}{K+c^\top x_u}-dy_u,
\end{equation}
where $b,c\in\mathbb{R}^n_+$, $\gamma = 1$, and $d,K,\beta\in\mathbb{R}_+$. Let $q_u=\partial_u y_u$. Differentiating with respect to $u$ gives
\begin{equation*}
    \dot q_u
    =
    \frac{\beta}{K+c^\top x_u}
    -
    \frac{\beta u c^\top p_u}{(K+c^\top x_u)^2}
    -dq_u,
    \qquad q_u(0)=0.
\end{equation*}
Thus, \eqref{eq:q_general} holds with $a_u=d$ and
\begin{equation*}
    g_u
    =
    \frac{\beta(K+c^\top x_u-uc^\top p_u)}
         {(K+c^\top x_u)^2}.
\end{equation*}
Since $x_u=e^{At}x_0+up_u$, we obtain
\begin{equation}
\label{eq:g_IFFL2}
    g_u
    =
    \frac{\beta\bigl(K+c^\top e^{At}x_0\bigr)}
         {(K+c^\top x_u)^2}
    >0.
\end{equation}
Moreover, $\lambda_u(t)>0$ for all $t<T$.
Hence, Theorem~\ref{thm:cDR_monotone}(i) gives
$\partial_u\mathrm{cDR}(u,T)>0$, and therefore
$u\mapsto\mathrm{cDR}(u,T)$ is strictly increasing for system
\eqref{eq:IFFL2}.

\subsection{{\IFFM}3: Monotonicity of $\mathrm{cDR}(u,T)$}
\label{subsec:{\IFFM}3}

Next, we consider the following {\IFFM}3 system:
\begin{equation}
\label{eq:IFFL3}
    \dot{x}_u=Ax_u+bu,\qquad
    \dot{y}_u=c^\top x_u-\beta duy_u,
\end{equation}
where $b,c\in\mathbb{R}^n_+$, $\gamma = 1$, and $d,\beta\in\mathbb{R}_+$. Let $q_u=\partial_u y_u$. Differentiating with respect to $u$ yields
\begin{equation*}
    \dot q_u=c^\top p_u-\beta dy_u-\beta duq_u,
    \qquad q_u(0)=0.
\end{equation*}
Thus, \eqref{eq:q_general} holds with
$a_u=\beta du$ and $g_u=c^\top p_u-\beta dy_u$.
In this case, $\lambda_u$ admits the explicit expression
\begin{equation}
\label{eq:lambda:IFFL3}
    \lambda_u
    =
    \int_t^T e^{-\beta du(\tau-t)}\,\mbox{d}\tau
    =
    \frac{1-e^{-\beta du(T-t)}}{\beta du}.
\end{equation}
Hence, $\lambda_u\geq0$ on $[0,T]$ and $\lambda_u(T)=0$.
Since its derivative is 
$\dot\lambda_u=-e^{-\beta du(T-t)}<0,$
it follows that $\lambda_u$ is strictly decreasing on $[0,T]$.

To analyze $g_u$, define
$h_u:=-g_u=\beta dy_u-c^\top p_u$.
Using $p_u(0)=0$ and $y_0=y_{\mathrm{ss}}$, where
$y_{\mathrm{ss}}=c^\top(-A^{-1}b)/(\beta d)$, we obtain
$h_u(0)=\beta dy_{\mathrm{ss}}=c^\top(-A^{-1}b)>0$.
Differentiating $h_u$ and using $x_u=e^{At}x_0+up_u$ give
\begin{equation}
\label{eq:g_filter}
    \begin{aligned}
        \dot h_u
        &=
        \beta d(c^\top x_u-\beta duy_u)
        -c^\top\dot p_u,\\
        \dot h_u+\beta duh_u
        &=
        \beta dc^\top e^{At}x_0-c^\top\dot p_u.
    \end{aligned}
\end{equation}
The following lemma establishes the required sign of the accumulated sensitivity.

\begin{lemma}
\label{lem:IFFL3}
    Consider system \eqref{eq:IFFL3}. Fix $u\in\mathbb{R}_+$ and assume that $x_0\geq0$ and
    $y_0=y_{\mathrm{ss}}$, where
    $y_{\mathrm{ss}}=c^\top(-A^{-1}b)/(\beta d)$.
    Define $g_u:=c^\top p_u-\beta dy_u$, where $p_u=\partial_u x_u$. Then
    $G_u(t)\leq0$ for all $t\in[0,T]$.
\end{lemma}

By Lemma~\ref{lem:IFFL3}, $G_u\leq0$; since $\dot\lambda_u<0$, Theorem~\ref{thm:cDR_monotone}(ii) gives
\begin{equation*}
    \partial_u\mathrm{cDR}(u,T)
    =
    -\int_0^T\dot\lambda_uG_u\,\mbox{d}t
    \leq0.
\end{equation*}
Hence, under the initial conditions of
Lemma~\ref{lem:IFFL3}, the map
$u\mapsto\mathrm{cDR}(u,T)$ is nonincreasing for system
\eqref{eq:IFFL3}.

\subsection{{\IFFM}4: Nonmonotonicity of $\mathrm{cDR}(u,T)$}
\label{subsec:{\IFFM}4}

Finally, we consider the following {\IFFM}4 system:
\begin{equation}\label{eq:IFFL4}
    \dot{x}_u = A x_u + b u, \qquad
    \dot{y}_u = \frac{1}{K + c^\top x_u} - \frac{dy_u}{\beta u},
\end{equation}
where $b,c\in\mathbb{R}^n_+$, $\gamma = 1$, and $d,K,\beta\in\mathbb{R}_+$.
Let $q_u=\partial_u y_u$. Differentiating with respect to $u$ yields
\begin{equation*}
    \dot q_u
    = -\frac{c^\top p_u}{(K + c^\top x_u)^2} + \frac{dy_u}{\beta u^2} - \frac{d}{\beta u} q_u,
    \quad q_u(0)=0.
\end{equation*}
Thus, $q_u$ satisfies \eqref{eq:q_general} with $a_u = d/(\beta u)$ and
\begin{equation}\label{eq:g_IFFL4}
    g_u= -\frac{c^\top p_u}{(K + c^\top x_u)^2}+\frac{dy_u}{\beta u^2}.
\end{equation}
In this case, $\lambda_u$ admits an explicit expression
\begin{equation*}
    \lambda_u
    =
    \int_t^T e^{-d(\tau-t)/(\beta u)}\,\mbox{d}\tau
    =
    \frac{\beta u}{d}\bigl(1-e^{-d(T-t)/(\beta u)}\bigr).
\end{equation*}
Hence $\lambda_u\ge 0$ on $[0,T]$ and $\lambda_u(T)=0$. Moreover,
$\dot{\lambda}_u=-e^{-d(T-t)/(\beta u)}<0$ for all $t\in[0,T]$.
Thus, $\lambda_u$ is strictly decreasing on $[0,T]$.
The following lemma establishes the nonmonotonicity of $\mathrm{cDR}(u,T)$.

\begin{lemma}
\label{lem:IFFL4}
    Consider system \eqref{eq:IFFL4} and fix $T>0$. Assume that
    $x_0=0$ and $y_0=\beta/(dK)$. Define $g_u$ as in \eqref{eq:g_IFFL4}.
    Then there exist $u_-,u_+>0$ such that
    \begin{itemize}
        \item $\partial_u\mathrm{cDR}(u,T)>0$ for all $0<u<u_-$;
        \item $\partial_u\mathrm{cDR}(u,T)<0$ for all $u>u_+$.
    \end{itemize}
    Consequently, $u\mapsto\mathrm{cDR}(u,T)$ is not monotone.
\end{lemma}

The preceding results establish monotonicity for
{\IFFM}1--{\IFFM}3 and nonmonotonicity for {\IFFM}4 under the stated
initial-condition assumptions. The second nonmonotone mechanism,
corresponding to model $7$ in Table~\ref{tab:systems}, is analyzed in
Appendix~\ref{Apx:vector}. The next section illustrates these results
numerically.

\section{Numerical Results}
\label{sec:NumResults}

In this section, we illustrate the theoretical results on monotonicity
of the cumulative dose response (cDR) using four {\IFFM} systems from
three different initial conditions $(x_0,y_0)$. The dose response (DR)
and cumulative dose response (cDR) are defined in
\eqref{eq:DR_cDR:def}. For each system, we fix the time horizon
$T=1.5$ and sweep the input over $u\in[10^{-3},10^3]$ using
logarithmically spaced samples. All {\IFFM} systems share the
$x$--subsystem $\dot x_u=Ax_u+bu$, with
\begin{equation*}
    \begin{aligned}
    A &= \begin{bmatrix}
    -3.0 & 0.8 & 0.0 & 0.0 & 0.0\\
    0.4 & -2.6 & 0.7 & 0.0 & 0.0\\
    0.0 & 0.5 & -2.8 & 0.6 & 0.0\\
    0.0 & 0.0 & 0.4 & -2.3 & 0.7\\
    0.0 & 0.0 & 0.0 & 0.3 & -1.7
    \end{bmatrix}, \qquad
    b &= \begin{bmatrix}1.0 \\ 0.8 \\ 0.9 \\ 0.7 \\ 0.6\end{bmatrix},\qquad
    c &= \begin{bmatrix}0.9 \\ 0.7 \\ 0.8 \\ 0.6 \\ 0.5\end{bmatrix},
    \end{aligned}
\end{equation*}
The matrix $A$ is Metzler and Hurwitz, and $b,c>0$. We use
$d=1.2$, $\beta=1.5$, $\gamma=1$, and $K=0.8$ whenever the
corresponding parameter appears. To examine behavior beyond the
structured initial conditions used in some proofs, we consider
\[
x_0^{(1)}=\mathbf{0}_5,\qquad
x_0^{(2)}=[0.5,0.6,0.7,0.8,0.9]^\top,\qquad
x_0^{(3)}=[2.0,2.1,2.3,2.4,2.5]^\top.
\]
For {\IFFM}1 and {\IFFM}3, we set $y_0=y_{\mathrm{ss}}$; for
{\IFFM}2 and {\IFFM}4, we use
\[
    y_0=\frac{\beta}{d(K+c^\top x_0)}.
\]

For {\IFFM}1 and {\IFFM}3,
Figures~\ref{fig:iffl1_DR} and~\ref{fig:iffl3_DR} show that the dose
response $\mathrm{DR}(u,T)$ is \emph{nonmonotone} with respect to $u$
for some initial conditions. However, the corresponding cumulative
responses in Figures~\ref{fig:iffl1_cDR} and~\ref{fig:iffl3_cDR} are
\emph{monotone}. These observations are consistent with the theoretical
results. For {\IFFM}1, the monotonicity of cDR follows from the sign of
either $g_u$ or the pair $(\dot\lambda_u,G_u)$, depending on the
initial condition relative to the input. For {\IFFM}3, although the
instantaneous sensitivity may change sign, the properties
$\dot\lambda_u<0$ and $G_u\leq0$ ensure that
$\partial_u\mathrm{cDR}(u,T)\leq0$.

For {\IFFM}2, Figures~\ref{fig:iffl2_DR}
and~\ref{fig:iffl2_cDR} show that both DR and cDR are strictly
increasing. This reflects the stronger sign structure of this system:
$g_u>0$, which implies $\partial_u y_u(t)>0$ for $t>0$ and therefore
strict monotonicity at both the instantaneous and cumulative levels.

In contrast, Figures~\ref{fig:iffl4_DR}
and~\ref{fig:iffl4_cDR} show that both DR and cDR are
\emph{nonmonotone} for {\IFFM}4. This is consistent with
Theorem~\ref{thm:cDR_nonmonotone} and
Lemma~\ref{lem:IFFL4}, which show that
$\partial_u\mathrm{cDR}(u,T)$ is positive for sufficiently small
inputs and negative for sufficiently large inputs. Thus, the sign
conditions ensuring monotonicity for {\IFFM}1--{\IFFM}3 do not hold
uniformly for {\IFFM}4. Moreover, because $K>0$ for {\IFFM}2 and
{\IFFM}4, their steady-state output
\[
    y_{\mathrm{ss}}(u)
    =
    \frac{\beta u}
    {d\bigl(K+c^\top(-A^{-1}b)u\bigr)}
\]
depends on $u$, and perfect adaptation is therefore lost.

The behaviors observed in the simulations are summarized in
Table~\ref{tab:summary}. These numerical experiments corroborate the
theoretical results and indicate that the monotonicity properties of
{\IFFM}1--{\IFFM}3 persist beyond the structured initial conditions
used in the analysis. In contrast, {\IFFM}4 can exhibit nonmonotone
cumulative responses.

\begin{figure*}[t!]
    \centering
    \subfloat[\label{fig:iffl1_DR}\scriptsize {\IFFM}1]{\includegraphics[width=0.25\linewidth]{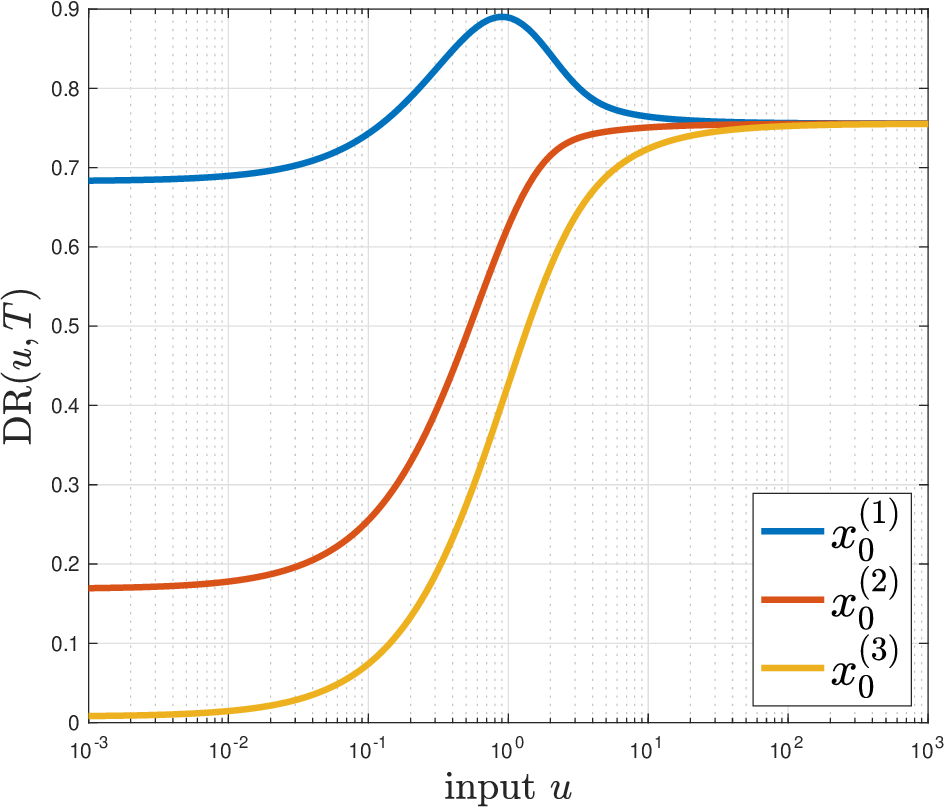}}
    \subfloat[\label{fig:iffl2_DR}\scriptsize {\IFFM}2]{\includegraphics[width=0.25\linewidth]{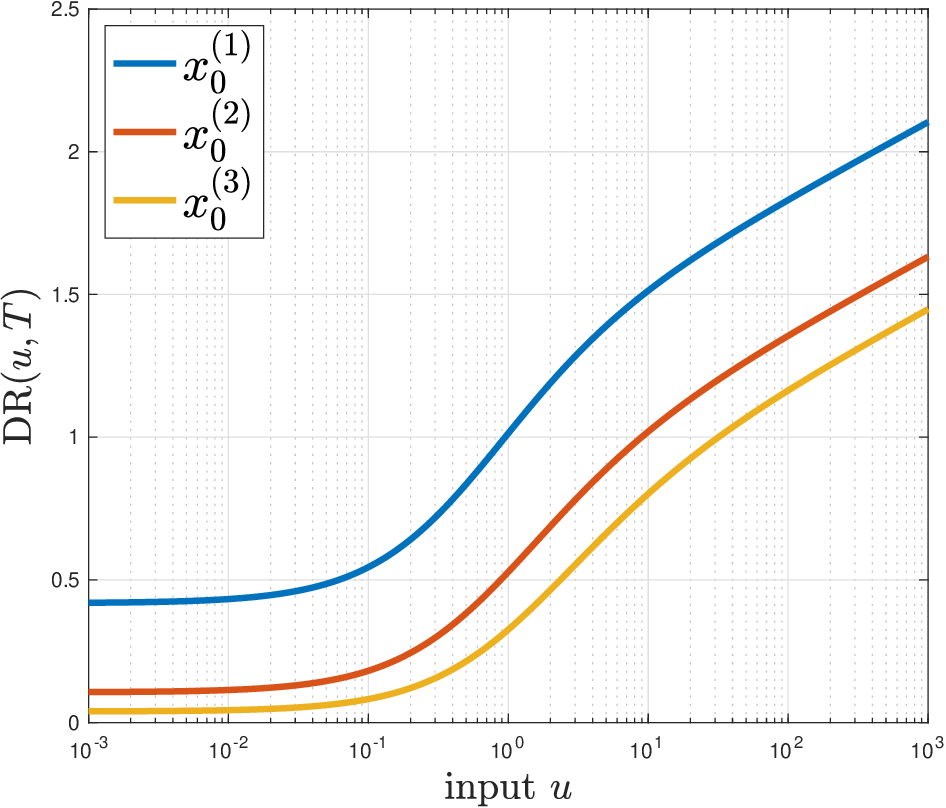}}
    \subfloat[\label{fig:iffl3_DR}\scriptsize {\IFFM}3]{\includegraphics[width=0.25\linewidth]{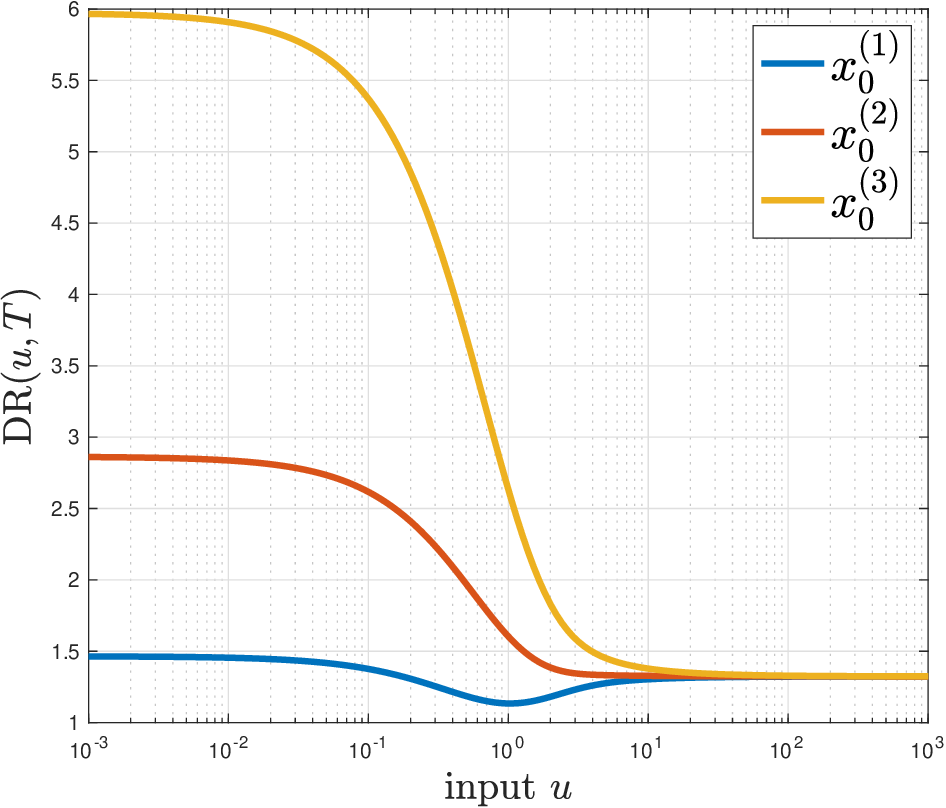}}
    \subfloat[\label{fig:iffl4_DR}\scriptsize {\IFFM}4]{\includegraphics[width=0.25\linewidth]{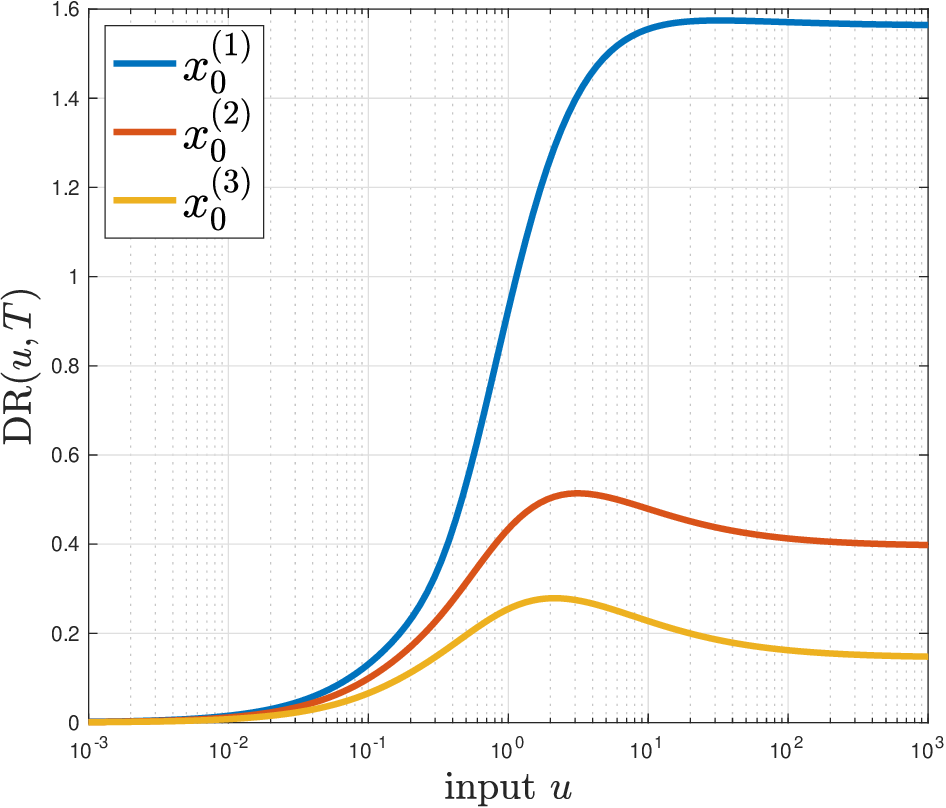}}
    \caption{Dose response $\mathrm{DR}(u,T)$ for the four {\IFFM} systems under three different initial conditions $(x_0,y_0)$. The input $u$ is varied over $u \in [10^{-3},10^{3}]$. 
    {\IFFM}1 and {\IFFM}3 exhibit nonmonotone DR, {\IFFM}2 shows monotone DR, and
    {\IFFM}4 exhibits nonmonotone DR.}
    \label{fig:DR}
\end{figure*}
\begin{figure*}[t!]
    \centering
    \subfloat[\label{fig:iffl1_cDR}\scriptsize {\IFFM}1]{\includegraphics[width=0.25\linewidth]{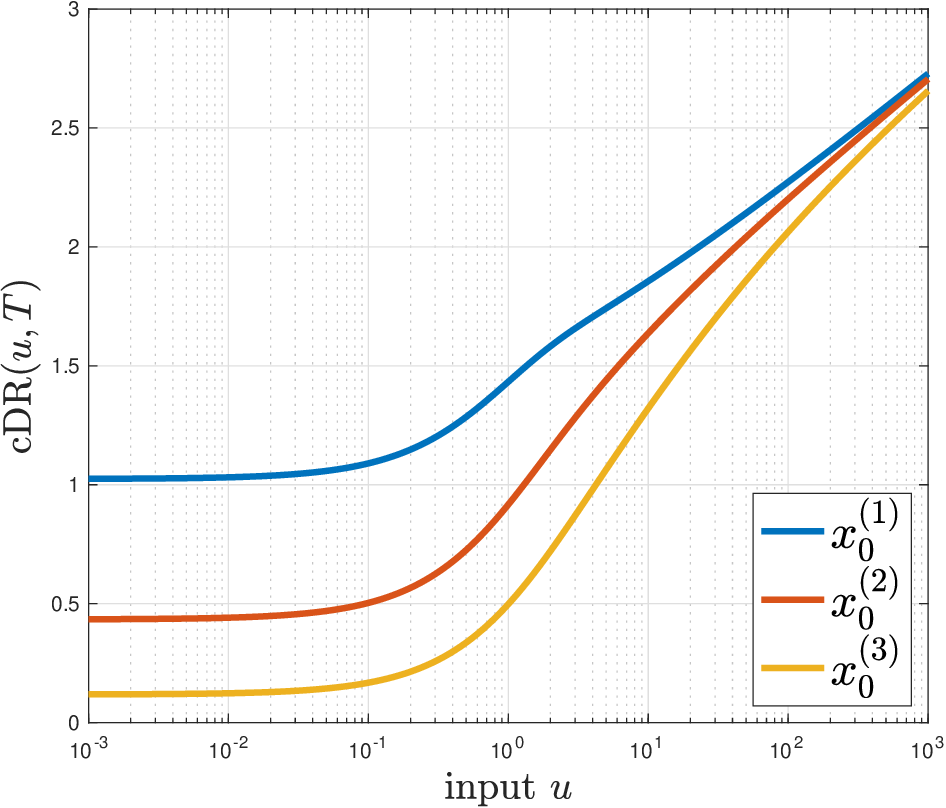}}
    \subfloat[\label{fig:iffl2_cDR}\scriptsize {\IFFM}2]{\includegraphics[width=0.2465\linewidth]{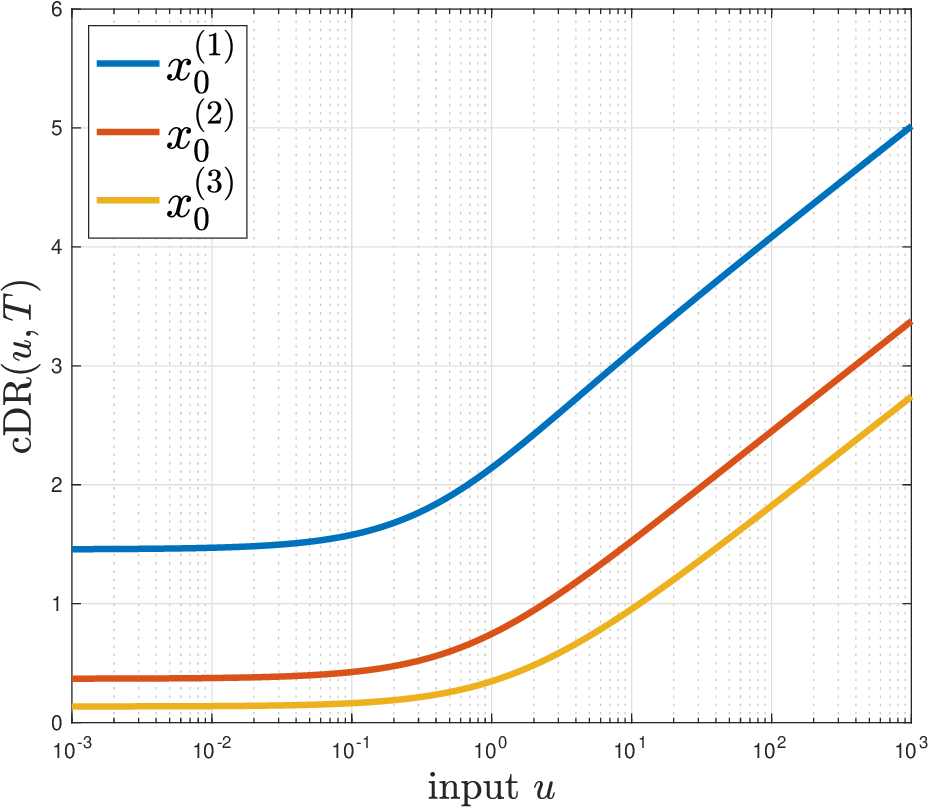}}
    \subfloat[\label{fig:iffl3_cDR}\scriptsize {\IFFM}3]{\includegraphics[width=0.2465\linewidth]{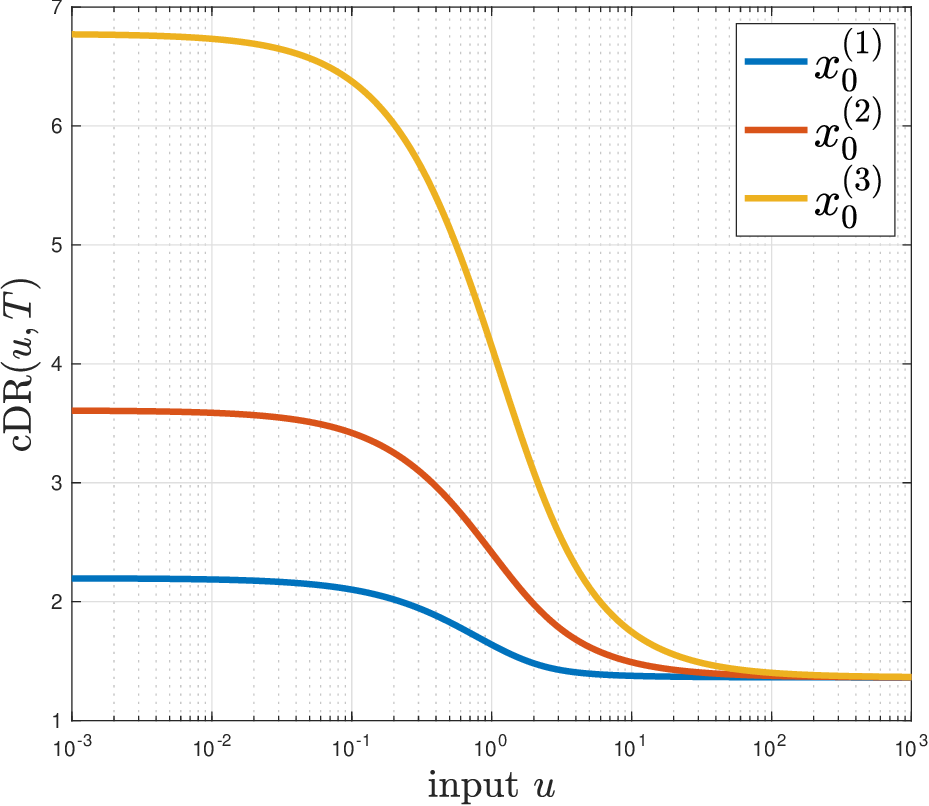}}
    \subfloat[\label{fig:iffl4_cDR}\scriptsize {\IFFM}4]{\includegraphics[width=0.25\linewidth]{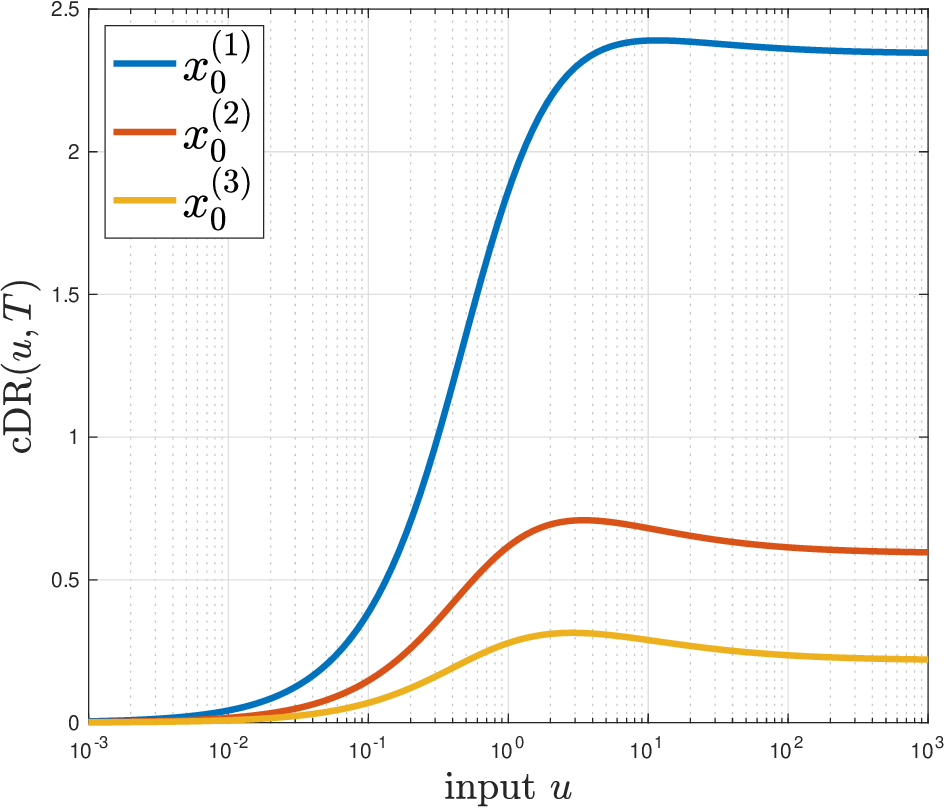}}
    \caption{Cumulative dose response $\mathrm{cDR}(u,T)$ for the four {\IFFM} systems under the same conditions as Figure~\ref{fig:DR}. Despite nonmonotone DR, {\IFFM}1 and {\IFFM}3 exhibit monotone cDR, in agreement with the theoretical results. {\IFFM}2 remains monotone, whereas {\IFFM}4 exhibits
    nonmonotone cDR, consistent with the absence of a fixed sign for
$G_u$.}
    \label{fig:cDR}
\end{figure*}
\begin{table}[t!]
    \centering
    \caption{Summary of monotonicity and sign properties}
    \begin{tabular}{c|c|c|c|c}
    \hline
    System & DR & cDR & $a_u$ & $G_u$\\
    \hline
    {\IFFM}1 & Nonmonotone & Monotone & $\ge 0$ & $\ge 0$\\
    {\IFFM}2 & Monotone & Monotone & $\ge 0$ & $\ge 0$\\
    {\IFFM}3 & Nonmonotone & Monotone & $\ge 0$ & $\le 0$\\
    {\IFFM}4 & Nonmonotone & Nonmonotone & $\ge 0$ & No fixed sign\\
    \hline
    \end{tabular}
    \label{tab:summary}
\end{table}

\section{Conclusion}
\label{sec:Conclusion}

We studied cDR monotonicity for a class of incoherent feedforward
{\motifs} ({\IFFM}) systems. We derived an integral representation of
$\partial_u\mathrm{cDR}(u,T)$ and reduced the analysis to sign
conditions on a sensitivity term and its associated kernel. This
yielded sufficient conditions for monotonicity
(Theorem~\ref{thm:cDR_monotone}) and nonmonotonicity
(Theorem~\ref{thm:cDR_nonmonotone}).

We identified eight canonical {\IFFM} models and characterized their
cDR monotonicity under the stated conditions. Models $1$, $2$, $3$,
$5$, $6$, and $8$ yield monotone cDR, whereas models $4$ and $7$
allow nonmonotone cDR. Four representative systems were analyzed in
the main text, with the remaining cases treated in the appendices.

Numerical simulations indicate that the monotonicity properties of
{\IFFM}1--{\IFFM}3 persist beyond the structured initial conditions
used in the analysis, whereas {\IFFM}4 can exhibit nonmonotone
cumulative responses. Future work includes broader nonlinear systems,
relaxed positivity assumptions, and minimal structural conditions for
cDR monotonicity.

\bibliographystyle{ieeetr}  
\bibliography{references}

\appendix

\section{Proofs}
    \subsection*{Proof of Lemma~\ref{lem:IFFL1}}

    First, we analyze $x_u$ and $y_u$ with respect to
    $x_{\mathrm{ss}}(u)$ and $y_{\mathrm{ss}}$, respectively, under the two cases
    $v\geq u$ and $v<u$. Since
    $x_0=-A^{-1}bv=x_{\mathrm{ss}}(v)$, the solution of the
    $x$--subsystem can be written as
    \begin{equation*}
        x_u(t)=x_0+p_u(t)(u-v).
    \end{equation*}
    From \eqref{eq:p_explicit}, $\dot p_u(t)=Ap_u(t)+b=e^{At}b\geq0,$ for all $t\geq0.$
    Therefore,
    \begin{equation*}
        \dot x_u(t)=e^{At}b(u-v),
    \end{equation*}
    so $x_u$ is nonincreasing when $v\geq u$ and nondecreasing when $v<u$.
    Since $x_u(t)\to x_{\mathrm{ss}}(u)$ as $t\to\infty$, it follows that
    \begin{equation*}
        x_{\mathrm{ss}}(u)\leq x_u(t)\leq x_0,
        \qquad v\geq u,
    \end{equation*}
    and
    \begin{equation*}
        x_0\leq x_u(t)\leq x_{\mathrm{ss}}(u),
        \qquad v<u.
    \end{equation*}
    
    Next, consider the $y$--subsystem in \eqref{eq:IFFL1}. Since
    $\beta u=dy_{\mathrm{ss}}c^\top x_{\mathrm{ss}}(u),$
    evaluating its vector field at $y_u=y_{\mathrm{ss}}$ gives
    \begin{equation*}
        \left.\dot y_u\right|_{y_u=y_{\mathrm{ss}}}
        =
        dy_{\mathrm{ss}}c^\top
        \bigl(x_{\mathrm{ss}}(u)-x_u\bigr).
    \end{equation*}
    If $v\geq u$, then $x_u\geq x_{\mathrm{ss}}(u)$, and hence
    \begin{equation*}
        \left.\dot y_u\right|_{y_u=y_{\mathrm{ss}}}\leq0.
    \end{equation*}
    Since $y_0=y_{\mathrm{ss}}$, the comparison principle yields
    $y_u(t)\leq y_{\mathrm{ss}}$ for all $t\geq0$.
    Similarly, if $v<u$, then $x_u\leq x_{\mathrm{ss}}(u)$, so
    \begin{equation*}
        \left.\dot y_u\right|_{y_u=y_{\mathrm{ss}}}\geq0,
    \end{equation*}
    and therefore $y_u(t)\geq y_{\mathrm{ss}}$ for all $t\geq0$.
    
    We now prove the two claims.
    
    \begin{itemize}
        \item \textbf{Case $v\geq u$.}
        Since
        \begin{equation*}
            p_u(t)
            =
            \int_0^t e^{A(t-s)}b\,\mbox{d}s
            \leq
            \int_0^\infty e^{As}b\,\mbox{d}s
            =
            -A^{-1}b,
        \end{equation*}
        and $c\in\mathbb{R}_+^n$, we have
        \begin{equation*}
            c^\top p_u(t)\leq c^\top(-A^{-1}b).
        \end{equation*}
        Together with $y_u(t)\leq y_{\mathrm{ss}}$, this gives
        \begin{equation*}
            dc^\top p_u(t)y_u(t)
            \leq
            dc^\top(-A^{-1}b)y_{\mathrm{ss}}
            =
            \beta.
        \end{equation*}
        Hence,
        \begin{equation*}
            g_u(t)
            =
            \beta-dc^\top p_u(t)y_u(t)
            \geq0,
            \qquad t\geq0.
        \end{equation*}
    
        \item \textbf{Case $v<u$.}
        Using
        \begin{equation*}
            x_u(t)=e^{At}x_0+up_u(t),
        \end{equation*}
        we obtain
        \begin{equation*}
            c^\top x_u(t)
            =
            c^\top e^{At}x_0+uc^\top p_u(t).
        \end{equation*}
        Substituting this identity into the $y$--subsystem and
        rearranging gives
        \begin{equation*}
            g_u(t)
            =
            \beta-dc^\top p_u(t)y_u(t)
            =
            \frac{1}{u}\dot y_u(t)
            +
            \frac{d}{u}
            \bigl(c^\top e^{At}x_0\bigr)y_u(t).
        \end{equation*}
        Integrating over $[0,t]$ and using
        $y_0=y_{\mathrm{ss}}$ yields
        \begin{equation*}
            G_u(t)
            =
            \frac{1}{u}\bigl(y_u(t)-y_{\mathrm{ss}}\bigr)
            +
            \frac{d}{u}
            \int_0^t
            \bigl(c^\top e^{As}x_0\bigr)y_u(s)\,\mbox{d}s.
        \end{equation*}
        Since $x_0\geq0$, $e^{As}\geq0$, $c\in\mathbb{R}_+^n$,
        and $y_u(s)\geq0$, the integral term is nonnegative.
        Moreover, $y_u(t)\geq y_{\mathrm{ss}}$. Therefore,
        \begin{equation*}
            G_u(t)\geq0,
            \qquad t\in[0,T].
        \end{equation*}
    \end{itemize}
    This completes the proof.

    \subsection*{Proof of Proposition~\ref{prop:lambda_{\IFFM}1}}

    Differentiating $\lambda_u$ gives
    \begin{equation*}
        \dot\lambda_u(t)
        =
        -1+dc^\top x_u(t)\lambda_u(t).
    \end{equation*}
    Since $v<u$,
    \begin{equation*}
        \dot x_u(t)=e^{At}b(u-v)\geq0,
    \end{equation*}
    so $x_u$ is nondecreasing. Hence, for every
    $s\in[t,T]$,
    \begin{equation*}
        c^\top x_u(\tau)\geq c^\top x_u(t),
        \qquad \tau\in[t,s],
    \end{equation*}
    and therefore
    \begin{equation*}
        \exp\!\left(
            -\int_t^s dc^\top x_u(\tau)\,\mbox{d}\tau
        \right)
        \leq
        \exp\!\left(-(s-t)dc^\top x_u(t)\right).
    \end{equation*}
    Integrating over $[t,T]$ and multiplying by
    $dc^\top x_u(t)\geq0$ give
    \begin{equation*}
        \begin{aligned}
            dc^\top x_u(t)\lambda_u(t)
            &\leq
            dc^\top x_u(t)
            \int_t^T
            \exp\!\left(-(s-t)dc^\top x_u(t)\right)
            \mbox{d}s\\
            &=
            1-e^{-(T-t)dc^\top x_u(t)}
            \leq1.
        \end{aligned}
    \end{equation*}
    Hence,
    \begin{equation*}
        \dot\lambda_u(t)
        =
        -1+dc^\top x_u(t)\lambda_u(t)
        \leq0,
    \end{equation*}
    as claimed.
    
    \subsection*{Proof of Lemma~\ref{lem:IFFL3}}
    
    Recall that
    \[
        h_u(t)=\beta d y_u(t)-c^\top p_u(t),
    \]
    and, from \eqref{eq:g_filter},
    \[
        \dot h_u(t)+\beta du\,h_u(t)
        =
        \beta d c^\top e^{At}x_0-c^\top\dot p_u(t).
    \]
    Define
    \[
        f_u(t):=
        \beta d c^\top e^{At}x_0-c^\top\dot p_u(t).
    \]
    From \eqref{eq:p_explicit}, $\dot p_u(t)=e^{At}b,$ and hence
    \[
        f_u(t)=c^\top e^{At}(\beta d x_0-b).
    \]
    
    By variation of parameters,
    \begin{equation*}
        h_u(t)
        =
        e^{-\beta dut}h_u(0)
        +
        \int_0^t
        e^{-\beta du(t-s)}f_u(s)\,\mbox{d}s.
    \end{equation*}
    Integrating over $[0,T]$ gives
    \begin{align*}
        \int_0^T h_u(t)\,\mbox{d}t
        &=
        h_u(0)\int_0^T e^{-\beta dut}\,\mbox{d}t 
        +
        \int_0^T\int_0^t
        e^{-\beta du(t-s)}f_u(s)\,\mbox{d}s\,\mbox{d}t.
    \end{align*}
    Applying Fubini's theorem on
    $\{(s,t):0\leq s\leq t\leq T\}$ yields
    \begin{align}
        \int_0^T h_u(t)\,\mbox{d}t
        &=
        \lambda_u(0)h_u(0)
        +
        \int_0^T
        \left(
            \int_s^T e^{-\beta du(t-s)}\,\mbox{d}t
        \right)
        f_u(s)\,\mbox{d}s \nonumber\\
        &=
        \lambda_u(0)h_u(0)
        +
        \int_0^T\lambda_u(s)f_u(s)\,\mbox{d}s.
        \label{eq:int_g}
    \end{align}
    
    Define
    \begin{equation*}
        H(T):=
        \int_0^T\lambda_u(s)e^{As}\,\mbox{d}s.
    \end{equation*}
    Since $e^{As}\geq0$ and $\lambda_u(s)\geq0$ on $[0,T]$,
    we have $H(T)\geq0$. Substituting
    $f_u(s)=c^\top e^{As}(\beta d x_0-b)$ into
    \eqref{eq:int_g} gives
    \begin{equation}
    \label{eq:intg_split_K}
        \begin{aligned}
            \int_0^T h_u(t)\,\mbox{d}t
            &=
            \lambda_u(0)h_u(0)
            +
            c^\top H(T)(\beta d x_0-b)\\
            &=
            \beta d c^\top H(T)x_0
            +
            \lambda_u(0)h_u(0)
            -
            c^\top H(T)b.
        \end{aligned}
    \end{equation}
    Since $H(T)\geq0$ and $x_0\geq0$, the first term is nonnegative.
    It therefore remains to show
    \[
        \lambda_u(0)h_u(0)-c^\top H(T)b\geq0.
    \]
    
    Since $\lambda_u$ is decreasing,
    \[
        \lambda_u(s)\leq\lambda_u(0),
        \qquad s\in[0,T].
    \]
    Therefore,
    \begin{align*}
        H(T)b
        =
        \int_0^T\lambda_u(s)e^{As}b\,\mbox{d}s
        \leq
        \lambda_u(0)\int_0^T e^{As}b\,\mbox{d}s
        \leq
        \lambda_u(0)\int_0^\infty e^{As}b\,\mbox{d}s
        =
        -\lambda_u(0)A^{-1}b.
    \end{align*}
    Multiplying by $c^\top$ gives
    \[
        c^\top H(T)b
        \leq
        \lambda_u(0)c^\top(-A^{-1}b).
    \]
    Moreover, since $p_u(0)=0$ and
    $y_0=y_{\mathrm{ss}}=c^\top(-A^{-1}b)/(\beta d)$,
    $h_u(0)=\beta d y_{\mathrm{ss}}=c^\top(-A^{-1}b).$
    Hence,
    \[
        c^\top H(T)b
        \leq
        \lambda_u(0)h_u(0).
    \]
    Using this in \eqref{eq:intg_split_K} yields
    \[
        \int_0^T h_u(t)\,\mbox{d}t\geq0.
    \]
    Since $g_u=-h_u$,
    \[
        G_u(T)
        =
        \int_0^T g_u(t)\,\mbox{d}t
        =
        -\int_0^T h_u(t)\,\mbox{d}t
        \leq0.
    \]
    The same argument applies with any $\tau\in[0,T]$ in place of $T$.
    Therefore,
    \[
        G_u(\tau)\leq0,
        \qquad \tau\in[0,T].
    \]
    This completes the proof.
    
    \subsection*{Proof of Lemma~\ref{lem:IFFL4}}
    
    Recall that $\dot p_u=Ap_u+b,$ with $p_u(0)=0.$
    Since the $x$--subsystem is linear and $x_0=0$, we have $x_u(t)=up_u(t).$
    Define
    \[
        r(t):=c^\top p_u(t)\geq0.
    \]
    Then $c^\top x_u(t)=ur(t),$ with $r(0)=0,$ and $\dot r(0)=c^\top b>0.$
    
    Recall also that
    \[
        -\dot\lambda_u(t)
        =
        e^{-d(T-t)/(\beta u)}
        >0,
        \qquad t\in[0,T].
    \]
    Using \eqref{eq:IFFL4}, we rewrite \eqref{eq:g_IFFL4} as
    \begin{equation*}
        g_u
        =
        -\frac{\dot y_u}{u}
        +
        \frac{K}{u(K+ur)^2}.
    \end{equation*}
    Indeed,
    \[
        -\frac{\dot y_u}{u}
        =
        -\frac{1}{u(K+ur)}
        +
        \frac{dy_u}{\beta u^2},
    \]
    and therefore
    \[
        -\frac{\dot y_u}{u}
        +
        \frac{K}{u(K+ur)^2}
        =
        -\frac{r}{(K+ur)^2}
        +
        \frac{dy_u}{\beta u^2}
        =
        g_u.
    \]
    Integrating from $0$ to $t$ gives
    \begin{equation}
    \label{eq:G_rewrite}
        G_u(t)
        =
        \frac{y_0-y_u(t)}{u}
        +
        \frac{K}{u}
        \int_0^t
        \frac{\mbox{d}s}{(K+ur(s))^2}.
    \end{equation}
    Hence, from \eqref{eq:int_by_part_G},
    \begin{equation}
    \label{eq:cDR_IFFL4_G}
        \partial_u\mathrm{cDR}(u,T)
        =
        \int_0^T
        e^{-d(T-t)/(\beta u)}
        G_u(t)\,\mbox{d}t.
    \end{equation}
    
    \medskip
    \noindent
    \textbf{Case 1: sufficiently small $u$.}
    
    If $0<u<dKy_0/\beta,$
    then, for every $t\in[0,T]$,
    \begin{equation*}
        \dot y_u\big|_{y_u=y_0}
        =
        \frac{1}{K+ur(t)}
        -
        \frac{dy_0}{\beta u}
        \leq
        \frac{1}{K}
        -
        \frac{dy_0}{\beta u}
        <0.
    \end{equation*}
    Since $y_u(0)=y_0$, the comparison principle yields
    \[
        y_u(t)\leq y_0,
        \qquad t\in[0,T].
    \]
    It follows from \eqref{eq:G_rewrite} that
    \[
        G_u(t)\geq0,
        \qquad t\in[0,T].
    \]
    Moreover, $G_u\not\equiv0$, since for every $t>0$,
    \[
        \frac{K}{u}
        \int_0^t
        \frac{\mbox{d}s}{(K+ur(s))^2}
        >0.
    \]
    Therefore, using \eqref{eq:cDR_IFFL4_G},
    \[
        \partial_u\mathrm{cDR}(u,T)>0
    \]
    for all sufficiently small $u>0$.
    
    \medskip
    \noindent
    \textbf{Case 2: sufficiently large $u$.}
    
    Let
    \[
        a_u:=\frac{d}{\beta u}.
    \]
    The $y$--dynamics can be written as
    \[
        \dot y_u+a_uy_u
        =
        \frac{1}{K+ur(t)}.
    \]
    By variation of parameters,
    \begin{equation*}
        y_u(t)
        =
        e^{-a_ut}y_0
        +
        \int_0^t
        \frac{e^{-a_u(t-s)}}{K+ur(s)}
        \,\mbox{d}s.
    \end{equation*}
    Substituting this expression into \eqref{eq:G_rewrite} gives
    \begin{align*}
        G_u(t)
        =
        \frac{y_0(1-e^{-a_ut})}{u}
        -
        \frac{1}{u}
        \int_0^t
        \frac{e^{-a_u(t-s)}}{K+ur(s)}
        \,\mbox{d}s 
        +
        \frac{K}{u}
        \int_0^t
        \frac{\mbox{d}s}{(K+ur(s))^2}.
    \end{align*}
    Using \eqref{eq:cDR_IFFL4_G} and exchanging the order of integration in
    the second term, we obtain
    \begin{align}
        \partial_u\mathrm{cDR}(u,T)
        &=
        \frac{y_0}{u}
        \int_0^T
        e^{-a_u(T-t)}
        \bigl(1-e^{-a_ut}\bigr)
        \,\mbox{d}t
        -
        \frac{1}{u}
        \int_0^T
        \frac{(T-s)e^{-a_u(T-s)}}{K+ur(s)}
        \,\mbox{d}s
        \nonumber\\
        &\quad
        +
        \frac{K}{u}
        \int_0^T
        e^{-a_u(T-t)}
        \int_0^t
        \frac{\mbox{d}s}{(K+ur(s))^2}
        \,\mbox{d}t.
    \label{eq:cDR_large_u}
    \end{align}
    
    We estimate the three terms in \eqref{eq:cDR_large_u}. Since
    $r(0)=0$ and $\dot r(0)=c^\top b>0,$
    we have
    \[
        \lim_{s\to0^+}\frac{r(s)}{s}=c^\top b.
    \]
    Therefore, there exist $\delta\in(0,T)$ and constants $m,M>0$ such that
    \begin{equation}
    \label{eq:r_local_bounds}
        ms\leq r(s)\leq Ms,
        \qquad s\in[0,\delta].
    \end{equation}
    Also, $r(s)>0$ for every $s>0$. Hence, by continuity, there exists
    $\rho>0$ such that
    \begin{equation}
    \label{eq:r_away_zero}
        r(s)\geq\rho,
        \qquad s\in[\delta,T].
    \end{equation}
    
    For the first term in \eqref{eq:cDR_large_u}, using
    $1-e^{-z}\leq z$ for $z\geq0$, we obtain
    \begin{align*}
        \frac{y_0}{u}
        \int_0^T
        e^{-a_u(T-t)}
        \bigl(1-e^{-a_ut}\bigr)
        \,\mbox{d}t
        \leq
        \frac{y_0a_u}{u}
        \int_0^T t\,\mbox{d}t
        =O(u^{-2}).
    \end{align*}
    
    For the third term,
    \begin{align*}
        \frac{K}{u}
        \int_0^T
        e^{-a_u(T-t)}
        \int_0^t
        \frac{\mbox{d}s}{(K+ur(s))^2}
        \,\mbox{d}t
        \leq
        \frac{KT}{u}
        \int_0^T
        \frac{\mbox{d}s}{(K+ur(s))^2}.
    \end{align*}
    Using \eqref{eq:r_local_bounds} and \eqref{eq:r_away_zero},
    \begin{align*}
        \int_0^T
        \frac{\mbox{d}s}{(K+ur(s))^2}
        \leq
        \int_0^\delta
        \frac{\mbox{d}s}{(K+ums)^2}
        +
        \int_\delta^T
        \frac{\mbox{d}s}{(K+u\rho)^2}
        =
        O(u^{-1}),
    \end{align*}
    and therefore the third term is $O(u^{-2})$.
    
    Now define
    \[
        I_u
        :=
        \frac{1}{u}
        \int_0^T
        \frac{(T-s)e^{-a_u(T-s)}}{K+ur(s)}
        \,\mbox{d}s.
    \]
    We show that
    \[
        I_u
        =
        \Theta\!\left(\frac{\log u}{u^2}\right).
    \]
    For the upper bound, using \eqref{eq:r_local_bounds} on $[0,\delta]$
    and \eqref{eq:r_away_zero} on $[\delta,T]$, we obtain
    \begin{align*}
        I_u
        \leq
        \frac{T}{u}
        \int_0^\delta
        \frac{\mbox{d}s}{K+ums}
        +
        \frac{T}{u}
        \int_\delta^T
        \frac{\mbox{d}s}{K+u\rho}
        =
        O\!\left(\frac{\log u}{u^2}\right).
    \end{align*}
    For the lower bound, using $r(s)\leq Ms$ on $[0,\delta]$ gives
    \begin{align*}
        I_u
        \geq
        \frac{(T-\delta)e^{-a_uT}}{u}
        \int_0^\delta
        \frac{\mbox{d}s}{K+uMs}
        =
        \frac{(T-\delta)e^{-a_uT}}{Mu^2}
        \log\!\left(1+\frac{uM\delta}{K}\right).
    \end{align*}
    Since $a_u\to0$ as $u\to\infty$, this implies
    \[
        I_u
        =
        \Theta\!\left(\frac{\log u}{u^2}\right).
    \]
    
    Consequently, \eqref{eq:cDR_large_u} yields
    \begin{equation*}
        \partial_u\mathrm{cDR}(u,T)
        =
        -\Theta\!\left(\frac{\log u}{u^2}\right)
        +
        O(u^{-2})
        <0
    \end{equation*}
    for all sufficiently large $u$.
    
    Thus, $\partial_u\mathrm{cDR}(u,T)>0$
    for all sufficiently small $u>0$, whereas $\partial_u\mathrm{cDR}(u,T)<0$
    for all sufficiently large $u$. Therefore, there exist
    $u_-,u_+>0$ such that the derivative has opposite signs at
    $u_-$ and $u_+$, and hence
    \[
        u\mapsto\mathrm{cDR}(u,T)
    \]
    is not monotone. This completes the proof.

\numberwithin{proposition}{section}

\section{Monotonicity of $\mathrm{cDR}(u,T)$ for the scalar IFFM systems}
\label{Apx:scalar}

In this appendix, we analyze the eight scalar two-state systems obtained
by specializing \eqref{eq:x_general} to $n=1$, $A=-1$, and $b=1$,
and the output models in Table~\ref{tab:systems} to
$c=d=\beta=\gamma=1$ and $K=0$.
These specializations give the systems summarized in
Table~\ref{tab:ag}. We refer to the resulting coupled systems as
Scalar Systems $1$--$8$. For each system, we derive
\[
    a_u(t),\qquad g_u(t),\qquad
    G_u(t):=\int_0^t g_u(s)\,\mbox{d}s,
\]
and apply Theorems~\ref{thm:cDR_monotone}
and~\ref{thm:cDR_nonmonotone} to determine the monotonicity of
$\mathrm{cDR}(u,T)$.

For all systems, the intermediate dynamics are
$\dot x_u=-x_u+u$, with initial condition $x_u(0)=x_0$ independent
of $u$. Differentiating with respect to $u$ gives
$\dot p_u=-p_u+1$ with $p_u(0)=0$, where
$p_u:=\partial_u x_u$. Hence,
\begin{equation}
\label{eq:xp_explicit}
    p_u=1-e^{-t},
    \qquad
    x_u=u+(x_0-u)e^{-t}.
\end{equation}

Let $q_u:=\partial_u y_u$. For a scalar system of the form
$\dot y_u=F(x_u,y_u,u)$, differentiation with respect to $u$ yields
\begin{equation}
\label{eq:q_scalar_general}
    \dot q_u
    =
    \partial_xF(x_u,y_u,u)p_u
    +\partial_yF(x_u,y_u,u)q_u
    +\partial_uF(x_u,y_u,u).
\end{equation}
Rewriting \eqref{eq:q_scalar_general} in the form
$\dot q_u+a_u(t)q_u=g_u(t)$, we identify
\begin{equation*}
    a_u(t)=-\partial_yF(x_u,y_u,u),
    \qquad
    g_u(t)=\partial_xF(x_u,y_u,u)p_u+\partial_uF(x_u,y_u,u).
\end{equation*}

\begin{subequations}
\paragraph{Model 1: $\dot y_u=\dfrac{x_u}{u}-y_u$.}
Here $F(x_u,y_u,u)=(x_u/u)-y_u.$ Thus $\partial_xF=1/u$, $\partial_yF=-1$, and $\partial_uF=-x_u/u^2$.

Therefore, $a_u=1$ and
\begin{equation}\label{eq:sys1:sca:g}
    g_u = \frac{up_u-x_u}{u^2} = \frac{u(1-e^{-t})-u-(x_0-u)e^{-t}}{u^2}
     = \frac{-x_0e^{-t}}{u^2} \le 0,
\end{equation}
$\forall t\ge 0, \forall u > 0$. Hence
\begin{equation}\label{eq:sys1:sca:G}
    G_u = -\frac{x_0}{u^2}\int_0^t e^{-s}\,\mbox{d}s
    = -\frac{x_0}{u^2}(1-e^{-t}) \le 0,
\end{equation}
$\forall t\ge 0, \forall u > 0$. Since $a_u=1$ is constant, the kernel $\lambda_u$ defined in Theorem~\ref{thm:cDR_monotone} is
\begin{equation}\label{eq:sys1:sca:lambda}
    \lambda_u=\int_t^T e^{-(s-t)}\,\mbox{d}s = 1-e^{-(T-t)}.
\end{equation}
In particular, $\lambda_u\ge 0$ for all $t\in[0,T]$, and $\dot\lambda_u=-e^{-(T-t)}<0, \forall t\in[0,T].$
\end{subequations}

\begin{proposition}\label{prop:sys1:sca}
    Consider Scalar System~1:
    $\dot x_u(t)=-x_u(t)+u$ and
    $\dot y_u(t)=(x_u(t)/u)-y_u(t)$, with arbitrary
    $x_0\geq0$ and $y_0\geq0$ chosen independently of $u$. Then, for
    every $T>0$, $\mathrm{cDR}(u,T)$ is monotone nonincreasing with
    respect to $u>0$.
\end{proposition}

\begin{proof}
    The sensitivity $q_u(t)=\partial_u y_u(t)$ satisfies $\dot q_u(t)+q_u(t)=g_u(t)$, where $g_u(t)$ is given in \eqref{eq:sys1:sca:g}.
    From \eqref{eq:sys1:sca:lambda}, $\lambda_u(t)\ge 0$, $\forall t\in[0,T]$, and $g_u(t)\le 0$, $\forall t\ge 0, \forall u > 0$.
    Thus $\lambda_u(t)g_u(t)\le 0$, $\forall t\in[0,T]$.
    By Theorem~\ref{thm:cDR_monotone}(i), $\partial_u\mathrm{cDR}(u,T)\le 0$.
    Hence $u\mapsto \mathrm{cDR}(u,T)$ is monotone nonincreasing.
\end{proof}

\begin{figure}[h!]
    \centering
    \subfloat[\label{Add1a}]{\includegraphics[width=0.23\linewidth]{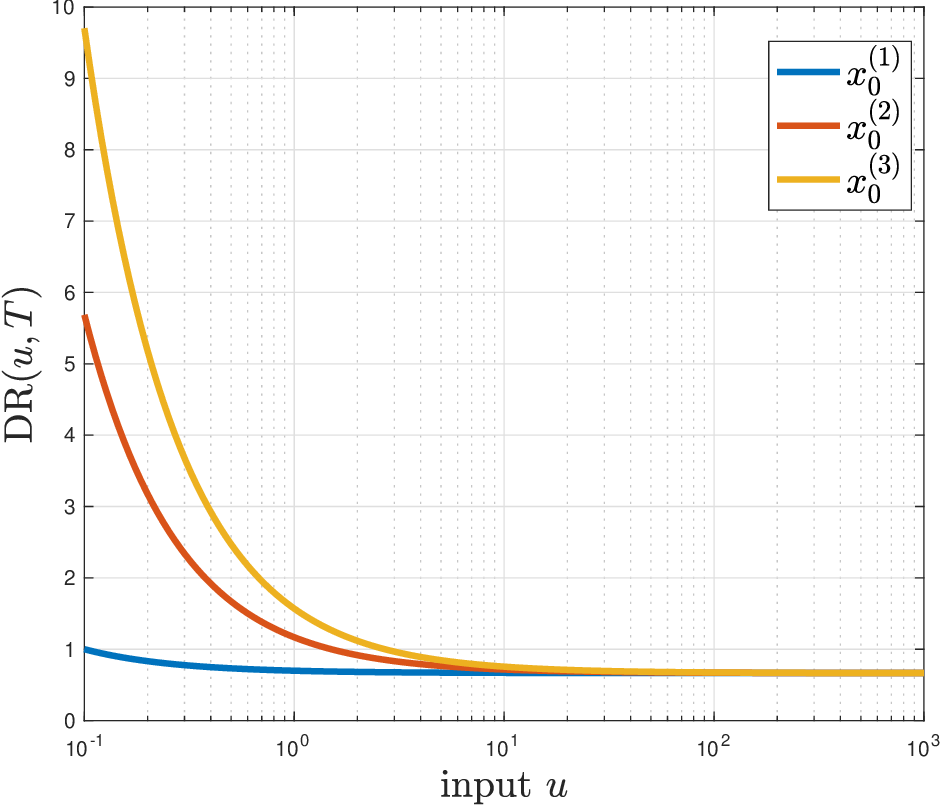}}
    \subfloat[\label{Add1b}]{\includegraphics[width=0.23\linewidth]{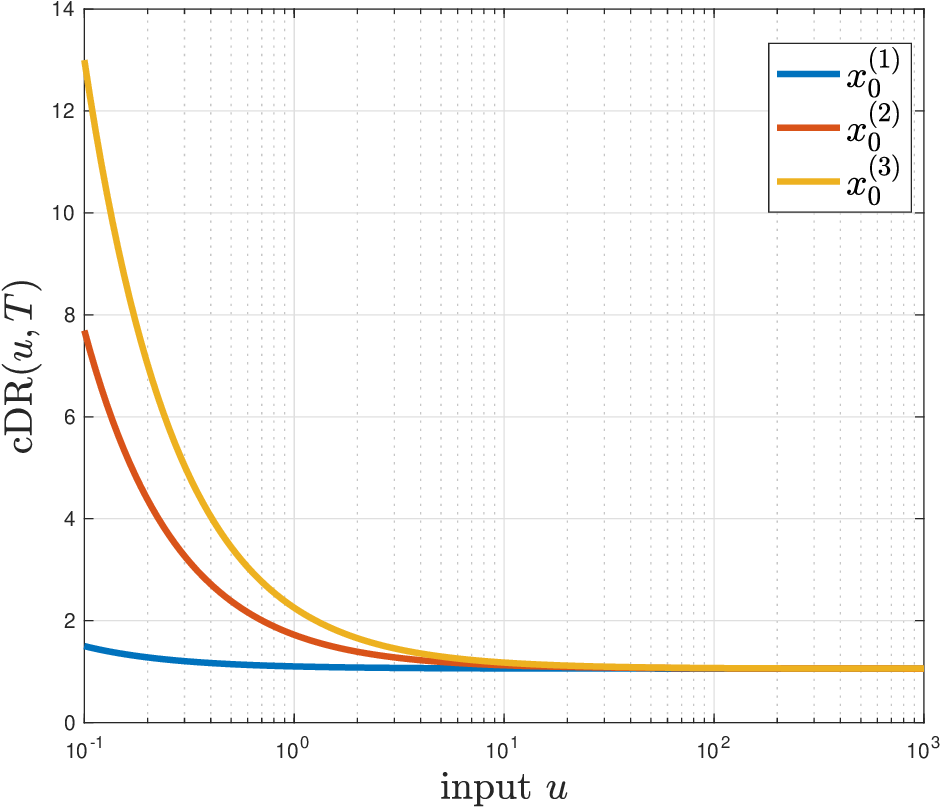}}
    \subfloat[\label{Add1c}]{\includegraphics[width=0.275\linewidth]{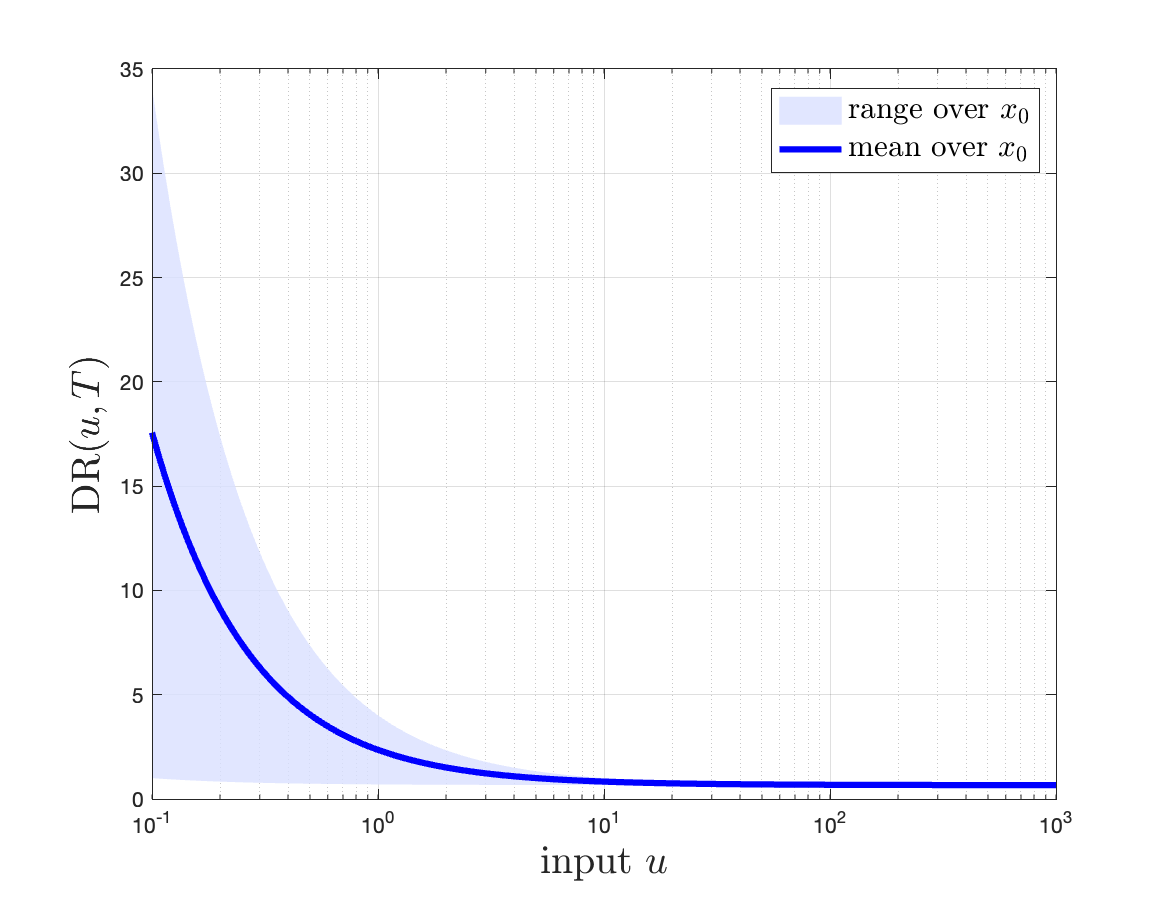}}
    \subfloat[\label{Add1d}]{\includegraphics[width=0.275\linewidth]{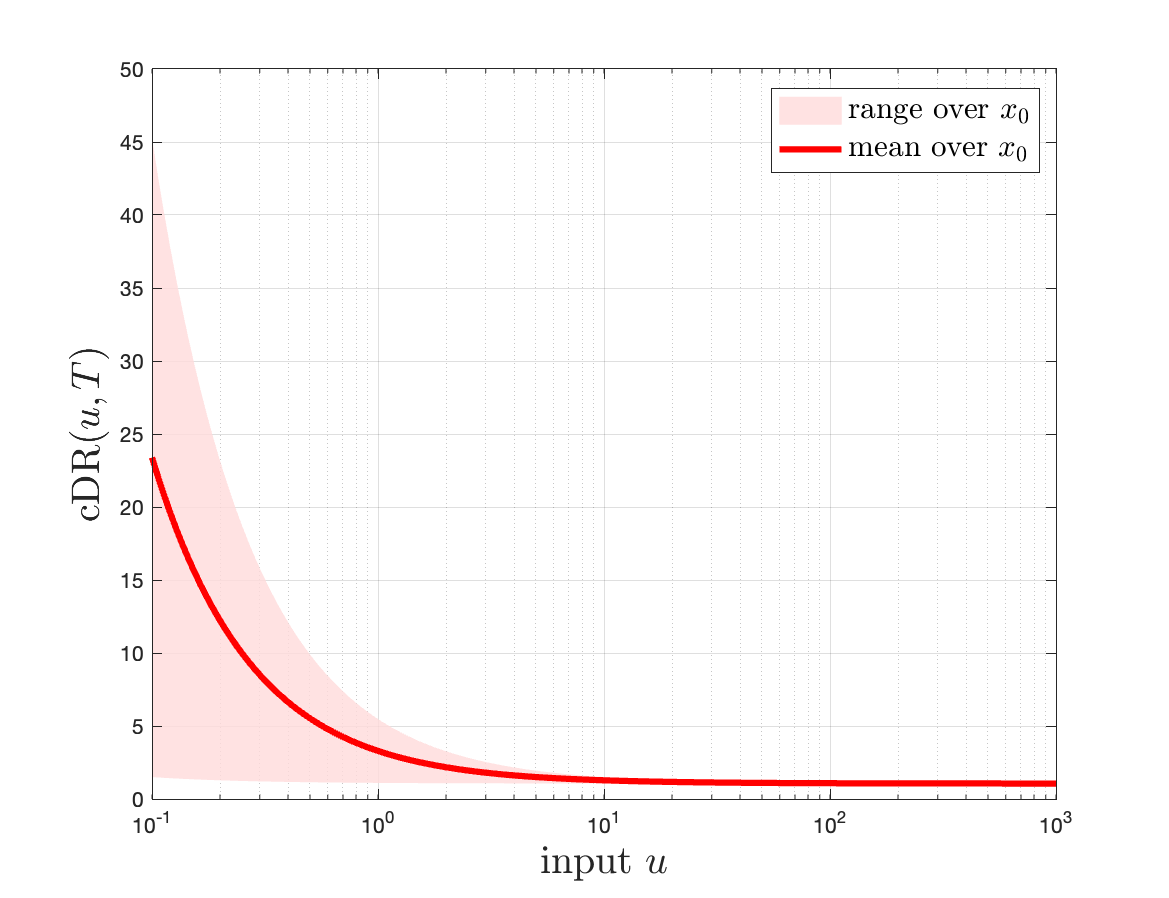}}
    \caption{Scalar System~1: Dose response $\mathrm{DR}(u,T)$ and cumulative dose response $\mathrm{cDR}(u,T)$ for $u\in[10^{-1},10^{3}]$. 
    (a)--(b) Responses for three fixed initial conditions $x_0$. 
    (c)--(d) Envelope (min--max band) and mean over $x_0\in[0.1,10]$. 
    The results illustrate that $\mathrm{cDR}(u,T)$ is monotone nonincreasing with respect to $u$, consistent with the theoretical analysis.}
    \label{Add1}
\end{figure}

\begin{subequations}
\paragraph{Model 2: $\dot y_u=x_u-u y_u$.}
Here $F(x_u,y_u,u)=x_u-u y_u$. Thus $\partial_xF=1$, $\partial_yF=-u$, and $\partial_uF=-y_u$.

Therefore, $a_u=u$ and $g_u=p_u-y_u$.
Since $a_u=u$ is constant, the kernel $\lambda_u$ defined in Theorem~\ref{thm:cDR_monotone} is
\begin{equation}\label{eq:sys2:sca:lambda}
    \lambda_u=\int_t^T e^{-u(s-t)}\,\mbox{d}s = \frac{1-e^{-u(T-t)}}{u}.
\end{equation}
In particular, $\lambda_u\ge 0$, for all $t\in[0,T]$ and $u > 0$, while the derivative is 
$\dot\lambda_u=-e^{-u(T-t)}<0, \forall t\in[0,T], \forall u > 0$.

Although \(\lambda_u(t)\ge 0\), Theorem~\ref{thm:cDR_monotone}(i) is not directly applicable, since \(g_u=p_u-y_u\) does not have a fixed sign uniformly in \(u\). 
Indeed, using \eqref{eq:xp_explicit} and solving $\dot y_u=x_u-u y_u$ with $y_0=1$ give, for $u\neq 1$,
\begin{equation*}
    g_u=p_u-y_u
    =-e^{-t}+\frac{u-x_0}{u-1}\bigl(e^{-t}-e^{-ut}\bigr).
\end{equation*}
Here $g_u(0)=p_u(0)-y_u(0)=-1<0, \forall u >0$ and for each fixed $t>0$, $\lim_{u\to 0^+}g_u(t)=-e^{-t}-x_0(1-e^{-t})<0.$
On the other hand, as $u\to\infty$,
\begin{equation*}
    g_u=\frac{(1-x_0)e^{-t}}{u-1}+o\!\left(\frac{1}{u}\right).
\end{equation*}
Thus, for instance when $0\le x_0<1$, for each fixed $t>0$, $g_u(t)>0$ for all sufficiently large $u$, while $g_u(t)<0$ for sufficiently small $u$. Hence, already for $0\le x_0<1$, $g_u$ does not have a fixed sign uniformly in $u$, and Theorem~\ref{thm:cDR_monotone}(i) cannot be used in general. 

We therefore turn to Theorem~\ref{thm:cDR_monotone}(ii), which only requires the sign of the accumulated quantity $G_u(t)$ together with $\dot\lambda_u(t)<0$.
We now show that $G_u\le 0$, $\forall t\ge 0, \forall u > 0$, assuming $y_0=y_{\mathrm{ss}}=1$ and arbitrary $x_0\ge 0$.
Using \eqref{eq:xp_explicit} and $\dot y_u=x_u-u y_u$, we rewrite $g_u$ as
\begin{equation}\label{eq:sys2:sca:g}
    g_u = p_u-y_u = \frac{1}{u}\dot y_u-\frac{x_0e^{-t}}{u}.
\end{equation}
Integrating \eqref{eq:sys2:sca:g} over $[0,t]$, we obtain
\begin{equation}\label{eq:sys2:sca:G1}
    G_u
    =
    \frac{1}{u}\bigl(y_u-y_0\bigr)
    -\frac{x_0}{u}\bigl(1-e^{-t}\bigr).
\end{equation}
Next, using $x_u$ in \eqref{eq:xp_explicit} and $y_0=1$, the $y$-subsystem may be written as
$\dot y_u+u(y_u-1)=(x_0-u)e^{-t}$ with $y_u(0)-1=0.$
By variation of parameters with integrating factor $e^{ut}$,
\begin{equation}\label{eq:sys2:sca:yminus1}
    y_u-1 = (x_0-u)\int_0^t e^{-u(t-s)}e^{-s}\,\mbox{d}s.
\end{equation}
Substituting \eqref{eq:sys2:sca:yminus1} into \eqref{eq:sys2:sca:G1} gives
\begin{align}
    G_u &= \frac{x_0-u}{u}\int_0^t e^{-u(t-s)}e^{-s}\,\mbox{d}s
    -\frac{x_0}{u}\int_0^t e^{-s}\,\mbox{d}s \notag\\
    &= -\int_0^t e^{-u(t-s)}e^{-s}\,\mbox{d}s
    +\frac{x_0}{u}\int_0^t\bigl(e^{-u(t-s)}-1\bigr)e^{-s}\,\mbox{d}s.
\label{eq:G-system2-final}
\end{align}
Since $e^{-u(t-s)}\le 1$ for all $0\le s\le t$, both terms in
\eqref{eq:G-system2-final} are nonpositive. Hence $G_u\le 0$, $\forall t\ge 0, \forall u > 0$.
\end{subequations}

\begin{proposition}\label{prop:sys2:sca}
    Consider Scalar System~2: $\dot x_u(t)=-x_u(t)+u$ and $\dot y_u(t)=x_u(t)-u y_u(t)$
    with arbitrary $x_0\ge 0$ and $y_0=y_{\mathrm{ss}}=1$. Then, for every $T>0$,
    $\mathrm{cDR}(u,T)$ is monotone nonincreasing with respect to $u>0$.
\end{proposition}

\begin{proof}
    The sensitivity $q_u(t)=\partial_u y_u(t)$ satisfies $\dot q_u(t)+u q_u(t)=g_u(t)$ where $g_u(t)$ is given by \eqref{eq:sys2:sca:g}.
    From \eqref{eq:sys2:sca:lambda}, $-\dot\lambda_u(t)>0$, $\forall t\in[0,T], \forall u>0$ while by \eqref{eq:G-system2-final}, $G_u\le 0$, $\forall t\ge 0, \forall u > 0$.
    Thus $-\dot\lambda_u(t)G_u(t)\le 0, \forall t\in[0,T].$
    By Theorem~\ref{thm:cDR_monotone}(ii), $\partial_u\mathrm{cDR}(u,T)\le 0$. Hence $u\mapsto \mathrm{cDR}(u,T)$ is monotone nonincreasing.
\end{proof}

\begin{figure}[h!]
    \centering
    \subfloat[\label{sys2a}]{\includegraphics[width=0.23\linewidth]{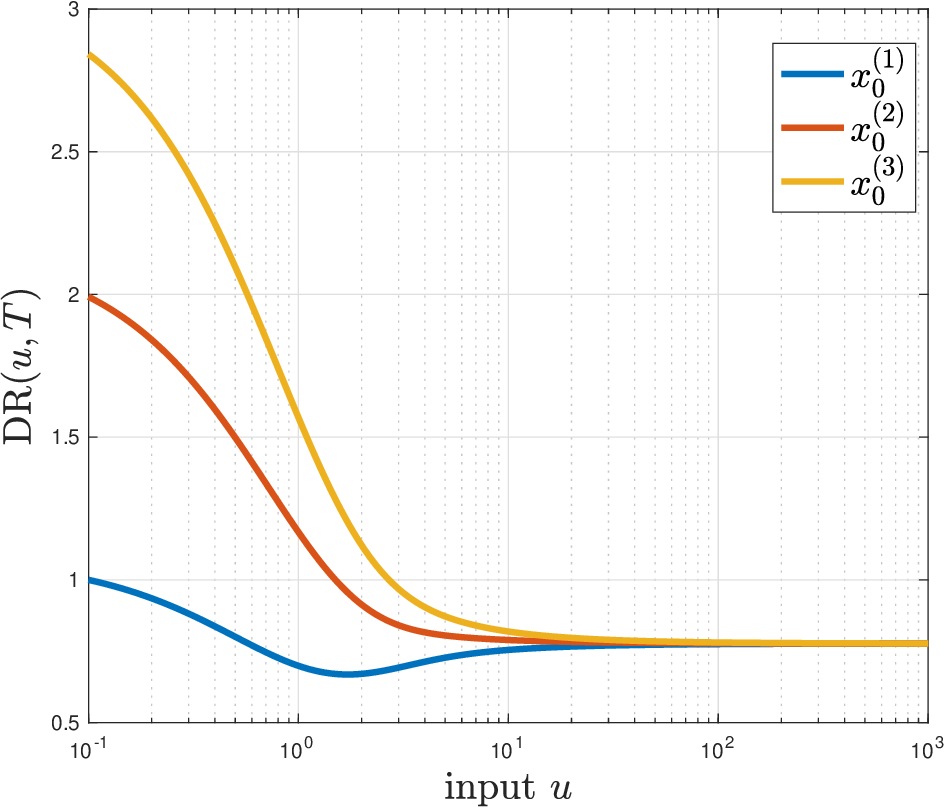}}
    \subfloat[\label{sys2b}]{\includegraphics[width=0.23\linewidth]{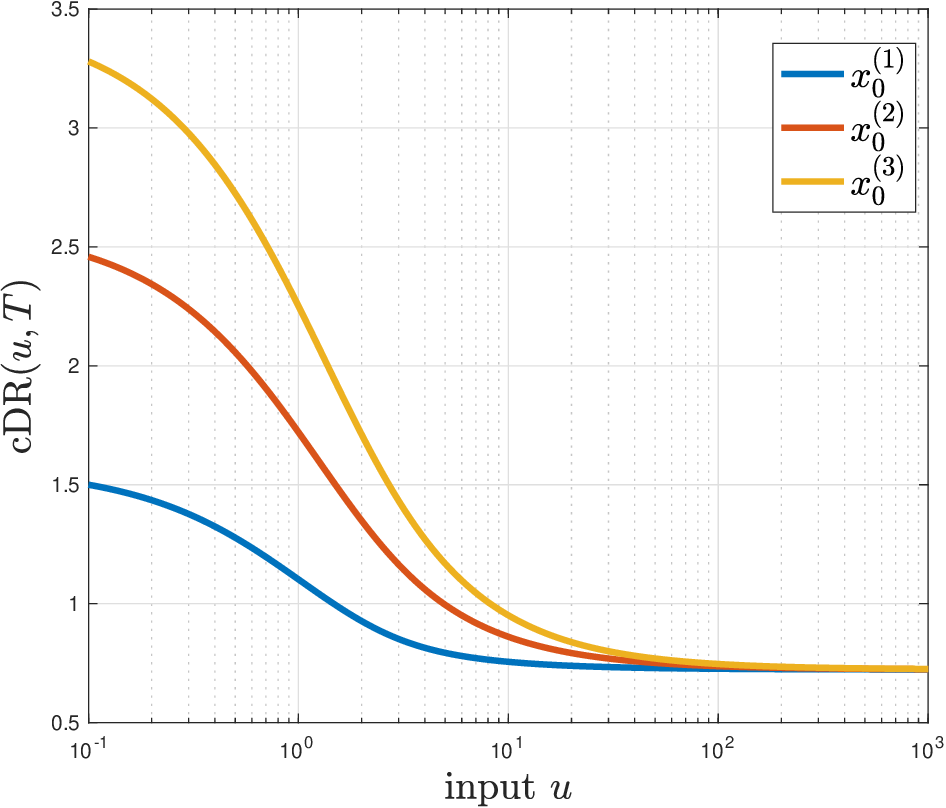}}
    \subfloat[\label{sys2c}]{\includegraphics[width=0.275\linewidth]{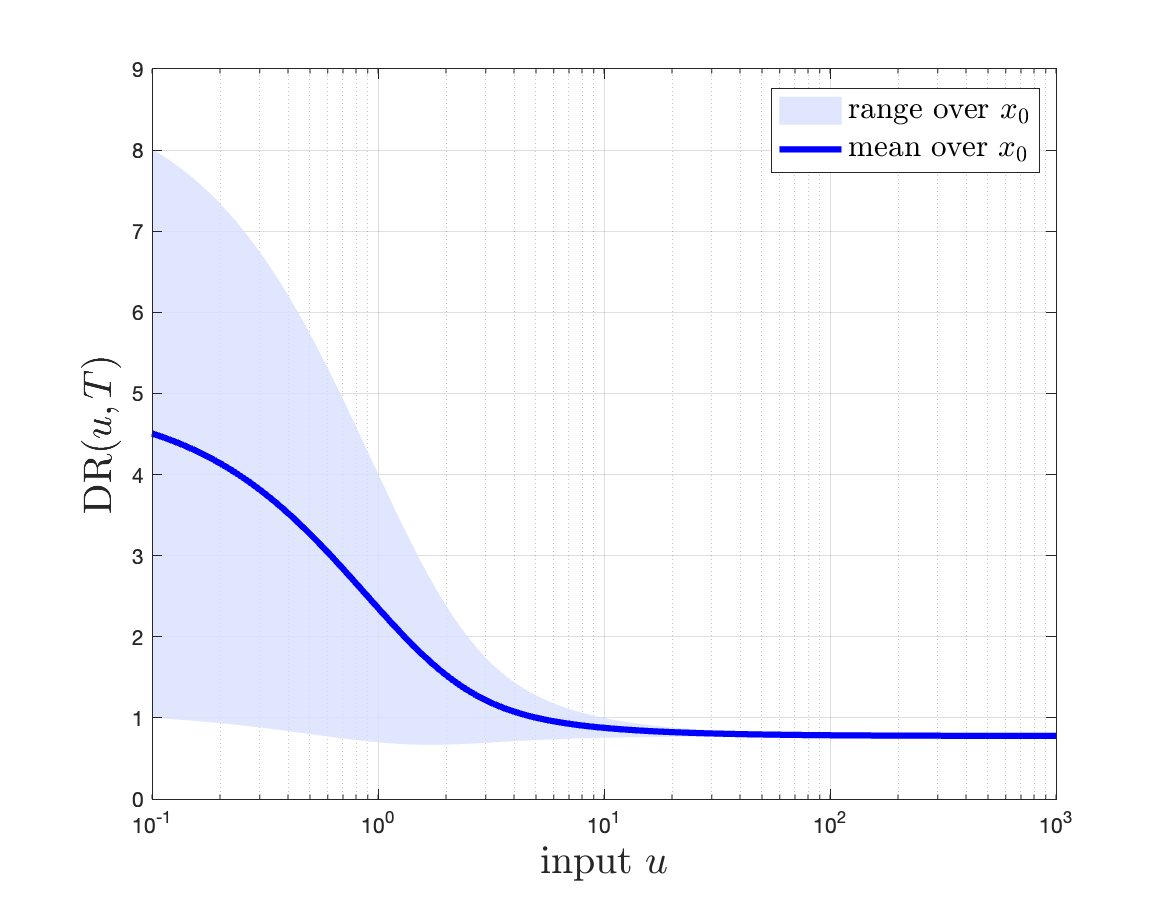}}
    \subfloat[\label{sys2d}]{\includegraphics[width=0.275\linewidth]{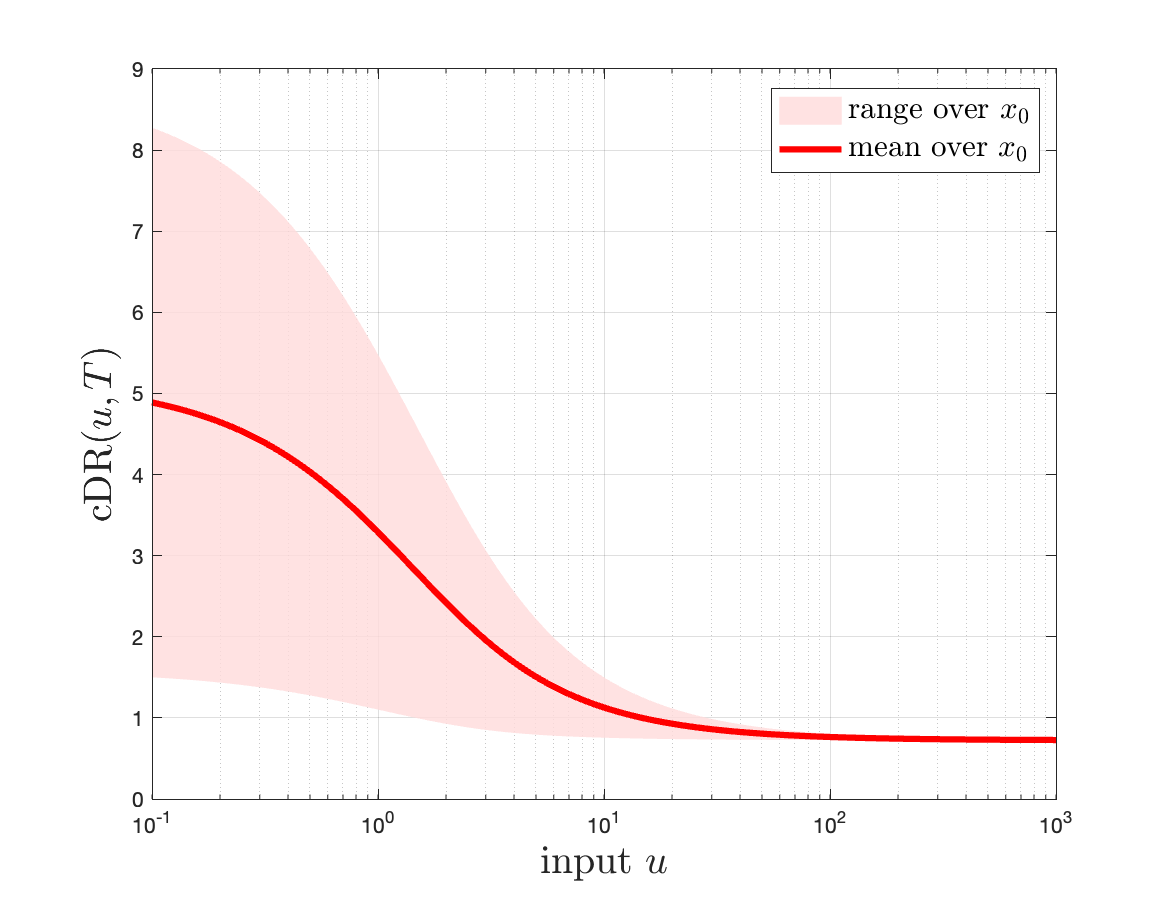}}
    \caption{Scalar System~2: Dose response $\mathrm{DR}(u,T)$ and cumulative dose response $\mathrm{cDR}(u,T)$ for $u\in[10^{-1},10^{3}]$. 
    (a)--(b) Responses for three fixed initial conditions $x_0$. 
    (c)--(d) Envelope (min--max band) and mean over $x_0\in[0.1,10]$. 
    The results show that $\mathrm{cDR}(u,T)$ is monotone nonincreasing with respect to $u$.}
    \label{fig:sys2}
\end{figure}

\begin{subequations}
\paragraph{Model 3: $\dot y_u=\dfrac{u}{x_u}-y_u$.}
Here $F(x_u,y_u,u)=(u/x_u)-y_u$. Thus $\partial_xF=-u/x_u^2$, $\partial_yF=-1$, and $\partial_uF=1/x_u$.

Therefore, $a_u=1$ and 
\begin{equation}\label{eq:sys3:sca:g}
    g_u =\frac{x_u-up_u}{x_u^2} 
    = \frac{u+(x_0-u)e^{-t}-u(1-e^{-t})}{x_u^2} 
    = \frac{x_0e^{-t}}{x_u^2} \ge 0,
\end{equation}
$\forall t\ge 0, \forall u > 0$. Hence
\begin{equation}\label{eq:sys3:sca:G}
    G_u=\int_0^t \frac{x_0e^{-s}}{x_u(s)^2}\,\mbox{d}s \ge 0,
\end{equation}
$\forall t\ge 0, \forall u > 0$. Since $a_u=1$ is constant, the kernel $\lambda_u$ defined in Theorem~\ref{thm:cDR_monotone} is
\begin{equation}\label{eq:sys3:sca:lambda}
    \lambda_u=\int_t^T e^{-(s-t)}\,\mbox{d}s = 1-e^{-(T-t)}.
\end{equation}
In particular, $\lambda_u\ge 0$ for all $t\in[0,T]$, and $\dot\lambda_u=-e^{-(T-t)}<0, \forall t\in[0,T].$
\end{subequations}

\begin{proposition}\label{prop:sys3:sca}
    Consider Scalar System~3:
    $\dot x_u(t)=-x_u(t)+u$ and
    $\dot y_u(t)=(u/x_u(t))-y_u(t)$, with arbitrary
    $x_0>0$ and $y_0\geq0$ chosen independently of $u$. Then, for every
    $T>0$, $\mathrm{cDR}(u,T)$ is monotone nondecreasing with respect
    to $u>0$.
\end{proposition}

\begin{proof}
    The sensitivity $q_u(t)=\partial_u y_u(t)$ satisfies $\dot q_u(t)+q_u(t)=g_u(t)$ where $g_u(t)$ is given in \eqref{eq:sys3:sca:g}.
    From \eqref{eq:sys3:sca:lambda}, $\lambda_u(t)\ge 0$, $\forall t\in[0,T]$, and $g_u(t)\ge 0$, $\forall t\ge 0, \forall u > 0$.
    Thus $\lambda_u(t)g_u(t)\ge 0, \forall t\in[0,T].$
    By Theorem~\ref{thm:cDR_monotone}(i), $\partial_u\mathrm{cDR}(u,T)\ge 0.$
    Hence $u\mapsto \mathrm{cDR}(u,T)$ is monotone nondecreasing.
\end{proof}

\begin{figure}[h!]
    \centering
    \subfloat[\label{sys3a}]{\includegraphics[width=0.23\linewidth]{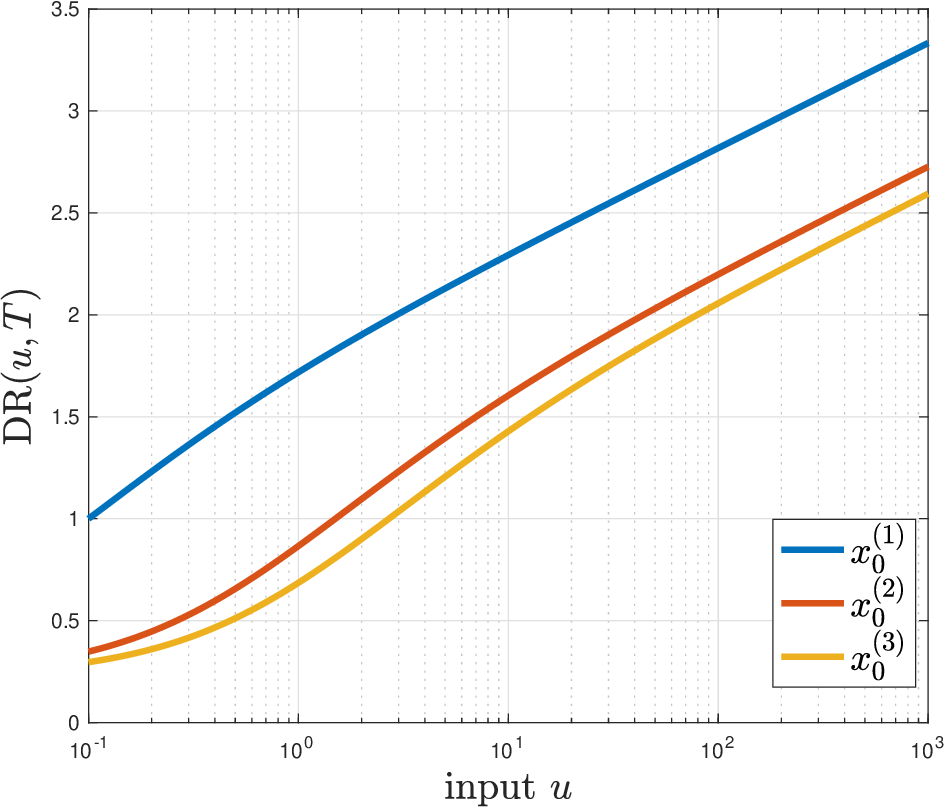}}
    \subfloat[\label{sys3b}]{\includegraphics[width=0.23\linewidth]{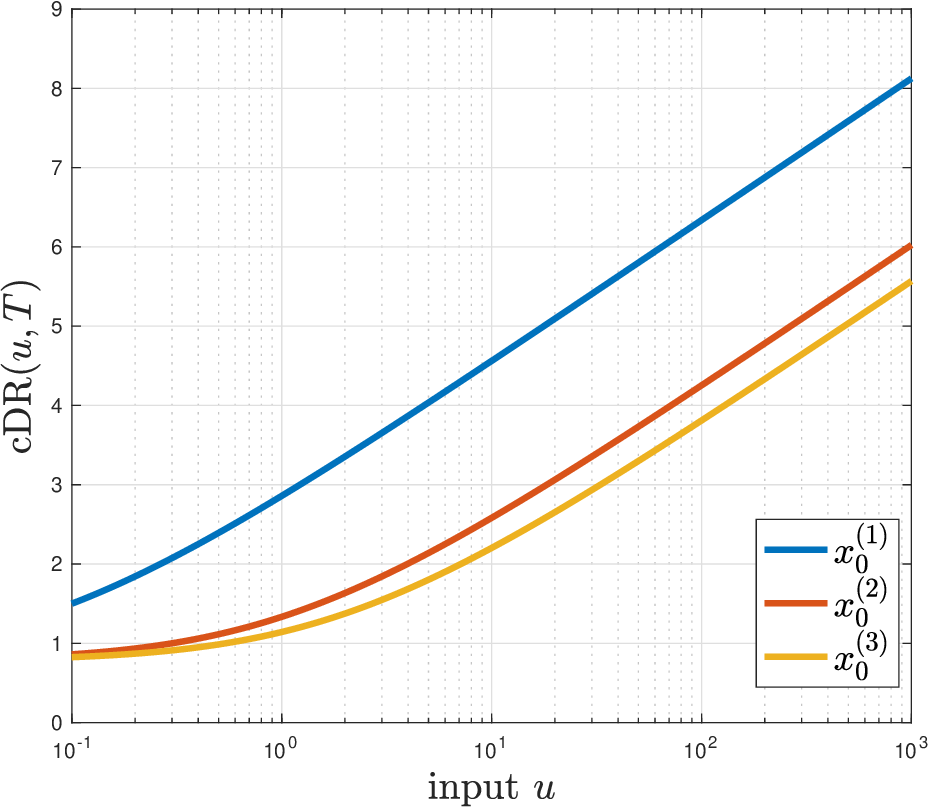}}
    \subfloat[\label{sys3c}]{\includegraphics[width=0.275\linewidth]{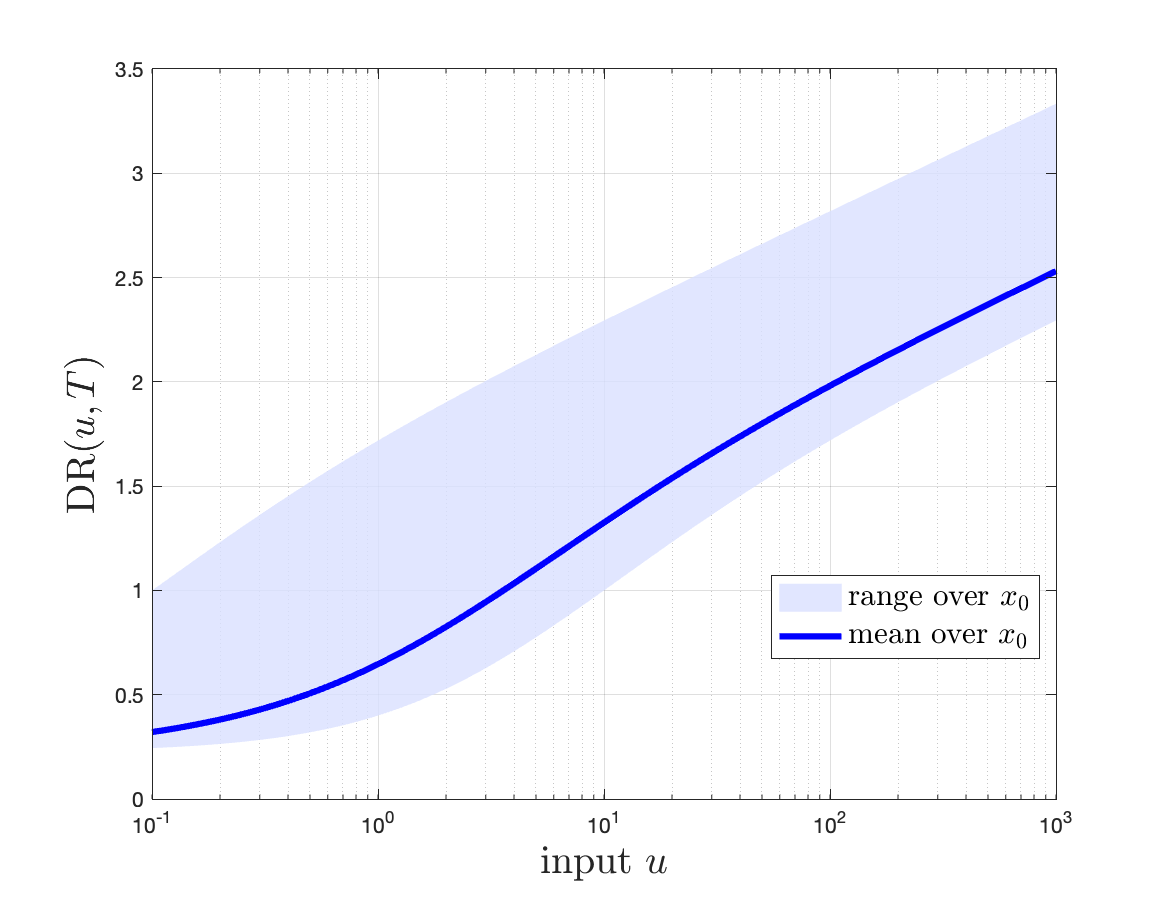}}
    \subfloat[\label{sys3d}]{\includegraphics[width=0.275\linewidth]{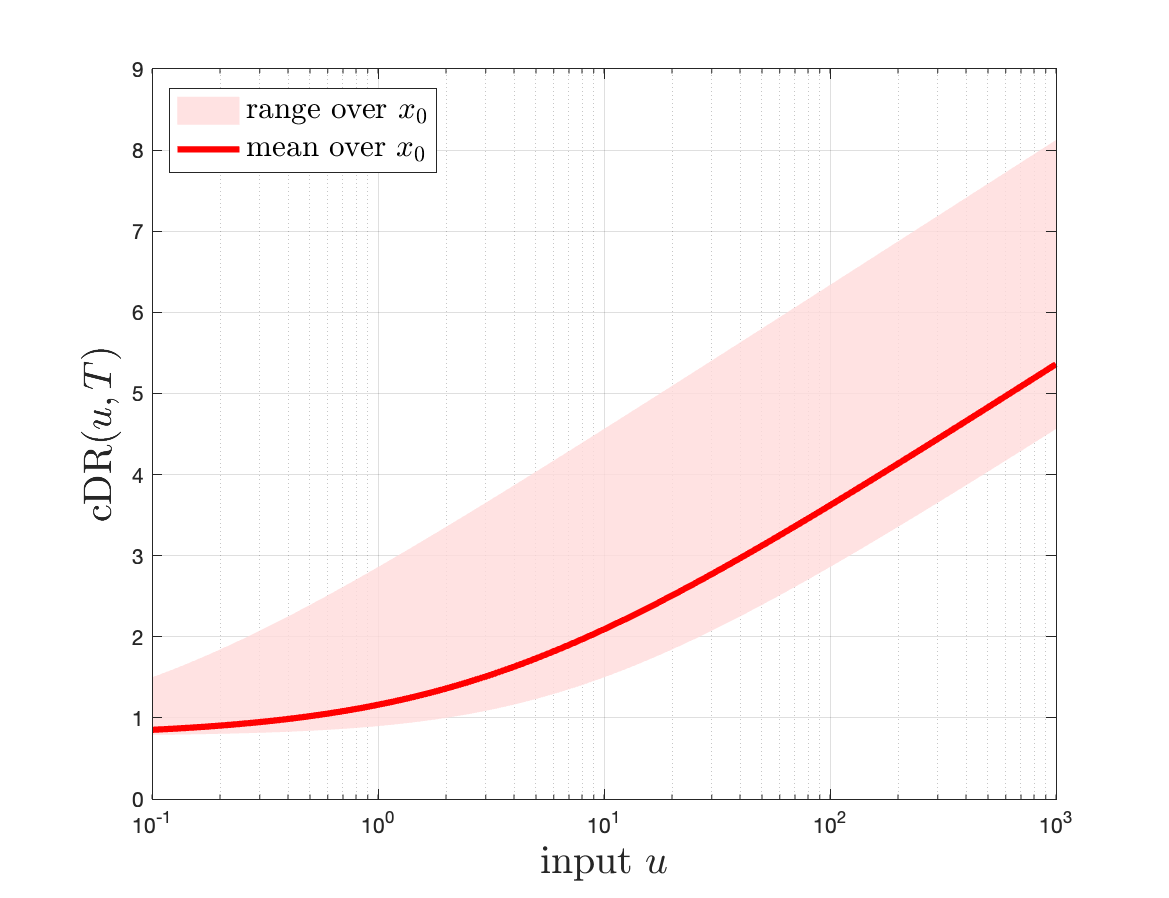}}
    \caption{Scalar System~3: Dose response $\mathrm{DR}(u,T)$ and cumulative dose response $\mathrm{cDR}(u,T)$ for $u\in[10^{-1},10^{3}]$. 
    (a)--(b) Responses for three fixed initial conditions $x_0$. 
    (c)--(d) Envelope (min--max band) and mean over $x_0\in[0.1,10]$. 
    The results show that $\mathrm{cDR}(u,T)$ is monotone nondecreasing with respect to $u$.}
    \label{fig:sys3}
\end{figure}

\begin{subequations}
\paragraph{Model 4: $\dot y_u=\dfrac{1}{x_u}-\dfrac{1}{u}y_u$.}
Here $F(x_u,y_u,u)=(1/x_u)-(y_u/u)$. Thus $\partial_xF=-1/x_u^2$, $\partial_yF=-1/u$, and $\partial_uF=y_u/u^2$.
Therefore, $a_u=1/u$ and $g_u=-(p_u/x_u^2) + (y_u/u^2).$ Hence
\begin{equation}\label{eq:sys4:sca:G}
    G_u=\int_0^t \left[-\frac{p_u(s)}{x_u(s)^2}+\frac{y_u(s)}{u^2}\right]\,\mbox{d}s.
\end{equation}
Since $a_u=1/u$ is constant, the kernel $\lambda_u$ defined in Theorem~\ref{thm:cDR_monotone} is
\begin{equation}\label{eq:sys4:sca:lambda}
    \lambda_u=\int_t^T e^{-(s-t)/u}\,\mbox{d}s
    =u\bigl(1-e^{-(T-t)/u}\bigr).
\end{equation}
In particular, $\lambda_u\ge 0$ for all $t\in[0,T]$ and $u>0$, while the derivative is $\dot\lambda_u=-e^{-(T-t)/u}<0,\forall t\in[0,T], \forall u>0.$ Using $x_u=up_u+x_0e^{-t}$ and $\dot y_u=(1/x_u)-(y_u/u)$, we rewrite $g_u$ as
\begin{equation}\label{eq:sys4:sca:g}
    g_u
    =-\frac{1}{u}\dot y_u+\frac{x_0e^{-t}}{u\,x_u^2}.
\end{equation}
Integrating \eqref{eq:sys4:sca:g} over $[0,t]$ and using $y_0=1$, we obtain
\begin{equation}\label{eq:sys4:sca:G-rewrite}
    G_u =\frac{1-y_u}{u}+\frac{x_0}{u}\int_0^t \frac{e^{-s}}{x_u(s)^2}\,\mbox{d}s.
\end{equation}

We now analyze the sign of $\partial_u\mathrm{cDR}(u,T)$ for small and large values of $u$.

\begin{itemize}
    \item First, assume that $0<u\le x_0$. Then $x_u=u+(x_0-u)e^{-t}\ge u, \forall t\ge 0.$
    Evaluating the $y$-subsystem at $y_u=1$ gives $\dot y_u\big|_{y_u=1}=(1/x_u)-(1/u) \le 0.$
    Since $y_0=1$, the comparison principle yields $y_u\le 1$ for all $t\ge 0$. Therefore, by \eqref{eq:sys4:sca:G-rewrite},
    $G_u\ge 0$, $\forall t\ge 0$, for all $0<u\le x_0$.
    Moreover, $G_u(t)>0$ for every $t>0$ because the second term in \eqref{eq:sys4:sca:G-rewrite} is strictly positive.
    Since $-\dot\lambda_u>0$ and $G_u\geq0$ with
    $G_u\not\equiv0$, \eqref{eq:int_by_part_G} gives
    $\partial_u\mathrm{cDR}(u,T)>0$ for all $0<u\leq x_0$.

    \item Next, we consider the case of large $u$. Using variation of parameters on
    $\dot y_u+(1/u)y_u=1/x_u$ and subtracting from $y_0=1$, we rewrite
    $G_u$ in \eqref{eq:sys4:sca:G-rewrite} as
    \begin{equation}\label{eq:sys4:sca:G-rewrite2}
        G_u(t)
        =
        \frac{1-e^{-t/u}}{u}
        -\frac1u\int_0^t \frac{e^{-(t-s)/u}}{x_u(s)}\,\mbox{d}s
        +\frac{x_0}{u}\int_0^t \frac{e^{-s}}{x_u(s)^2}\,\mbox{d}s.
    \end{equation}
    Rather than requiring a uniform sign for $G_u(t)$ on $[0,T]$, we estimate
    $\partial_u\mathrm{cDR}(u,T)$ directly. Using
    $-\dot\lambda_u(t)=e^{-(T-t)/u}$, \eqref{eq:sys4:sca:G-rewrite2} and exchanging the order of
    integration in the negative term yield
    \begin{align*}
        \partial_u\mathrm{cDR}(u,T) &= -\int_0^T \dot{\lambda}_u(t)G_u(t)\,\mbox{d}t \\
        &=
        \int_0^T e^{-(T-t)/u}\frac{1-e^{-t/u}}{u}\,\mbox{d}t -\frac1u\int_0^T
        \frac{(T-s)e^{-(T-s)/u}}{x_u(s)}\,\mbox{d}s \\
        &\quad
        +\frac{x_0}{u}\int_0^T
        \left(\int_s^T e^{-(T-t)/u}\,\mbox{d}t\right)
        \frac{e^{-s}}{x_u(s)^2}\,\mbox{d}s.
    \end{align*}
    The first and third terms are $O(u^{-2})$ as $u\to\infty$. For the negative term, set $z=1-e^{-s}$, so that $s=-\log(1-z)$ and $x_u(s)=x_0+(u-x_0)z.$
    Then
    \begin{align*}
        \int_0^T
        \frac{(T-s)e^{-(T-s)/u}}{x_u(s)}\,\mbox{d}s
        =
        \int_0^{1-e^{-T}}
        \frac{h_u(z)}
             {x_0+(u-x_0)z}\,\mbox{d}z,
    \end{align*}
    where
    \[
        h_u(z)
        :=
        \frac{(T+\log(1-z))
        e^{-(T+\log(1-z))/u}}{1-z}.
    \]
    Since $h_u(0)=Te^{-T/u}=T+O(u^{-1})$ and
    $h_u(z)-h_u(0)=O(z)$ uniformly for
    $z\in[0,1-e^{-T}]$, we obtain
    \begin{align*}
        \int_0^T
        \frac{(T-s)e^{-(T-s)/u}}{x_u(s)}\,\mbox{d}s
        =
        h_u(0)
        \int_0^{1-e^{-T}}
        \frac{\mbox{d}z}{x_0+(u-x_0)z}
        +O(u^{-1})
        =
        \frac{T\log u}{u}+O(u^{-1}),
        \quad u\to\infty.
    \end{align*}
    Therefore
    \begin{equation}\label{eq:sys4:sca:cDR-large-u}
        \partial_u\mathrm{cDR}(u,T)
        =
        -\frac{T\log u}{u^2}
        +O(u^{-2}),
        \qquad u\to\infty.
    \end{equation}
    Hence $\partial_u\mathrm{cDR}(u,T)<0$ for all sufficiently large $u$.
\end{itemize}
\end{subequations}

\begin{proposition}\label{prop:sys4:sca}
    Consider Scalar System~4:
    $\dot x_u(t)=-x_u(t)+u$ and
    $\dot y_u(t)=(1/x_u(t))-(y_u(t)/u)$,
    with arbitrary $x_0>0$ and $y_0=y_{\mathrm{ss}}=1$.
    Then, for every $T>0$, there exist $u_-,u_+>0$ such that
    \begin{equation*}
        \partial_u\mathrm{cDR}(u_-,T)>0,
        \qquad
        \partial_u\mathrm{cDR}(u_+,T)<0.
    \end{equation*}
    In particular, $\mathrm{cDR}(u,T)$ is not monotone with respect to $u$.
\end{proposition}

\begin{proof}
The sensitivity $q_u(t)=\partial_u y_u(t)$ satisfies
$\dot q_u(t)+(1/u)q_u(t)=g_u(t)$, where $g_u(t)$ is given by
\eqref{eq:sys4:sca:g}. Also, $-\dot\lambda_u(t)>0$ for all
$t\in[0,T]$.

\begin{enumerate}
    \item If $0<u\leq x_0$, then $x_u(t)\geq u$ and thus
    $y_u(t)\leq1$ for all $t\geq0$. By
    \eqref{eq:sys4:sca:G-rewrite}, this implies $G_u(t)\geq0$ on
    $[0,T]$, with $G_u\not\equiv0$. Since
    $-\dot\lambda_u>0$, \eqref{eq:int_by_part_G} gives
    $\partial_u\mathrm{cDR}(u,T)>0$.

    \item For sufficiently large $u$,
    \eqref{eq:sys4:sca:cDR-large-u} gives
    $\partial_u\mathrm{cDR}(u,T)<0$.
\end{enumerate}

Therefore, choosing any $u_-\in(0,x_0]$ and a sufficiently large
$u_+$ gives
\[
    \partial_u\mathrm{cDR}(u_-,T)>0,
    \qquad
    \partial_u\mathrm{cDR}(u_+,T)<0.
\]
Hence, $\mathrm{cDR}(u,T)$ is not monotone with respect to $u$.
\end{proof}

\begin{figure}[h!]
    \centering
    \subfloat[\label{sys4a}]{\includegraphics[width=0.23\linewidth]{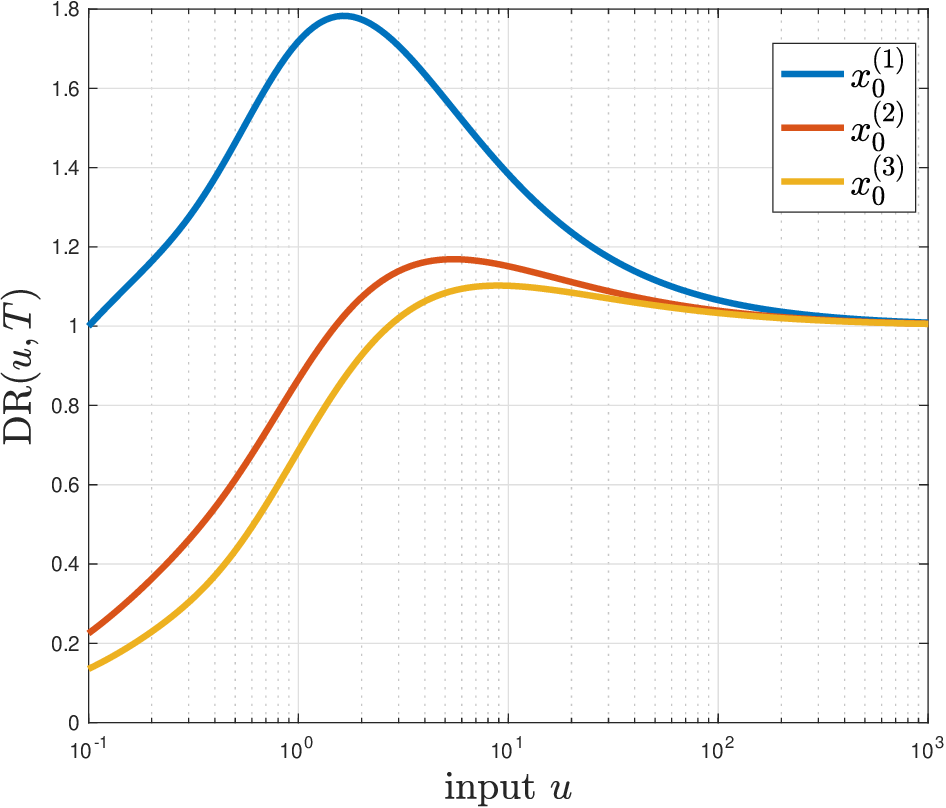}}
    \subfloat[\label{sys4b}]{\includegraphics[width=0.23\linewidth]{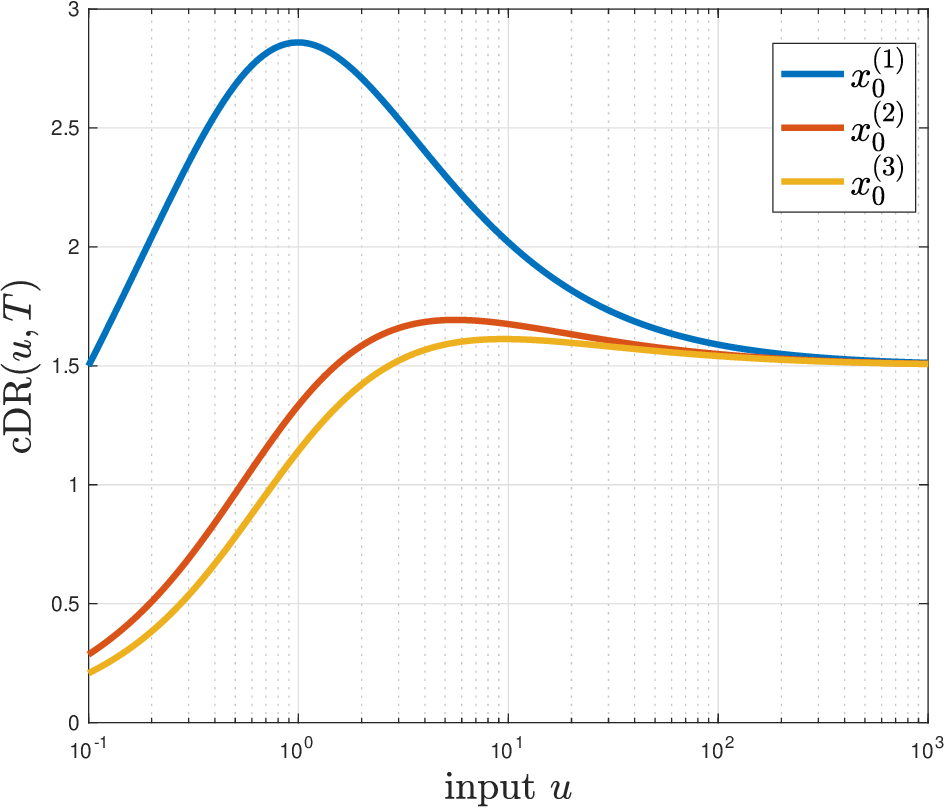}}
    \subfloat[\label{sys4c}]{\includegraphics[width=0.275\linewidth]{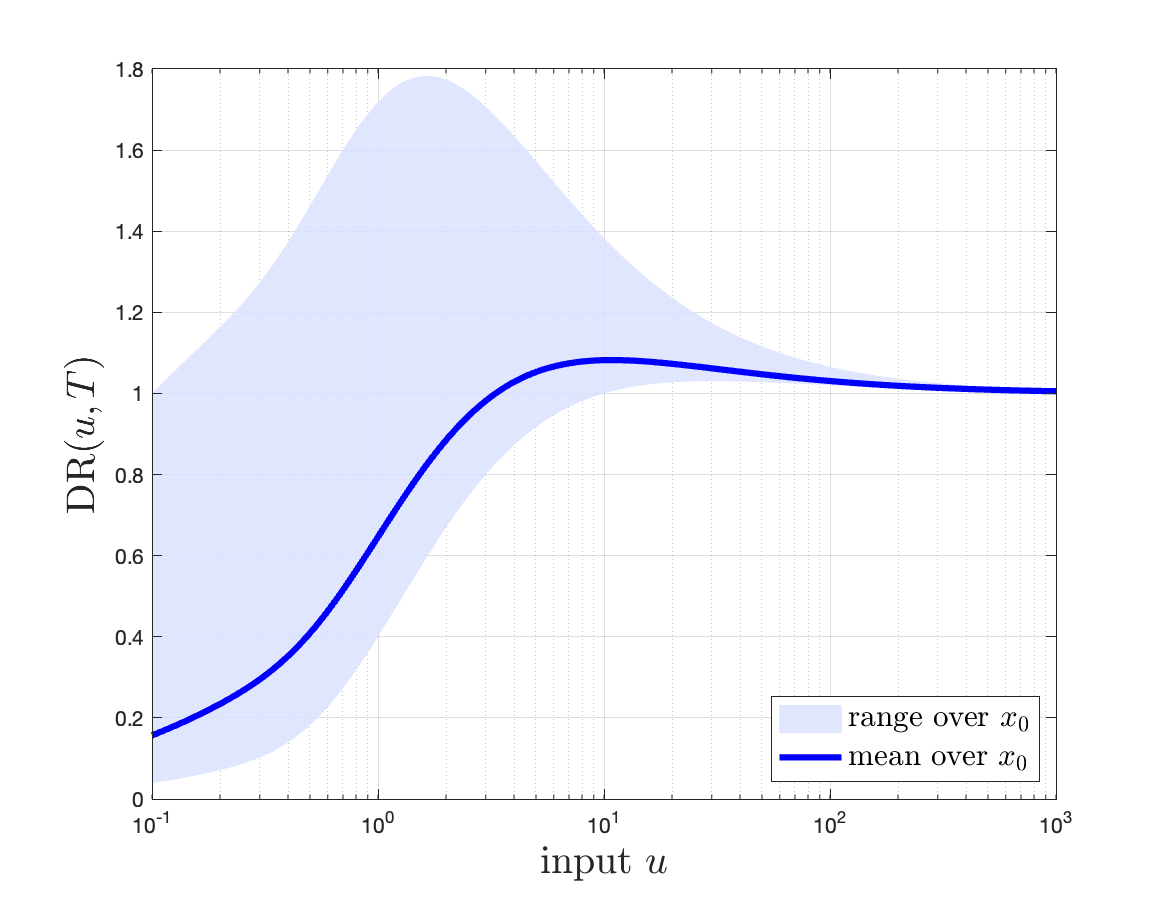}}
    \subfloat[\label{sys4d}]{\includegraphics[width=0.275\linewidth]{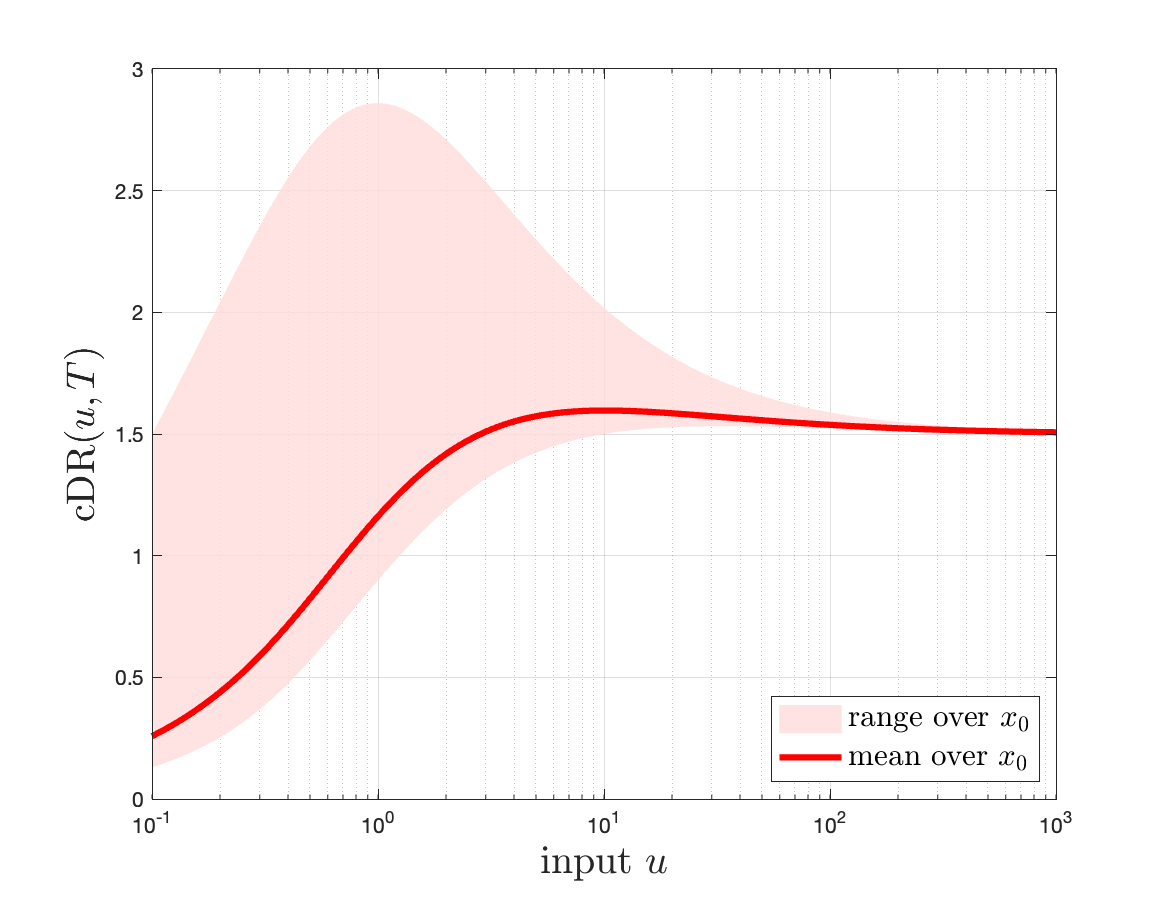}}
    \caption{Scalar System~4: Dose response $\mathrm{DR}(u,T)$ and cumulative dose response $\mathrm{cDR}(u,T)$ for $u\in[10^{-1},10^{3}]$. 
    (a)--(b) Responses for three fixed initial conditions $x_0$. 
    (c)--(d) Envelope (min--max band) and mean over $x_0\in[0.1,10]$. 
    The results show that $\mathrm{cDR}(u,T)$ is not monotone with respect to $u$, exhibiting a change in monotonicity.}
    \label{fig:sys4}
\end{figure}

\begin{subequations}
\paragraph{Model 5: $\dot y_u=u-x_u y_u$.}
Here $F(x_u,y_u,u)=u-x_uy_u$. Thus $\partial_xF=-y_u$, $\partial_yF=-x_u$, and $\partial_uF=1$.

Therefore, $a_u=x_u$ and $g_u=1-y_up_u.$ Hence
\begin{equation}\label{eq:sys5:sca:G}
        G_u=\int_0^t \bigl[1-y_u(s)p_u(s)\bigr]\,\mbox{d}s
    = t-\int_0^t y_u(s)p_u(s)\,\mbox{d}s.
\end{equation}

Since $a_u=x_u\ge 0$, the kernel $\lambda_u$ defined in Theorem~\ref{thm:cDR_monotone} satisfies
\begin{equation}\label{eq:sys5:sca:lambda}
    \lambda_u=\int_t^T \exp\!\left(-\int_t^s x_u(\tau)\,\mbox{d}\tau\right)\,\mbox{d}s.
\end{equation}
In particular, $\lambda_u\ge 0$ for all $t\in[0,T]$ and $u>0$.

We now distinguish two cases.

\begin{itemize}
    \item If $x_0\ge u$, then $x_u=u+(x_0-u)e^{-t}\ge u, \forall t\ge 0$. Since $y_0=y_{\mathrm{ss}}=1$, evaluating the $y$-subsystem at $y_u=1$ gives $\dot y_u\big|_{y_u=1}=u-x_u\le 0.$
    Hence, by comparison principle, $y_u\le 1$ for all $t\ge 0$. Moreover, since $0\le p_u\coloneqq 1-e^{-t}\le 1$, it follows that
    $y_up_u\le 1$, and therefore
    \begin{equation}\label{eq:sys5:sca:gcase1}
        g_u=1-y_up_u\geq0,
        \qquad \forall t\geq0,\, 0<u\leq x_0.
    \end{equation}

    \item Next, suppose that $x_0<u$. Then $x_u=u+(x_0-u)e^{-t}\le u, \forall t\ge 0.$
    Again, since $y_0=1$, evaluating at $y_u=1$ gives $\dot y_u\big|_{y_u=1}=u-x_u\ge 0,$
    so by comparison principle $y_u\ge 1$ for all $t\ge 0$.
    Using $x_u=u+(x_0-u)e^{-t}$ in the $y$-subsystem, we rewrite $g_u$ as
    \begin{equation}\label{eq:sys5:sca:g}
    \begin{aligned}
        g_u=1-y_up_u
        &=1-(1-e^{-t})y_u\\
        &=1-y_u+\frac{u-x_0+x_0}{u}e^{-t}y_u\\
        &=1-\frac{x_u}{u}y_u
          +\frac{x_0}{u}e^{-t}y_u\\
        &=\frac{1}{u}\dot y_u
          +\frac{x_0}{u}e^{-t}y_u.
    \end{aligned}
    \end{equation}
    Integrating \eqref{eq:sys5:sca:g} over $[0,t]$ and using $y_0=1$, we obtain
    \begin{equation}\label{eq:sys5:sca:Gcase2}
        G_u
        =\frac{1}{u}(y_u-1)+\frac{x_0}{u}\int_0^t e^{-s}y_u(s)\,\mbox{d}s \ge 0,
    \end{equation}
    since $y_u\ge 1$ and $y_u\ge 0$. Thus, $G_u(t)\geq0$, $\forall t\geq0$ and $u>x_0$.
    
    It remains to show that $\dot\lambda_u\le 0$ when $x_0<u$. Differentiating \eqref{eq:sys5:sca:lambda} gives
    $\dot\lambda_u=-1+x_u\lambda_u.$
    Since $x_0<u$, the function $x_u=u+(x_0-u)e^{-t}$ is nondecreasing. Thus, for every $s\in[t,T]$,
    $x_u(\tau)\ge x_u(t), \forall \tau\in[t,s].$ Therefore,
    \begin{equation*}
        \exp\!\left(
            -\int_t^s x_u(\tau)\,\mbox{d}\tau
        \right)
        \leq
        e^{-(s-t)x_u(t)}.
    \end{equation*}
    Integrating from $t$ to $T$ yields
    \begin{equation*}
        x_u\lambda_u \le x_u\int_t^T e^{-(s-t)x_u}\,\mbox{d}s
        = 1-e^{-(T-t)x_u}\le 1.
    \end{equation*}
    Hence
    \begin{equation}\label{eq:sys5:sca:lambdadot}
        \dot\lambda_u=-1+x_u\lambda_u\le 0, \quad t\in[0,T],\, u > x_0.
    \end{equation}
\end{itemize}
\end{subequations}

\begin{proposition}\label{prop:sys5:sca}
    Consider Scalar System~5:
    $\dot x_u(t)=-x_u(t)+u$ and
    $\dot y_u(t)=u-x_u(t)y_u(t)$, with arbitrary $x_0\geq0$ and
    $y_0=y_{\mathrm{ss}}=1$. Then, for every $T>0$,
    $\mathrm{cDR}(u,T)$ is monotone nondecreasing with respect to $u>0$.
\end{proposition}

\begin{proof}
The sensitivity $q_u(t)=\partial_u y_u(t)$ satisfies
$\dot q_u(t)+x_u(t)q_u(t)=g_u(t)$, where
$g_u(t)=1-y_u(t)p_u(t)$.
\begin{enumerate}
    \item For $u\in(0,x_0]$, \eqref{eq:sys5:sca:gcase1} gives
    $g_u(t)\geq0$ on $[0,T]$. Since $\lambda_u(t)\geq0$,
    Theorem~\ref{thm:cDR_monotone}(i) gives
    $\partial_u\mathrm{cDR}(u,T)\geq0$.

    \item For $u\in(x_0,\infty)$, \eqref{eq:sys5:sca:Gcase2} gives
    $G_u(t)\geq0$, while \eqref{eq:sys5:sca:lambdadot} gives
    $\dot\lambda_u(t)\leq0$ on $[0,T]$. Hence,
    Theorem~\ref{thm:cDR_monotone}(ii) gives
    $\partial_u\mathrm{cDR}(u,T)\geq0$.
\end{enumerate}
Thus, $\partial_u\mathrm{cDR}(u,T)\geq0$ for every $u>0$, and
$u\mapsto\mathrm{cDR}(u,T)$ is monotone nondecreasing.
\end{proof}

\begin{figure}[h!]
    \centering
    \subfloat[\label{sys5a}]{\includegraphics[width=0.23\linewidth]{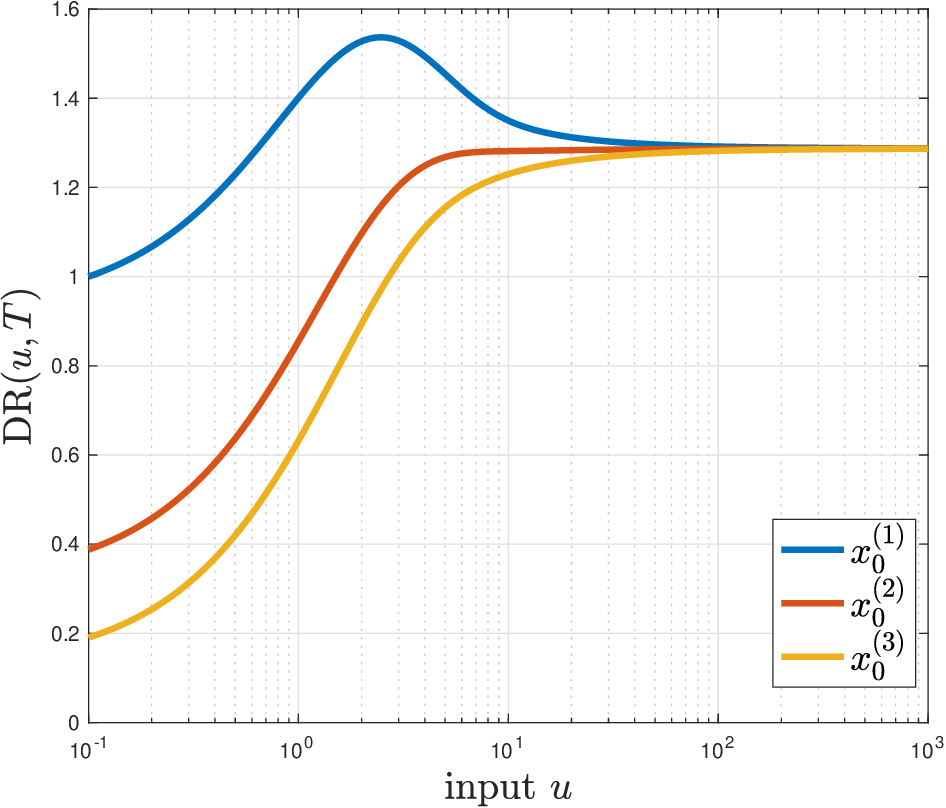}}
    \subfloat[\label{sys5b}]{\includegraphics[width=0.23\linewidth]{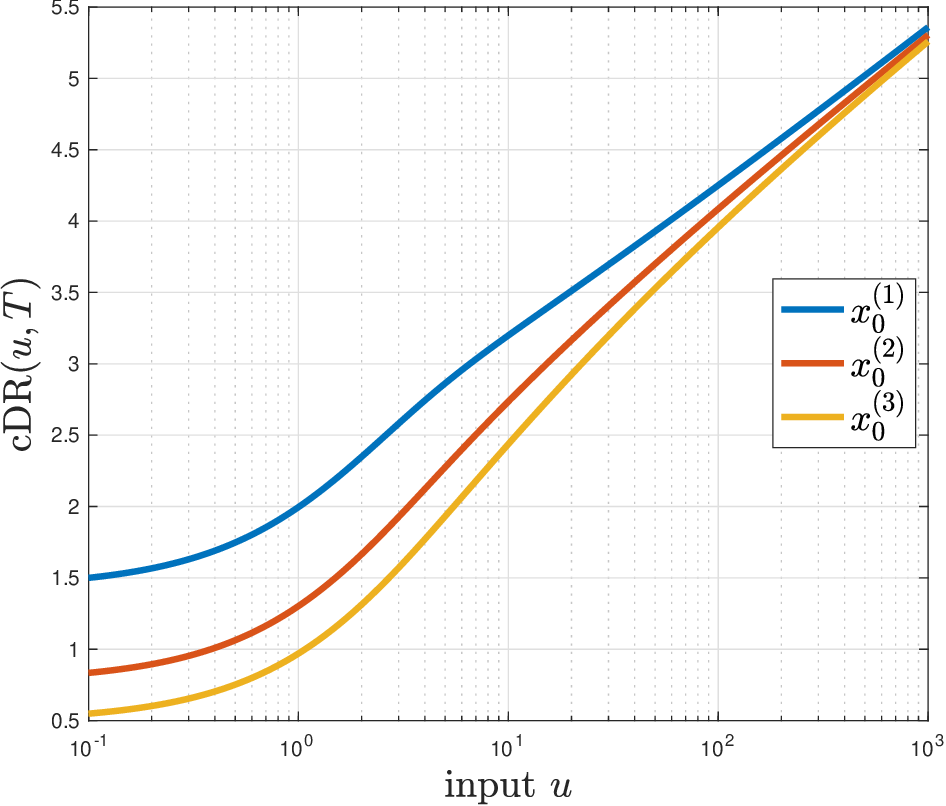}}
    \subfloat[\label{sys5c}]{\includegraphics[width=0.275\linewidth]{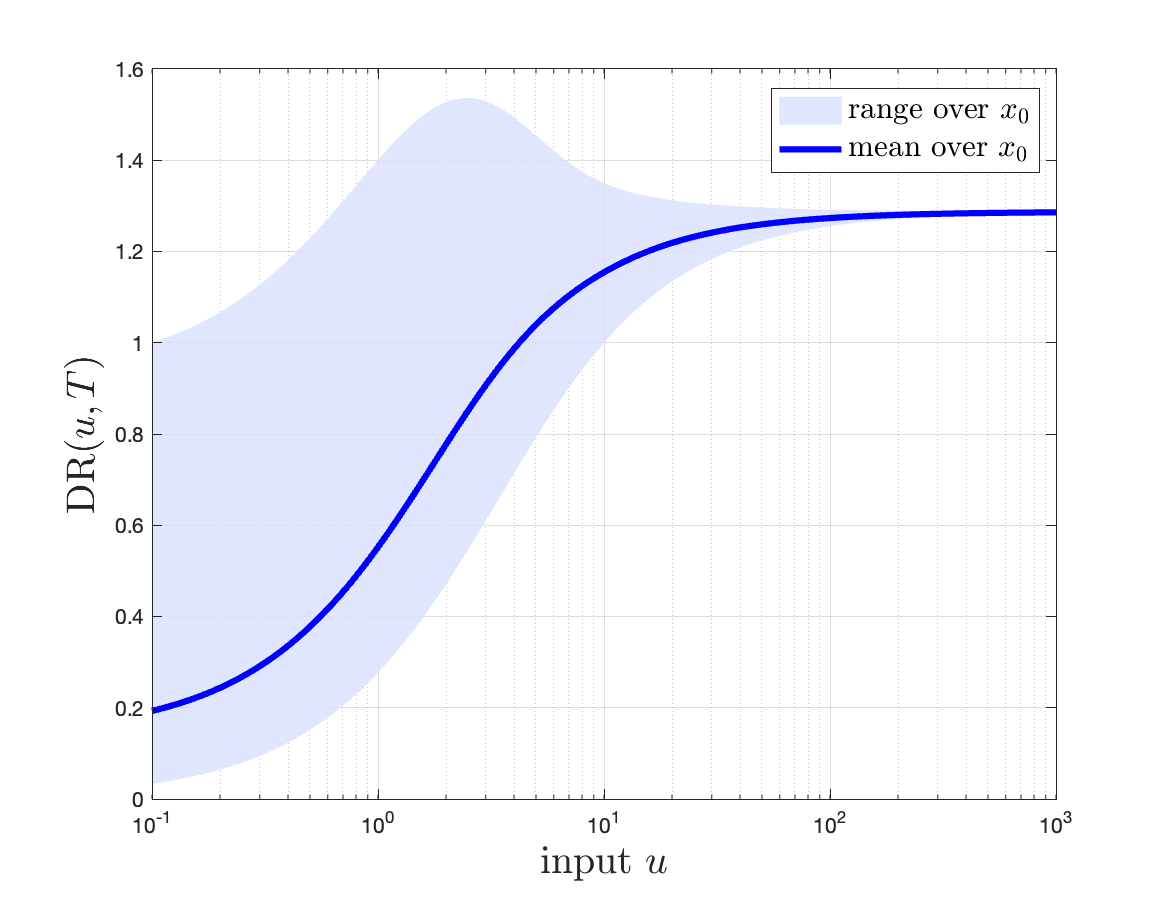}}
    \subfloat[\label{sys5d}]{\includegraphics[width=0.275\linewidth]{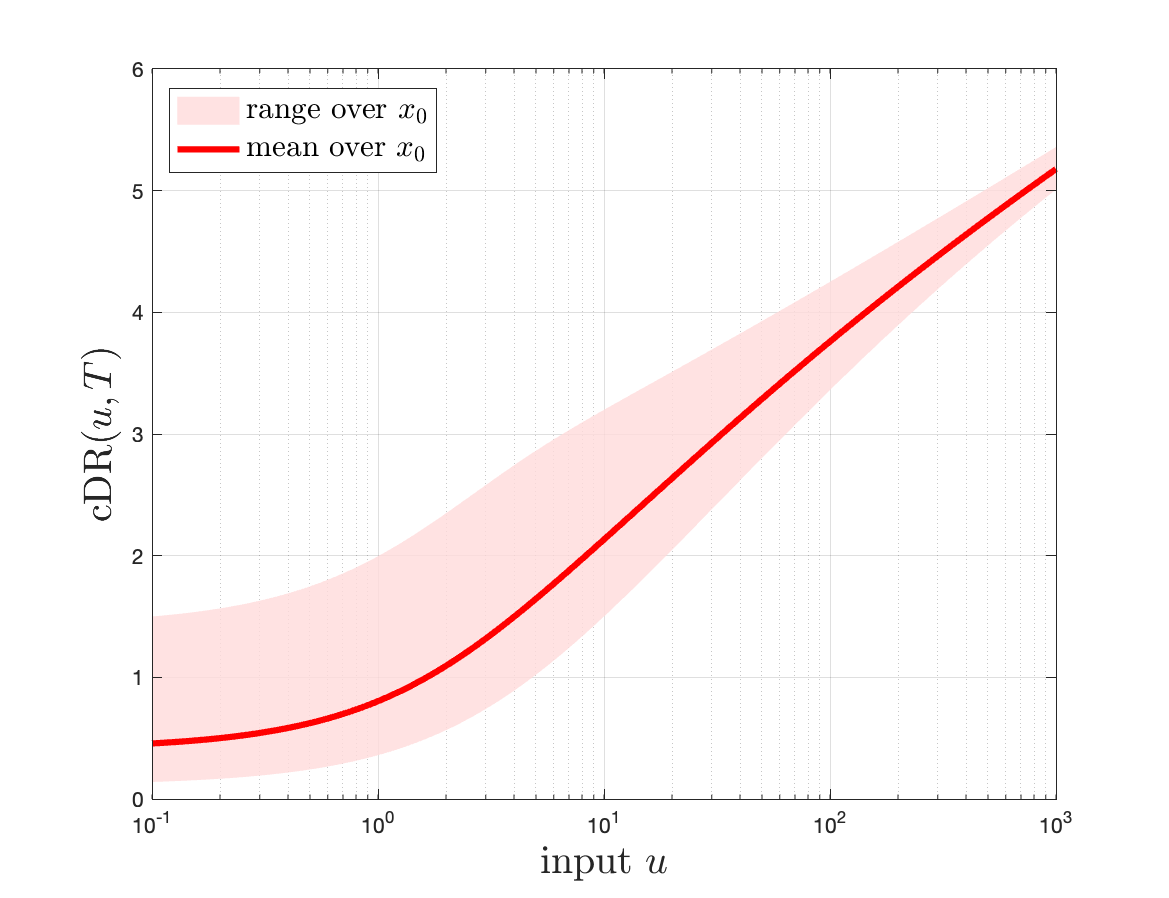}}
    \caption{Scalar System~5: Dose response $\mathrm{DR}(u,T)$ and cumulative dose response $\mathrm{cDR}(u,T)$ for $u\in[10^{-1},10^{3}]$. 
    (a)--(b) Responses for three fixed initial conditions $x_0$. 
    (c)--(d) Envelope (min--max band) and mean over $x_0\in[0.1,10]$. 
    The results show that $\mathrm{cDR}(u,T)$ is monotone nondecreasing with respect to $u$.}
    \label{fig:sys5}
\end{figure}

\begin{subequations}
\paragraph{Model 6: $\dot y_u=1-\dfrac{x_u}{u}y_u$.}
Here $F(x_u,y_u,u)=1-(x_uy_u/u)$. Thus $\partial_xF=-y_u/u$, $\partial_yF=-x_u/u$, and $\partial_uF=x_uy_u/u^2$.
Therefore, $a_u=x_u/u$ and 
\begin{equation}\label{eq:sys6:sca:g}
    g_u=\frac{y_u(x_u-up_u)}{u^2} = \frac{x_0e^{-t}}{u^2}y_u \ge 0,
\end{equation}
$\forall t\ge 0, \forall u > 0$. Hence
\begin{equation}\label{eq:sys6:sca:G}
    G_u=\int_0^t \frac{x_0e^{-s}}{u^2}y_u(s)\,\mbox{d}s \ge 0,
\end{equation}
$\forall t\ge 0, \forall u > 0$. Since $a_u=x_u/u\ge 0$, the kernel $\lambda_u$ defined in Theorem~\ref{thm:cDR_monotone} satisfies
\begin{equation}\label{eq:sys6:sca:lambda}
    \lambda_u
    =
    \int_t^T
    \exp\!\left(
    -\int_t^s\frac{x_u(\tau)}{u}\,\mbox{d}\tau
    \right)\mbox{d}s.
\end{equation}
In particular, $\lambda_u\ge 0$ for all $t\in[0,T]$ and $u > 0$.
\end{subequations}

\begin{proposition}\label{prop:sys6:sca}
    Consider Scalar System~6:
    $\dot x_u(t)=-x_u(t)+u$ and
    $\dot y_u(t)=1-(x_u(t)y_u(t)/u)$, with arbitrary
    $x_0\geq0$ and $y_0\geq0$ chosen independently of $u$. Then, for
    every $T>0$, $\mathrm{cDR}(u,T)$ is monotone nondecreasing with
    respect to $u>0$.
\end{proposition}

\begin{proof}
    The sensitivity $q_u(t)=\partial_u y_u(t)$ satisfies $\dot q_u(t)+(x_u(t)/u)q_u(t)=g_u(t)$ where $g_u(t)$ is given in \eqref{eq:sys6:sca:g}.
    From \eqref{eq:sys6:sca:lambda}, $\lambda_u(t)\ge 0$, $\forall t\in[0,T], \forall u>0$, and $g_u(t)\ge 0$, $\forall t\ge 0, \forall u > 0$.
    Thus $\lambda_u(t)g_u(t)\ge 0, \forall t\in[0,T].$
    By Theorem~\ref{thm:cDR_monotone}(i), $\partial_u\mathrm{cDR}(u,T)\ge 0.$
    Hence $u\mapsto \mathrm{cDR}(u,T)$ is monotone nondecreasing.
\end{proof}

\begin{figure}[h!]
    \centering
    \subfloat[\label{sys6a}]{\includegraphics[width=0.23\linewidth]{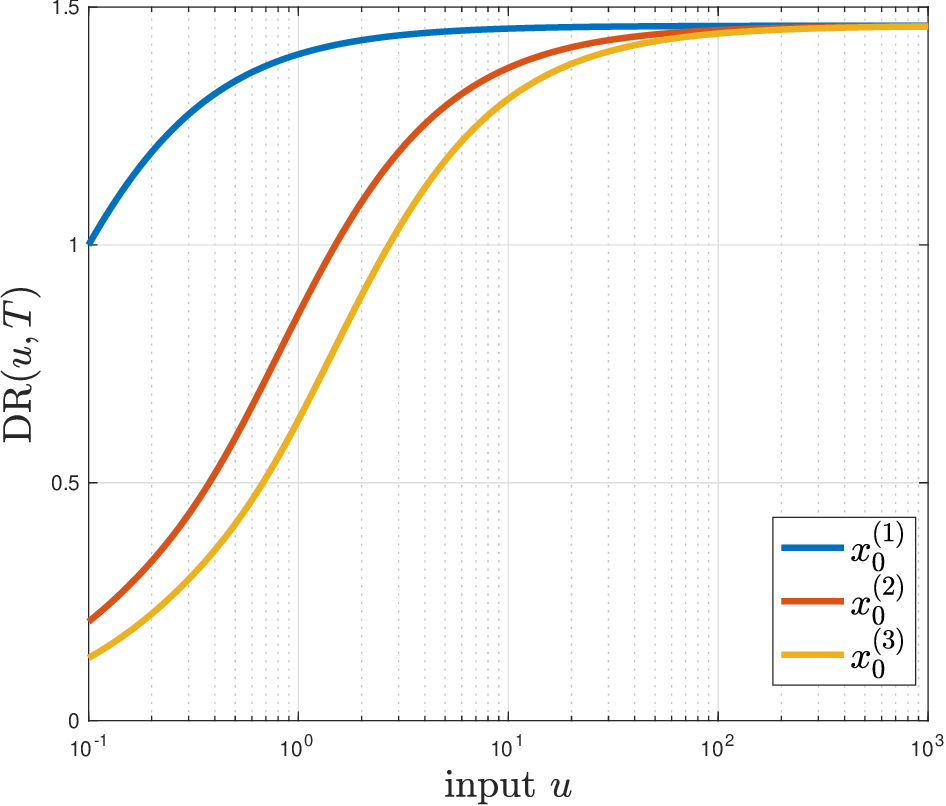}}
    \subfloat[\label{sys6b}]{\includegraphics[width=0.23\linewidth]{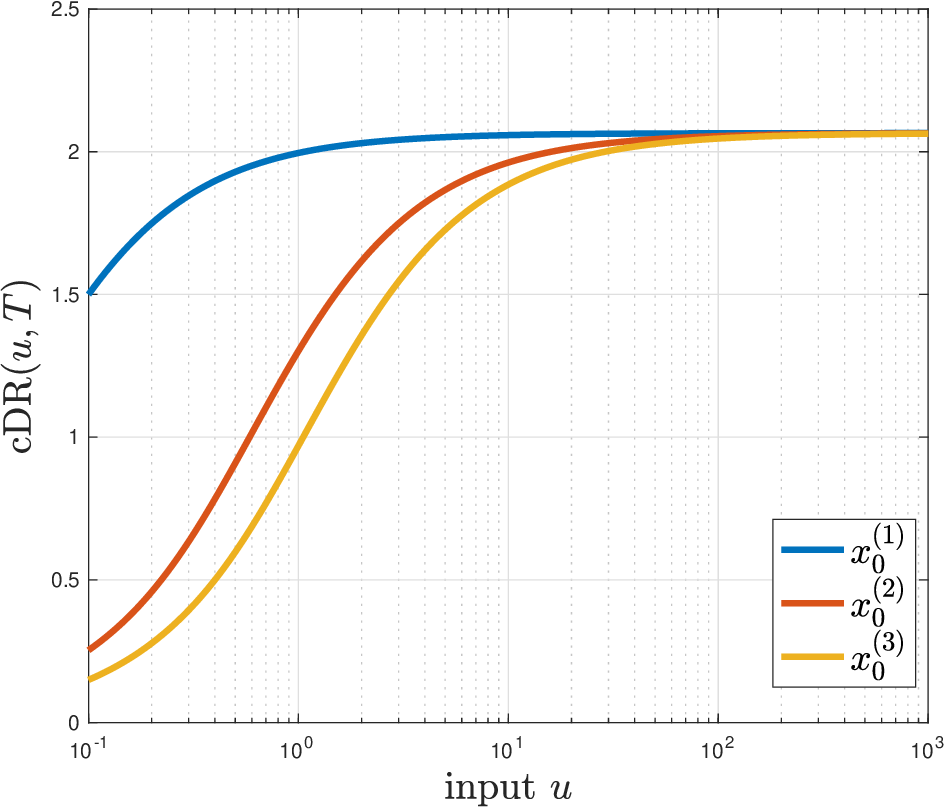}}
    \subfloat[\label{sys6c}]{\includegraphics[width=0.275\linewidth]{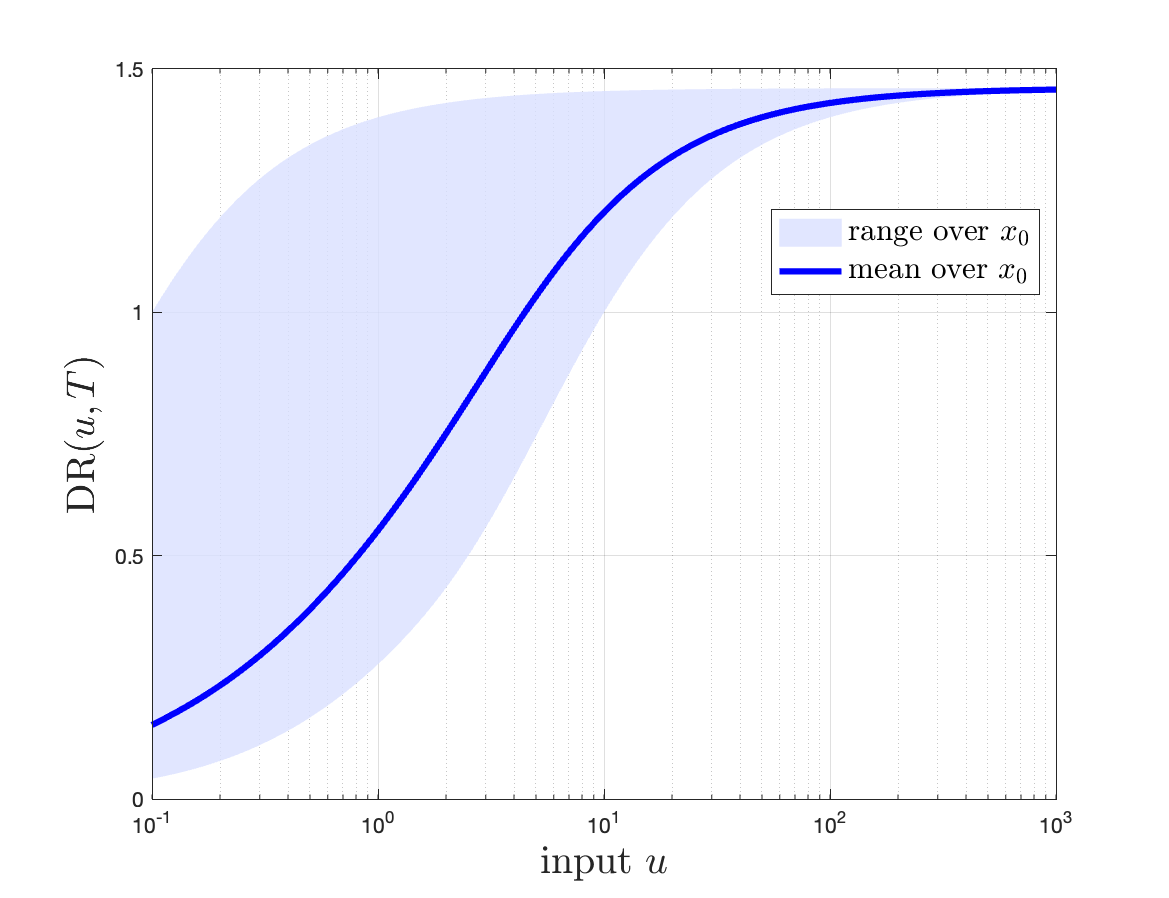}}
    \subfloat[\label{sys6d}]{\includegraphics[width=0.275\linewidth]{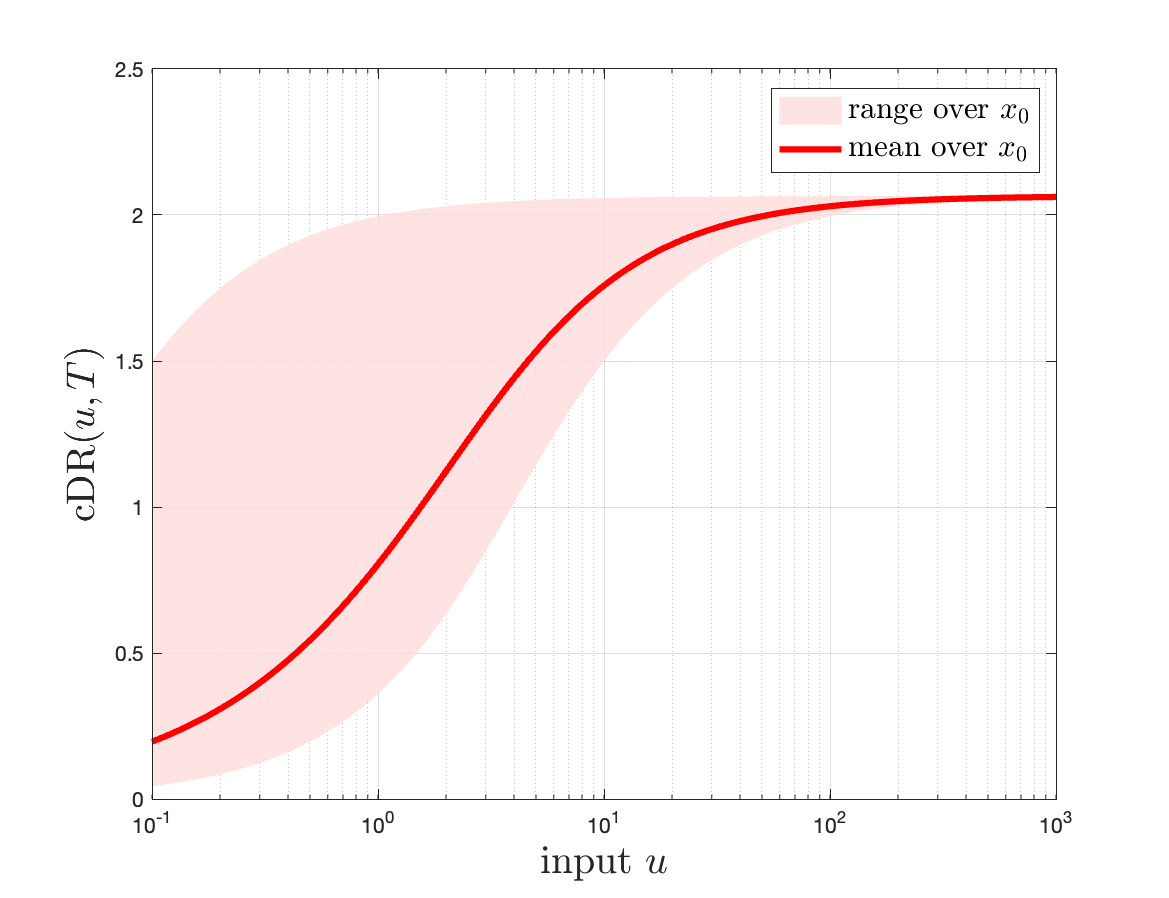}}
    \caption{Scalar System~6: Dose response $\mathrm{DR}(u,T)$ and cumulative dose response $\mathrm{cDR}(u,T)$ for $u\in[10^{-1},10^{3}]$. 
    (a)--(b) Responses for three fixed initial conditions $x_0$. 
    (c)--(d) Envelope (min--max band) and mean over $x_0\in[0.1,10]$. 
    The results show that $\mathrm{cDR}(u,T)$ is monotone nondecreasing with respect to $u$.}
    \label{fig:sys6}
\end{figure}

\begin{subequations}
\paragraph{Model 7: $\dot y_u=\dfrac{1}{u}-\dfrac{1}{x_u}y_u$.}
Here $F(x_u,y_u,u)=(1/u)-(y_u/x_u)$. Thus $\partial_xF=y_u/x_u^2$, $\partial_yF=-1/x_u$, and $\partial_uF=-1/u^2$.
Therefore, $a_u=1/x_u$ and $g_u=(y_up_u/x_u^2)-(1/u^2).$
Using $x_u=up_u+x_0e^{-t}$ and $\dot y_u=(1/u)-(y_u/x_u)$, we rewrite $g_u$ as
\begin{equation}\label{eq:sys7:sca:g}
    g_u=-\frac{1}{u}\dot y_u-\frac{x_0e^{-t}}{u}\frac{y_u}{x_u^2}.
\end{equation}
Integrating \eqref{eq:sys7:sca:g} over $[0,t]$ and using $y_0=1$, we obtain
\begin{equation}\label{eq:sys7:sca:G}
    G_u=\frac{1-y_u}{u}-\frac{x_0}{u}\int_0^t \frac{e^{-s}y_u(s)}{x_u(s)^2}\,\mbox{d}s.
\end{equation}
Since $a_u=1/x_u>0$, the kernel $\lambda_u$ defined in Theorem~\ref{thm:cDR_monotone} satisfies
\begin{equation}\label{eq:sys7:sca:lambda}
    \lambda_u=\int_t^T \exp\!\left(-\int_t^s \frac{1}{x_u(\tau)}\,\mbox{d}\tau\right)\,\mbox{d}s.
\end{equation}
In particular, $\lambda_u\ge 0$ for all $t\in[0,T]$ and $u>0$.

We now analyze two cases.

\begin{itemize}
    \item If $0<u\le x_0$, then $x_u=u+(x_0-u)e^{-t}\ge u$ for all $t\ge 0$. Evaluating the $y$-subsystem at $y_u=1$ gives
    $\dot y_u\big|_{y_u=1}=(1/u)-(1/x_u)\ge 0.$ Since $y_u(0)=1$, the comparison principle yields $y_u\ge 1$ for all $t\ge 0$. By \eqref{eq:sys7:sca:G}, this implies $G_u\le 0,\forall t\in[0,T].$
    Moreover, since $x_u$ is nonincreasing, $a_u=1/x_u$ is nondecreasing. Hence, by the same comparison argument used for $\lambda_u$ in System~5, $\dot\lambda_u\le 0$. Hence $-\dot\lambda_uG_u\le 0,\forall t\in[0,T],$
    and therefore, by Theorem~\ref{thm:cDR_monotone} (ii), $\partial_u\mathrm{cDR}(u,T)<0.$

    \item Next, consider large $u$. By \eqref{eq:int_by_part},
    \begin{equation*}
        \partial_u\mathrm{cDR}(u,T)
        =
        \int_0^T\lambda_u(s)
        \frac{y_u(s)p_u(s)}{x_u(s)^2}\,\mbox{d}s
        -
        \frac{1}{u^2}\int_0^T\lambda_u(s)\,\mbox{d}s.
    \end{equation*}
    Since $x_u(s)=x_0+(u-x_0)(1-e^{-s})$, variation of
    parameters in the $y$-equation gives, uniformly on $[0,T]$,
    \[
        y_u(s)=1+O\!\left(\frac{\log u}{u}\right).
    \]
    Also,
    $\int_0^T x_u(s)^{-1}\,\mbox{d}s
      =O(\log u/u)$; hence, uniformly on $[0,T]$,
    \[
        \lambda_u(s)=T-s+
        O\!\left(\frac{\log u}{u}\right).
    \]
    Using $p_u(s)=1-e^{-s}$ and
    $\int_0^T p_u(s)/x_u(s)^2\,\mbox{d}s
      =O(\log u/u^2)$, we therefore obtain
    \begin{equation*}
        \partial_u\mathrm{cDR}(u,T)
        =
        \int_0^T
        \frac{(T-s)(1-e^{-s})}
             {[x_0+(u-x_0)(1-e^{-s})]^2}
        \,\mbox{d}s
        +O(u^{-2}).
    \end{equation*}
    Setting $z=1-e^{-s}$ and arguing as for
    \eqref{eq:sys4:sca:cDR-large-u} gives
    \[
        \int_0^T
        \frac{(T-s)(1-e^{-s})}
             {[x_0+(u-x_0)(1-e^{-s})]^2}
        \,\mbox{d}s
        =
        \frac{T\log u}{u^2}+O(u^{-2}).
    \]
    Consequently,
    \begin{equation}\label{eq:sys7:sca:cDR-large-u}
        \partial_u\mathrm{cDR}(u,T)
        =
        \frac{T\log u}{u^2}+O(u^{-2})>0
    \end{equation}
    for all sufficiently large $u$.
\end{itemize}
\end{subequations}

\begin{proposition}\label{prop:sys7:sca}
    Consider Scalar System~7:
    $\dot x_u(t)=-x_u(t)+u$ and
    $\dot y_u(t)=\dfrac{1}{u}-\dfrac{y_u(t)}{x_u(t)}$,
    with arbitrary $x_0>0$ and $y_0=y_{\mathrm{ss}}=1$.
    Then, for every $T>0$, there exist $u_-,u_+>0$ such that
    \begin{equation*}
        \partial_u\mathrm{cDR}(u_-,T)<0,
        \qquad
        \partial_u\mathrm{cDR}(u_+,T)>0.
    \end{equation*}
    In particular, $\mathrm{cDR}(u,T)$ is not monotone with respect to $u$.
\end{proposition}

\begin{proof}
The sensitivity $q_u(t)=\partial_u y_u(t)$ satisfies
$\dot q_u(t)+\dfrac{1}{x_u(t)}q_u(t)=g_u(t)$, where $g_u(t)$
is given by \eqref{eq:sys7:sca:g}.

\begin{enumerate}
    \item If $0<u\leq x_0$, then $x_u(t)\geq u$ and comparison gives
    $y_u(t)\geq1$ for all $t\geq0$. Since $x_0>0$ and $y_u(t)>0$,
    \eqref{eq:sys7:sca:G} gives $G_u(t)<0$ for every $t>0$.
    Moreover, $x_u$ is nonincreasing, so $a_u=1/x_u$ is
    nondecreasing and $\dot\lambda_u<0$ on $[0,T]$.
    Hence, \eqref{eq:int_by_part_G} gives
    $\partial_u\mathrm{cDR}(u,T)<0$.

    \item For sufficiently large $u$,
    \eqref{eq:sys7:sca:cDR-large-u} gives
    $\partial_u\mathrm{cDR}(u,T)>0$.
\end{enumerate}

Therefore, choosing any $u_-\in(0,x_0]$ and a sufficiently large
$u_+$ gives
\[
    \partial_u\mathrm{cDR}(u_-,T)<0,
    \qquad
    \partial_u\mathrm{cDR}(u_+,T)>0.
\]
Hence, $\mathrm{cDR}(u,T)$ is not monotone with respect to $u$.
\end{proof}

\begin{figure}[h!]
    \centering
    \subfloat[\label{sys7a}]{\includegraphics[width=0.23\linewidth]{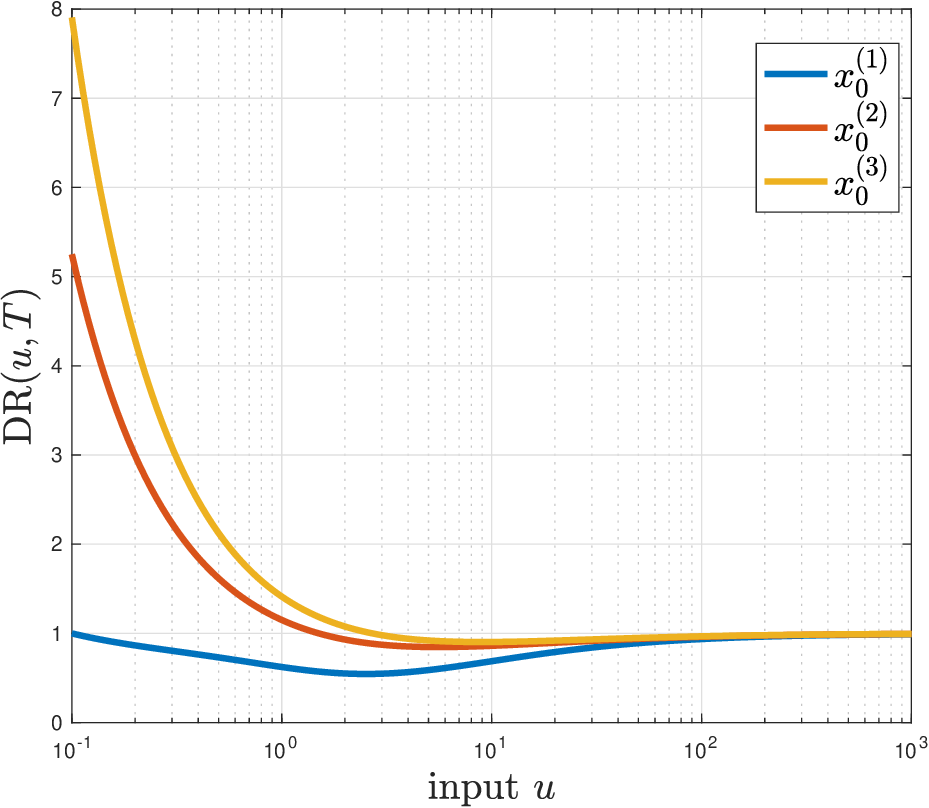}}
    \subfloat[\label{sys7b}]{\includegraphics[width=0.23\linewidth]{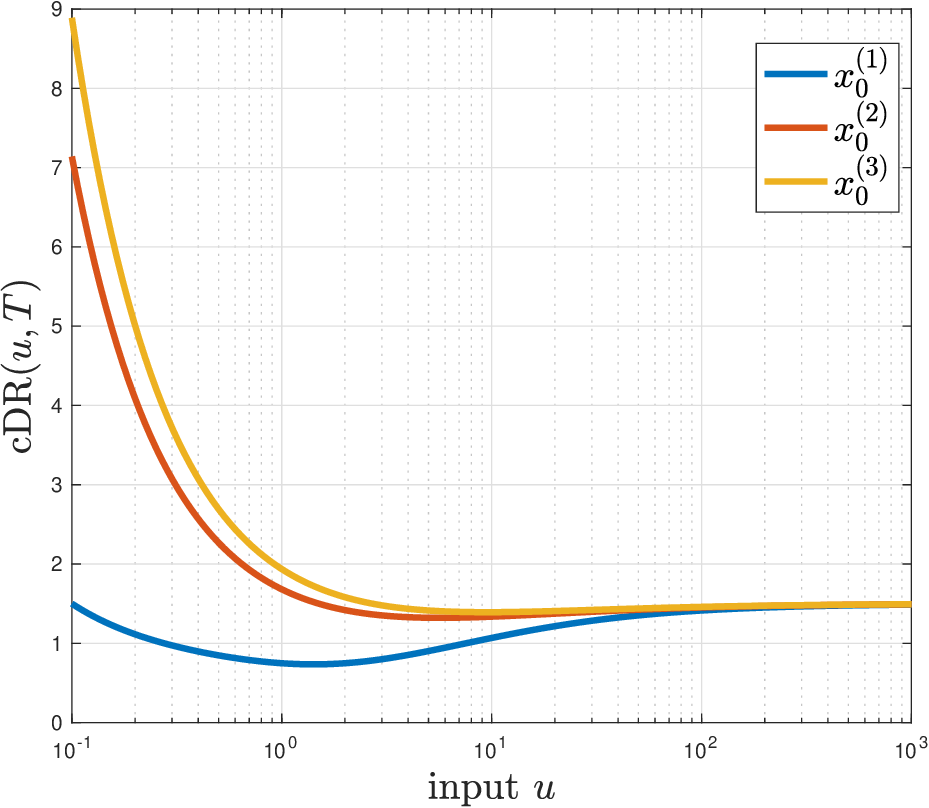}}
    \subfloat[\label{sys7c}]{\includegraphics[width=0.275\linewidth]{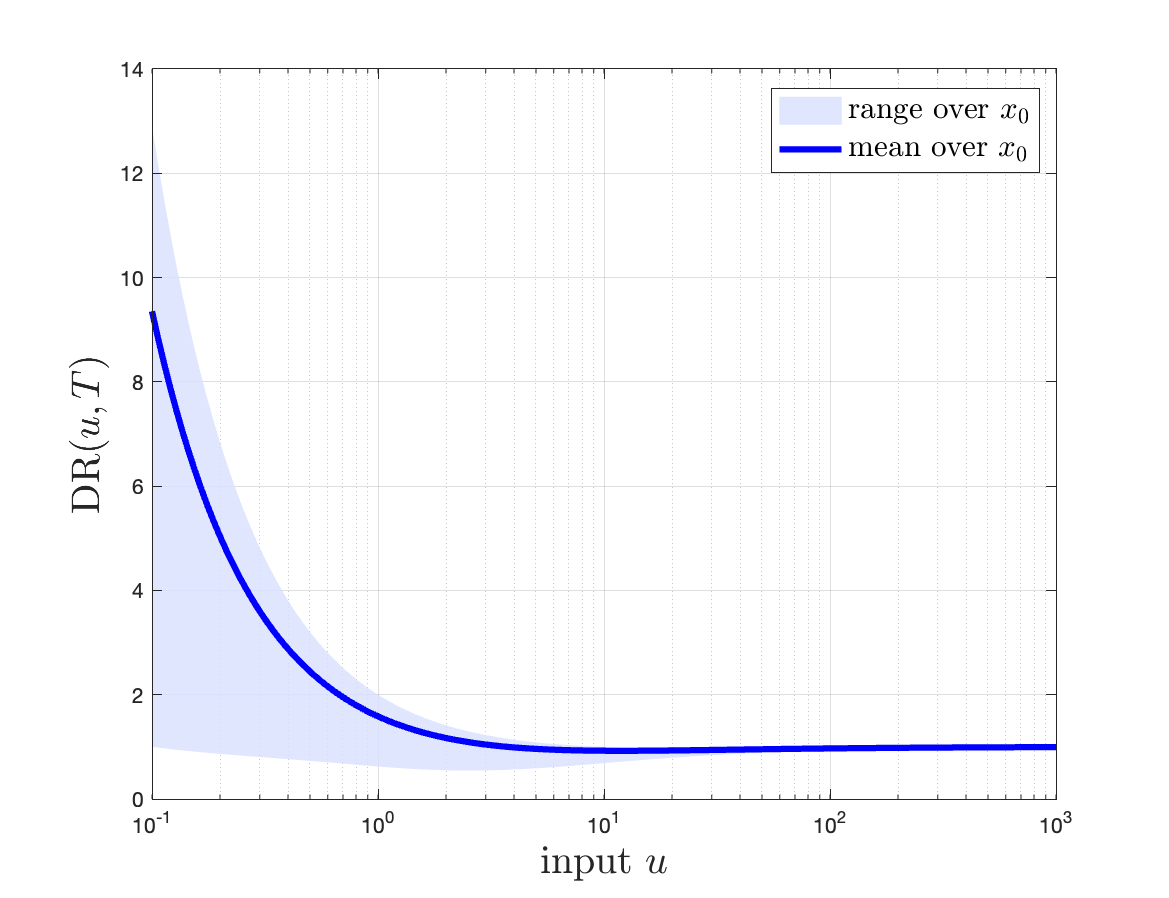}}
    \subfloat[\label{sys7d}]{\includegraphics[width=0.275\linewidth]{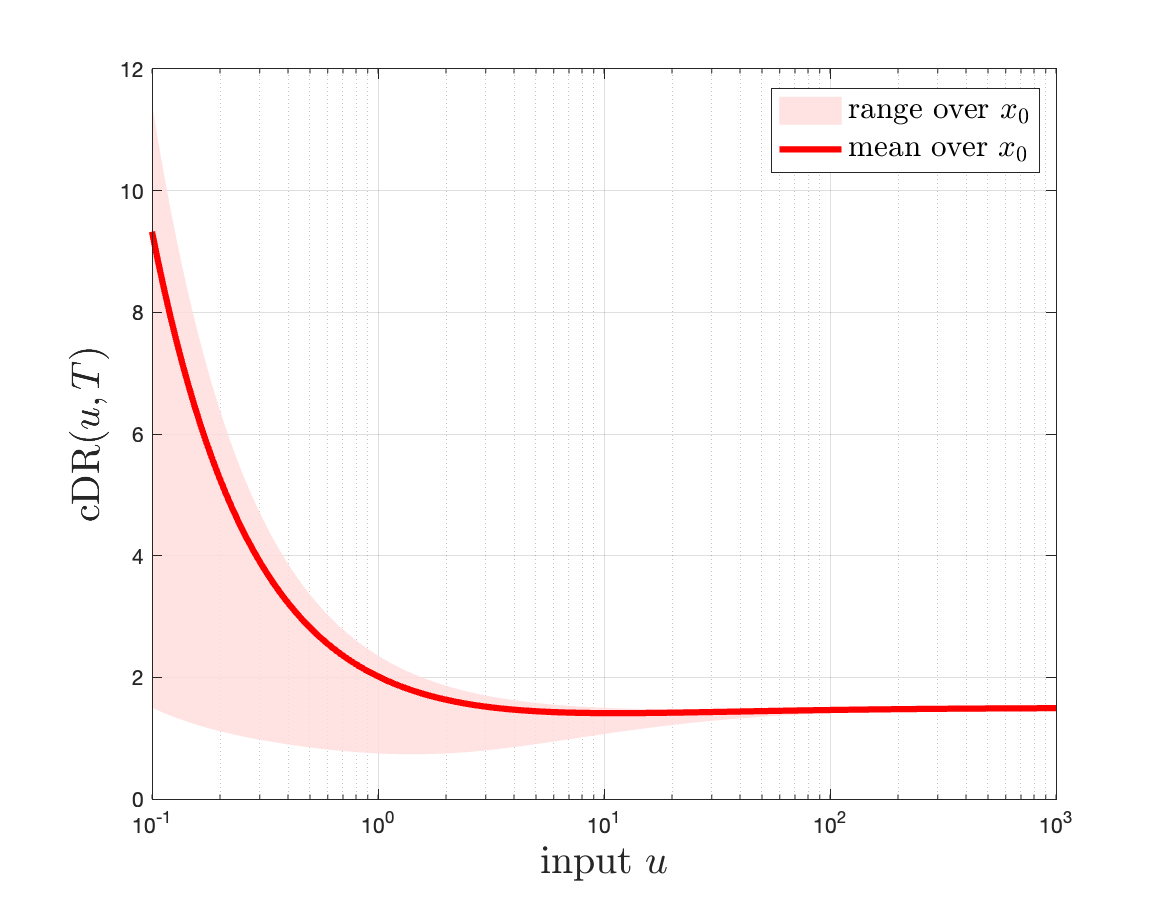}}
    \caption{Scalar System~7: Dose response $\mathrm{DR}(u,T)$ and cumulative dose response $\mathrm{cDR}(u,T)$ for $u\in[10^{-1},10^{3}]$. 
    (a)--(b) Responses for three fixed initial conditions $x_0$. 
    (c)--(d) Envelope (min--max band) and mean over $x_0\in[0.1,10]$. 
    The results show that $\mathrm{cDR}(u,T)$ is not monotone with respect to $u$, exhibiting a change in monotonicity.}
    \label{fig:sys7}
\end{figure}

\begin{subequations}
\paragraph{Model 8: $\dot y_u=1-\dfrac{u}{x_u}y_u$.}
Here $F(x_u,y_u,u)=1-(uy_u/x_u)$. Thus $\partial_xF=uy_u/x_u^2$, $\partial_yF=-u/x_u$, and $\partial_uF=-y_u/x_u$.
Therefore, $a_u=u/x_u$ and 
\begin{equation}\label{eq:sys8:sca:g}
    g_u=\frac{uy_u}{x_u^2}p_u-\frac{y_u}{x_u}
    =\frac{y_u(up_u-x_u)}{x_u^2} = -\frac{x_0e^{-t}}{x_u^2}y_u \le 0,
\end{equation}
$\forall t\ge 0, \forall u > 0$. Hence
\begin{equation}\label{eq:sys8:sca:G}
    G_u=-\int_0^t \frac{x_0e^{-s}}{x_u(s)^2}y_u(s)\,\mbox{d}s \le 0,
\end{equation}
$\forall t\ge 0, \forall u > 0$. Since $a_u=u/x_u\ge 0$, the kernel $\lambda_u$ defined in Theorem~\ref{thm:cDR_monotone} satisfies
\begin{equation}\label{eq:sys8:sca:lambda}
    \lambda_u=\int_t^T \exp\!\left(-\int_t^s \frac{u}{x_u(\tau)}\,\mbox{d}\tau\right)\,\mbox{d}s.
\end{equation}
In particular, $\lambda_u\ge 0$ for all $t\in[0,T]$ and $u>0$.
\end{subequations}

\begin{proposition}\label{prop:sys8:sca}
    Consider Scalar System~8:
    $\dot x_u(t)=-x_u(t)+u$ and
    $\dot y_u(t)=1-(uy_u(t)/x_u(t))$, with arbitrary
    $x_0>0$ and $y_0\geq0$ chosen independently of $u$. Then, for every
    $T>0$, $\mathrm{cDR}(u,T)$ is monotone nonincreasing with respect
    to $u>0$.
\end{proposition}

\begin{proof}
    The sensitivity $q_u(t)=\partial_u y_u(t)$ satisfies $\dot q_u(t)+(u/x_u(t))q_u(t)=g_u(t)$ where $g_u(t)$ is given in \eqref{eq:sys8:sca:g}.
    From \eqref{eq:sys8:sca:lambda}, $\lambda_u(t)\ge 0$, $\forall t\in[0,T], \forall u>0$, and $g_u(t)\le 0$, $\forall t\ge 0, \forall u > 0$.
    Thus $\lambda_u(t)g_u(t)\le 0, \forall t\in[0,T].$
    By Theorem~\ref{thm:cDR_monotone} (i), $\partial_u\mathrm{cDR}(u,T)\le 0.$
    Hence $u\mapsto \mathrm{cDR}(u,T)$ is monotone nonincreasing.
\end{proof}

\begin{figure}[h!]
    \centering
    \subfloat[\label{sys8a}]{\includegraphics[width=0.23\linewidth]{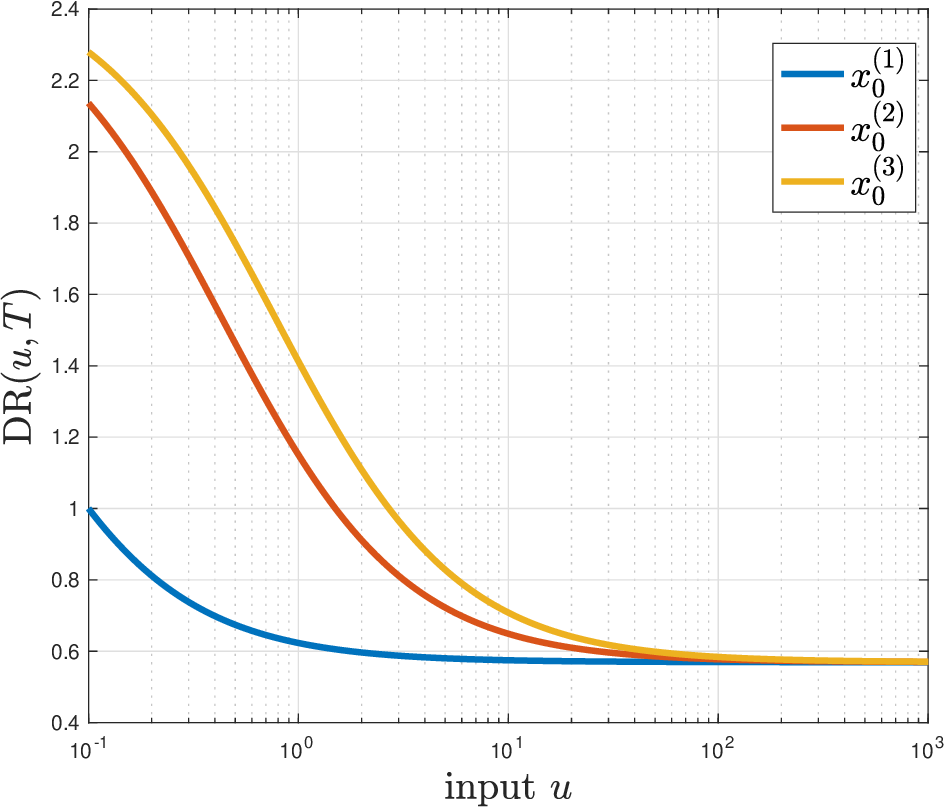}}
    \subfloat[\label{sys8b}]{\includegraphics[width=0.23\linewidth]{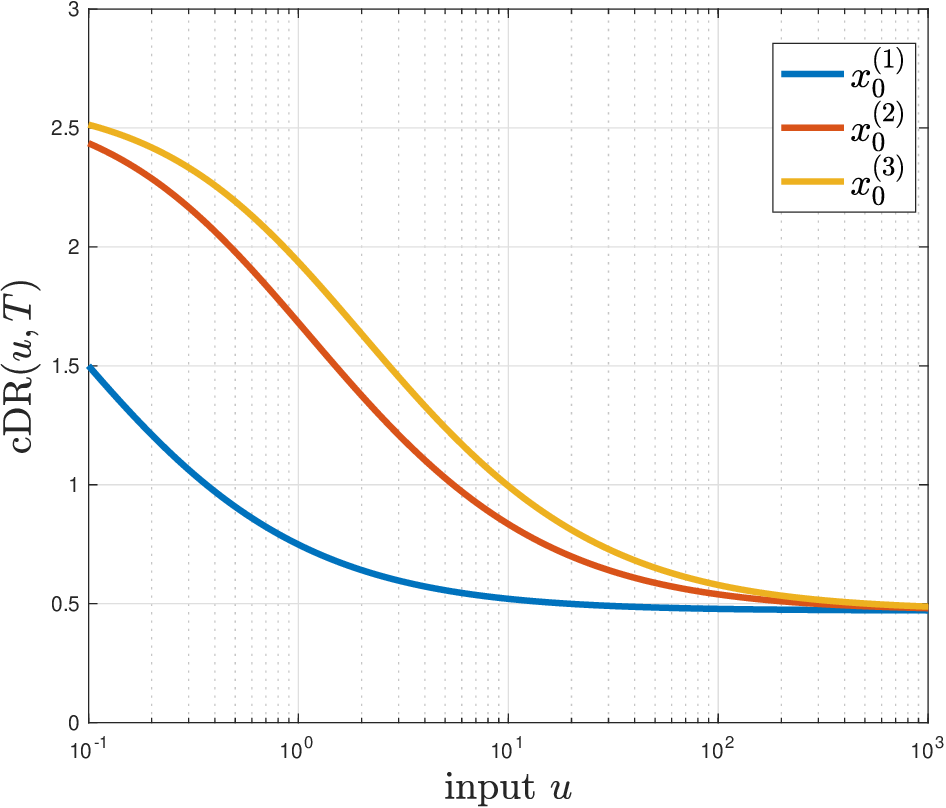}}
    \subfloat[\label{sys8c}]{\includegraphics[width=0.275\linewidth]{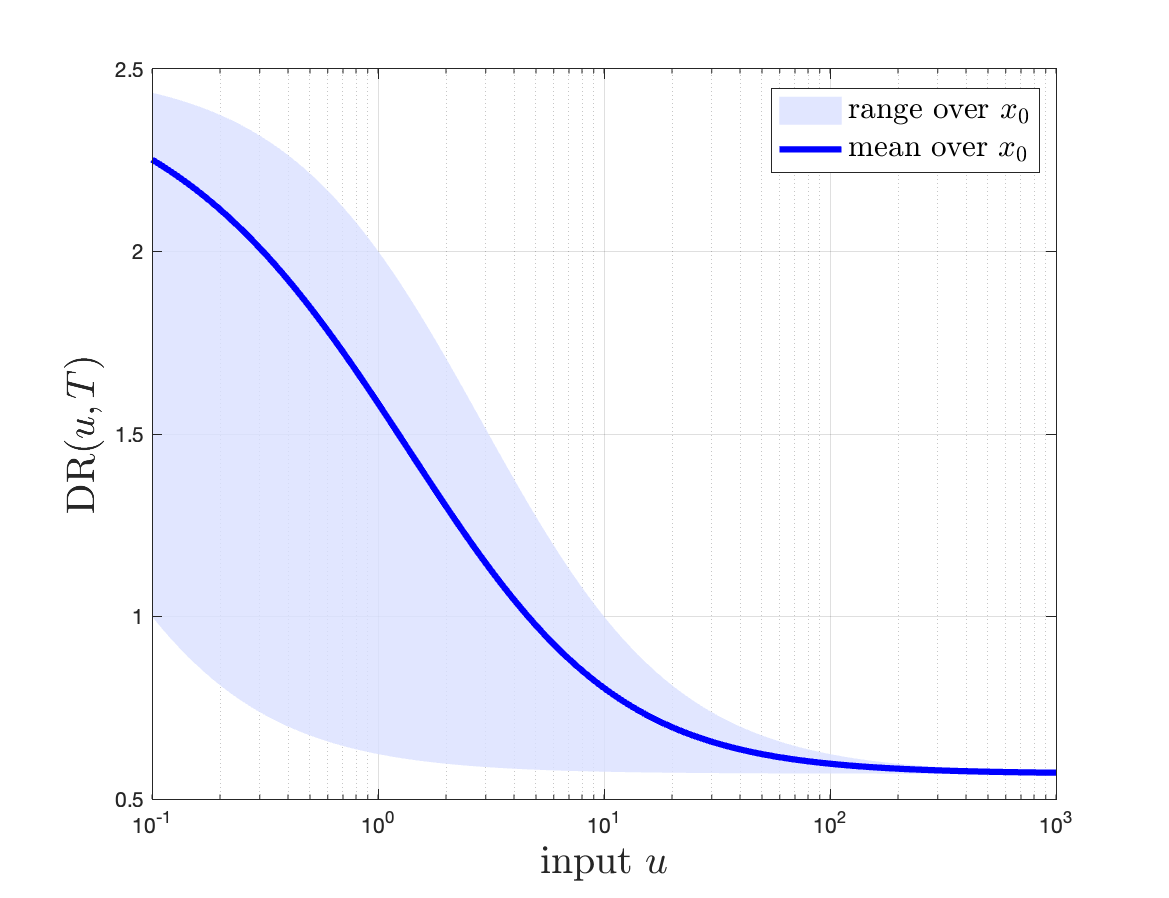}}
    \subfloat[\label{sys8d}]{\includegraphics[width=0.275\linewidth]{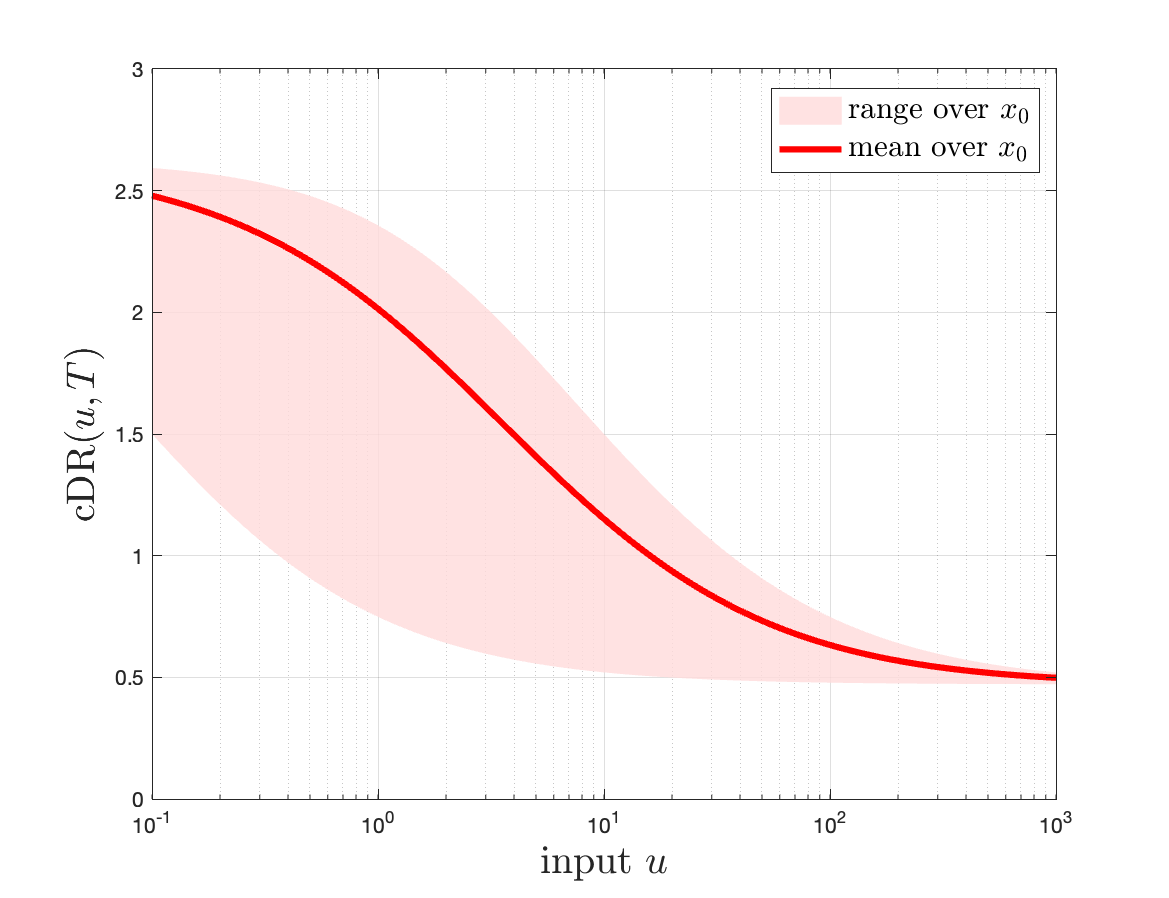}}
    \caption{Scalar System~8: Dose response $\mathrm{DR}(u,T)$ and cumulative dose response $\mathrm{cDR}(u,T)$ for $u\in[10^{-1},10^{3}]$. 
    (a)--(b) Responses for three fixed initial conditions $x_0$. 
    (c)--(d) Envelope (min--max band) and mean over $x_0\in[0.1,10]$. 
    The results show that $\mathrm{cDR}(u,T)$ is monotone nonincreasing with respect to $u$.}
    \label{fig:sys8}
\end{figure}

\section{Monotonicity of $\mathrm{cDR}(u,T)$ for the generalized IFFM systems}
\label{Apx:vector}

In this appendix, we analyze the systems obtained by combining the
intermediate dynamics in \eqref{eq:x_general} with output models
$1$, $3$, $6$, $7$, and $8$ in Table~\ref{tab:systems}.
For each system, we derive
\[
    a_u(t),\qquad g_u(t),\qquad
    G_u(t):=\int_0^t g_u(s)\,\mbox{d}s,
\]
and apply Theorems~\ref{thm:cDR_monotone}
and~\ref{thm:cDR_nonmonotone} to determine the monotonicity of
$\mathrm{cDR}(u,T)$.

For all systems, the intermediate dynamics are $\dot x_u=Ax_u+bu,$ with initial condition $x_u(0)=x_0$, independent of $u$. 
Differentiating the $x$-subsystem with respect to $u$ gives $\dot p_u = Ap_u + b$ with $p_u(0)=0$
where $p_u:=\partial_u x_u\in\mathbb{R}^n$. Hence,
\begin{equation}\label{eq:p_vec_explicit}
    p_u(t)=\int_0^t e^{A(t-s)}\,b\,\mbox{d}s
          = A^{-1}(e^{At}-I)b.
\end{equation}
Since $A$ is Metzler, $e^{At}\ge 0$ for all $t\ge 0$, and since $b\ge 0$, it follows that
$p_u(t)\ge 0, \forall t\ge 0.$ Moreover, by variation of parameters formula,
\begin{equation}\label{eq:x_vec_explicit}
    x_u(t)=e^{At}x_0+\int_0^t e^{A(t-s)}bu\,\mbox{d}s
    = e^{At}x_0 + up_u(t).
\end{equation}

Let $q_u:=\partial_u y_u$. If $\dot y_u = F(x_u,y_u,u),$ then differentiating with respect to $u$ yields
\begin{equation}\label{eq:q_vec_general}
    \dot q_u = \partial_x F(x_u,y_u,u)\,p_u
    +\partial_y F(x_u,y_u,u)\,q_u +\partial_u F(x_u,y_u,u).
\end{equation}
Rewriting \eqref{eq:q_vec_general} in the form $\dot q_u+a_u(t)q_u=g_u(t),$ we identify
\begin{equation*}
    a_u(t)=-\partial_yF(x_u,y_u,u), \qquad
    g_u(t)=\partial_xF(x_u,y_u,u)\,p_u+\partial_uF(x_u,y_u,u).
\end{equation*}

\begin{subequations}
\paragraph{System~1:} 
Consider
\begin{equation}\label{eq:sys1-vector}
    \dot x_u = A x_u + bu, \qquad
    \dot y_u = \frac{c^\top x_u}{\beta u} - d y_u,
\end{equation}
where $A\in\mathbb{R}^{n\times n}$ is Metzler and Hurwitz, $b,c\in\mathbb{R}^n_+$, $\gamma = 1$, and $d,\beta\in\mathbb{R}_+$.
We assume that the initial condition $x_u(0)=x_0$ is independent of $u$, with $x_0\ge 0$, and that $y_u(0)=y_0$ is also independent of $u$.

Let $q_u:=\partial_u y_u\in\mathbb{R}$. Differentiating the $y$-equation in \eqref{eq:sys1-vector} with respect to $u$ gives
\begin{equation*}
    \dot q_u = \frac{c^\top p_u}{\beta u}-\frac{c^\top x_u}{\beta u^2}-d q_u,
    \qquad q_u(0)=0.
\end{equation*}
Therefore, $q_u$ satisfies $\dot q_u + a_u q_u = g_u$, $q_u(0)=0$ with
\begin{equation*}
    a_u=d, \qquad
    g_u=\frac{c^\top p_u}{\beta u}-\frac{c^\top x_u}{\beta u^2}.
\end{equation*}
Using \eqref{eq:x_vec_explicit}, we rewrite $g_u$ as
\begin{equation}\label{eq:sys1:vec:g}
    g_u
    = \frac{uc^\top p_u-c^\top x_u}{\beta u^2}
    = \frac{uc^\top p_u-c^\top\!\bigl[e^{At}x_0+up_u\bigr]}{\beta u^2}
    = -\frac{c^\top e^{At}x_0}{\beta u^2}.
\end{equation}
Since $e^{At}\ge 0$, $c\in\mathbb{R}^n_+$, and $x_0\ge 0$, we conclude that $g_u(t)\le 0, \forall t\ge 0,\forall u > 0.$
Integrating $g_u$ over $[0,t]$, we obtain
\begin{equation}\label{eq:sys1:vec:G}
    G_u:=\int_0^t g_u(s)\,\mbox{d}s = -\frac{1}{\beta u^2}\int_0^t c^\top e^{As}x_0\,\mbox{d}s \le 0, \quad \forall t\ge 0,\forall u > 0.
\end{equation}
Since $a_u=d$ is constant, the kernel $\lambda_u$ defined in Theorem~\ref{thm:cDR_monotone} is
\begin{equation}\label{eq:sys1:vec:lambda}
    \lambda_u = \int_t^T e^{-d(s-t)}\,\mbox{d}s = \frac{1-e^{-d(T-t)}}{d}.
\end{equation}
Therefore, $\lambda_u\ge 0$ for all $t\in[0,T]$ and $u>0$.
\end{subequations}

\begin{proposition}
    Consider the system \eqref{eq:sys1-vector}, where $A$ is Metzler
    and Hurwitz, $b,c\in\mathbb{R}^n_+$, $\gamma=1$, $\beta,d\in\mathbb{R}_+$,
    and $x_0\geq0$ and $y_0\geq0$ are independent of $u$.
    Then, for every $T>0$, $\mathrm{cDR}(u,T)$ is monotone
    nonincreasing with respect to $u>0$.
\end{proposition}

\begin{proof}
    The sensitivity $q_u(t)=\partial_u y_u(t)$ satisfies $\dot q_u(t) + d q_u(t) = g_u(t)$ where $g_u(t)$ is given in \eqref{eq:sys1:vec:g}. From \eqref{eq:sys1:vec:lambda}, $\lambda_u\ge 0, \forall t\in[0,T], \forall u > 0$ and $g_u(t)\le 0$, $\forall t\ge 0, \forall u > 0$. Thus $\lambda_u(t)g_u(t)\le 0, \forall t\in[0,T]$. By Theorem~\ref{thm:cDR_monotone}(i), $\partial_u \mathrm{cDR}(u,T) \le 0.$
    Hence $u\mapsto \mathrm{cDR}(u,T)$ is monotone nonincreasing.
\end{proof}

\begin{figure}[h!]
    \centering
    \subfloat[\label{v1a}]{\includegraphics[width=0.3\linewidth]{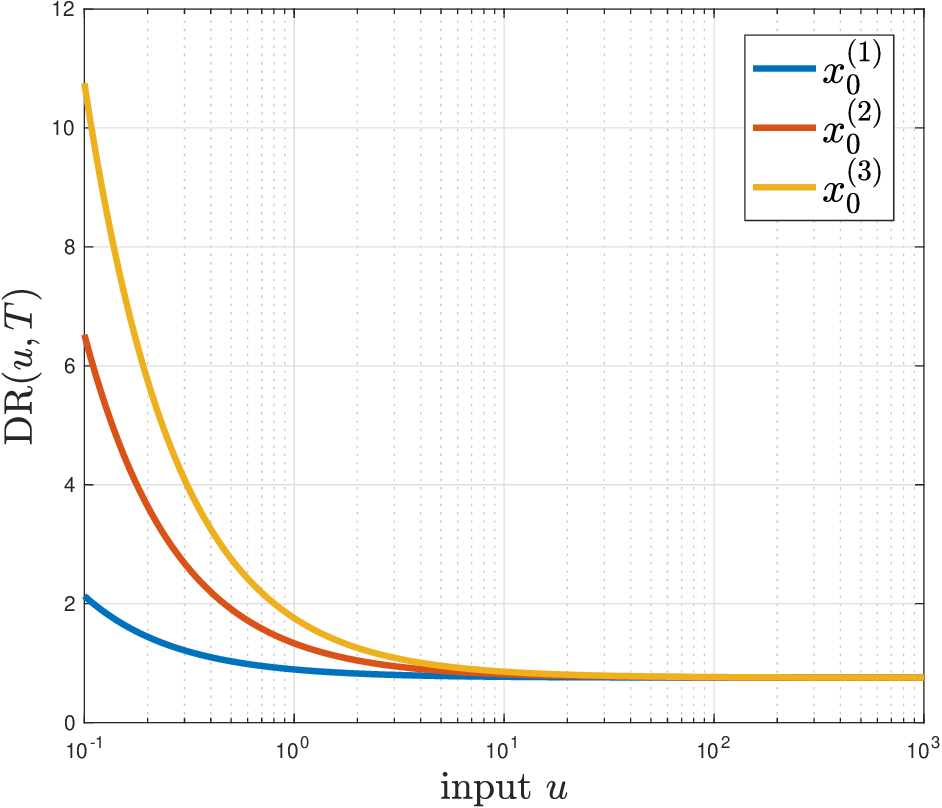}}\qquad
    \subfloat[\label{v1b}]{\includegraphics[width=0.3\linewidth]{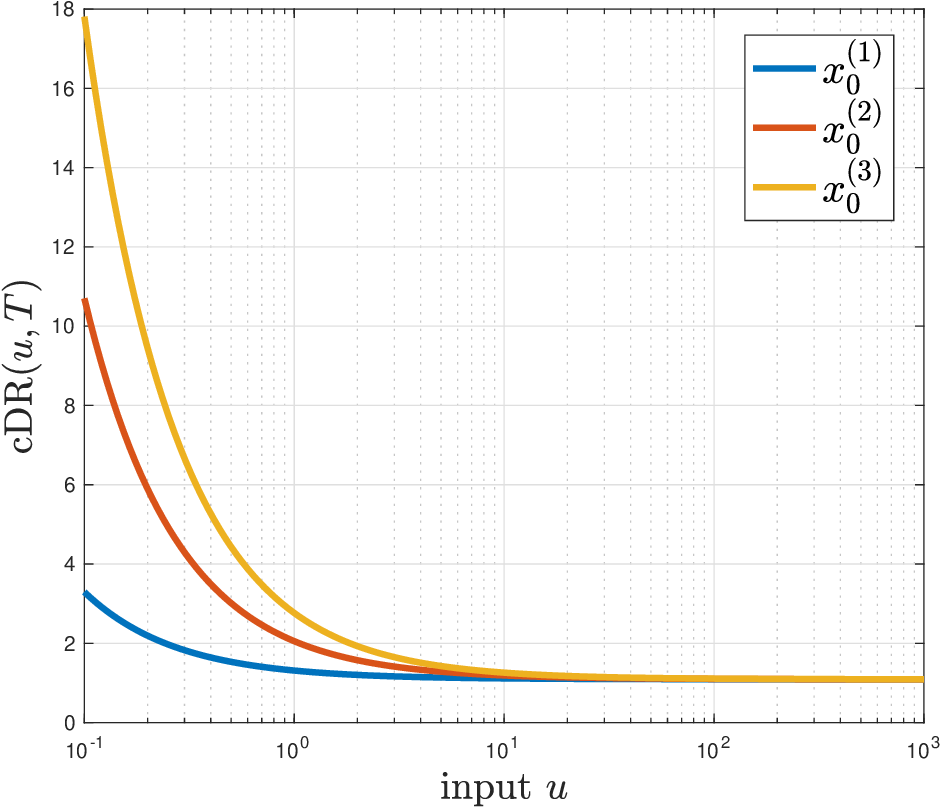}}
    \caption{System~1: Dose response $\mathrm{DR}(u,T)$ and cumulative dose response $\mathrm{cDR}(u,T)$ for $u\in[10^{-1},10^{3}]$. 
    Results are shown for three fixed initial conditions $x_0$. 
    The curves illustrate that $\mathrm{cDR}(u,T)$ is monotone nonincreasing with respect to $u$.}
    \label{fig:v_sys1}
\end{figure}

\begin{subequations}
\paragraph{System~3:} 
Consider
\begin{equation}\label{eq:sys3-vector}
    \dot x_u = A x_u + bu, \qquad
    \dot y_u = \frac{\beta u}{K+c^\top x_u} - d y_u,
\end{equation}
where $A\in\mathbb{R}^{n\times n}$ is Metzler and Hurwitz, $b,c\in\mathbb{R}^n_+$, $\gamma = 1$, and $d,K,\beta\in\mathbb{R}_+$.
We assume that the initial condition $x_u(0)=x_0$ is independent of $u$, with $x_0\ge 0$, and that $y_u(0)=y_0$ is also independent of $u$.

Let $q_u:=\partial_u y_u\in\mathbb{R}$. Differentiating the $y$-equation in \eqref{eq:sys3-vector} with respect to $u$ gives
\begin{equation*}
    \dot q_u = \frac{\beta}{K+c^\top x_u}
    -\frac{\beta uc^\top p_u}{(K+c^\top x_u)^2}
    -d q_u,
    \qquad q_u(0)=0.
\end{equation*}
Therefore, $q_u$ satisfies $\dot q_u + a_u q_u = g_u$, $q_u(0)=0$ with
\begin{equation*}
    a_u=d, \qquad
    g_u=\frac{\beta}{K+c^\top x_u}
    -\frac{\beta uc^\top p_u}{(K+c^\top x_u)^2}.
\end{equation*}
Using \eqref{eq:x_vec_explicit}, we rewrite $g_u$ as
\begin{equation}\label{eq:sys3:vec:g}
    g_u
    = \frac{\beta\bigl[K+c^\top x_u-uc^\top p_u\bigr]}{(K+c^\top x_u)^2}
    = \frac{\beta\bigl[K+c^\top e^{At}x_0\bigr]}{(K+c^\top x_u)^2}.
\end{equation}
Since $K>0$, $e^{At}\ge 0$, $c\in\mathbb{R}^n_+$, and $x_0\ge 0$, we conclude that $g_u\ge 0, \forall t\ge 0, \forall u>0.$
Integrating $g_u$ over $[0,t]$, we obtain
\begin{equation}\label{eq:sys3:vec:G}
    G_u:=\int_0^t g_u(s)\,\mbox{d}s
    =\beta\int_0^t
    \frac{K+c^\top e^{As}x_0}{(K+c^\top x_u(s))^2}\,\mbox{d}s \geq0.
    \quad \forall t\ge 0,\forall u > 0.
\end{equation}
Since $a_u=d$ is constant, the kernel $\lambda_u$ defined in Theorem~\ref{thm:cDR_monotone} is
\begin{equation}\label{eq:sys3:vec:lambda}
    \lambda_u = \int_t^T e^{-d(s-t)}\,\mbox{d}s = \frac{1-e^{-d(T-t)}}{d}.
\end{equation}
Therefore, $\lambda_u\ge 0$ for all $t\in[0,T]$ and $u>0$.
\end{subequations}

\begin{proposition}\label{prop:sys3:vec}
    Consider the system \eqref{eq:sys3-vector}, where $A$ is Metzler
    and Hurwitz, $b,c\in\mathbb{R}^n_+$, $\gamma=1$, and
    $d,K,\beta\in\mathbb{R}_+$. Assume that $x_0\geq0$ and
    $y_0\geq0$ are independent of $u$. Then, for every $T>0$,
    $\mathrm{cDR}(u,T)$ is monotone nondecreasing with respect to $u>0$.
\end{proposition}

\begin{proof}
    The sensitivity $q_u(t)=\partial_u y_u(t)$ satisfies $\dot q_u(t) + d q_u(t) = g_u(t)$ where $g_u(t)$ is given in \eqref{eq:sys3:vec:g}. From \eqref{eq:sys3:vec:lambda}, $\lambda_u\ge 0, \forall t\in[0,T], \forall u > 0$ and $g_u(t)\ge 0$, $\forall t\ge 0, \forall u > 0$. Thus $\lambda_u(t)g_u(t)\ge 0, \forall t\in[0,T]$. By Theorem~\ref{thm:cDR_monotone}(i), $\partial_u \mathrm{cDR}(u,T) \ge 0.$
    Hence $u\mapsto \mathrm{cDR}(u,T)$ is monotone nondecreasing.
\end{proof}

\begin{figure}[h!]
    \centering
    \subfloat[\label{v3a}]{\includegraphics[width=0.3\linewidth]{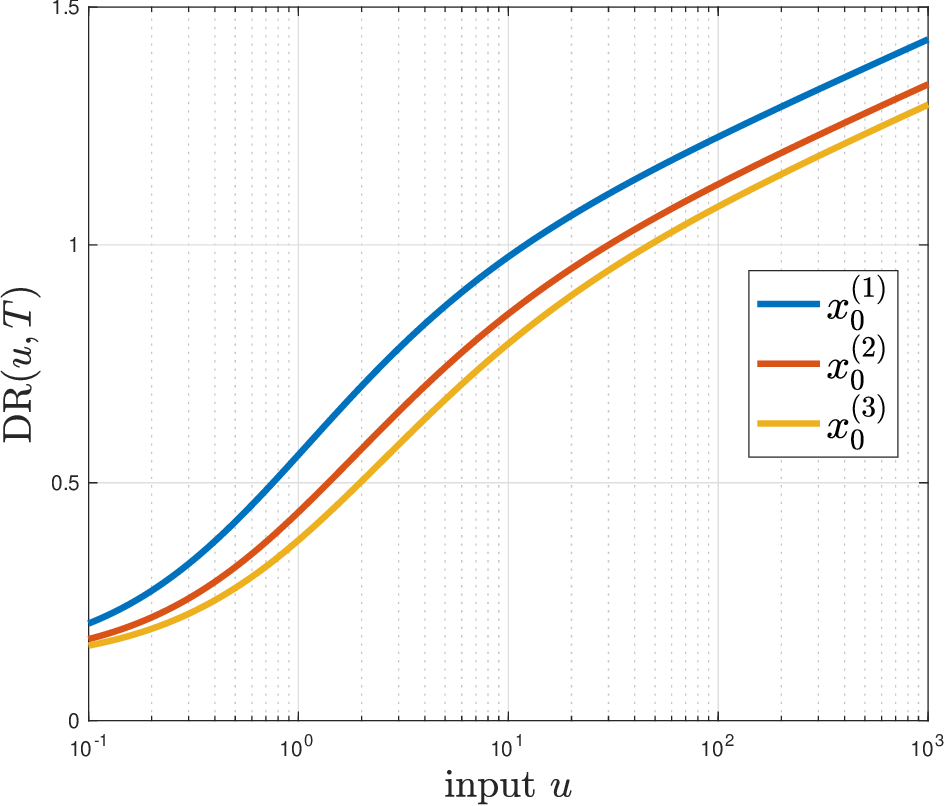}}\qquad
    \subfloat[\label{v3b}]{\includegraphics[width=0.3\linewidth]{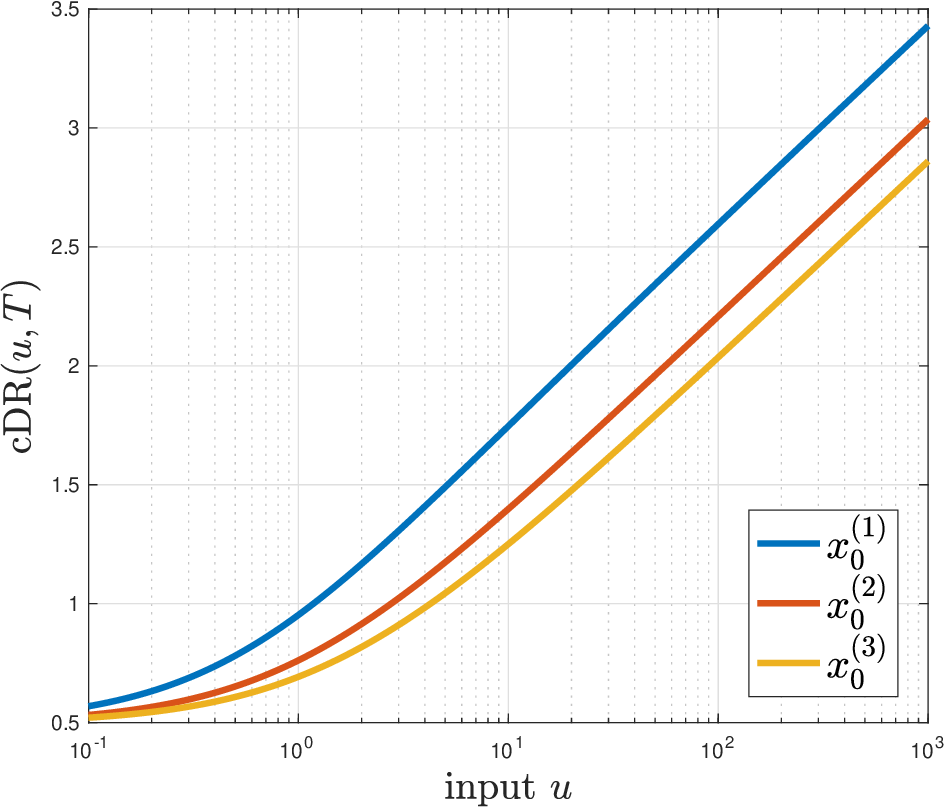}}
    \caption{System~3: Dose response $\mathrm{DR}(u,T)$ and cumulative dose response $\mathrm{cDR}(u,T)$ for $u\in[10^{-1},10^{3}]$. 
    Results are shown for three fixed initial conditions $x_0$. 
    The curves illustrate that $\mathrm{cDR}(u,T)$ is monotone nondecreasing with respect to $u$.}
    \label{fig:v_sys3}
\end{figure}

\begin{subequations}
\paragraph{System~6:} 
Consider
\begin{equation}\label{eq:sys6-vector}
    \dot x_u = A x_u + bu, \qquad
    \dot y_u = \gamma-\frac{dc^\top x_u}{\beta u}y_u,
\end{equation}
where $A\in\mathbb{R}^{n\times n}$ is Metzler and Hurwitz,
$b,c\in\mathbb{R}^n_+$, and $d,\beta,\gamma\in\mathbb{R}_+$.
We assume that the initial condition $x_u(0)=x_0$ is independent of $u$, with $x_0\ge 0$, and that $y_u(0)=y_0$ is also independent of $u$.

Let $q_u:=\partial_u y_u\in\mathbb{R}$. Differentiating the $y$-equation in \eqref{eq:sys6-vector} with respect to $u$ gives
\begin{equation*}
    \dot q_u
    =-\frac{dc^\top p_u}{\beta u}y_u
    +\frac{dc^\top x_u}{\beta u^2}y_u
    -\frac{dc^\top x_u}{\beta u}q_u,
    \qquad q_u(0)=0.
\end{equation*}
Therefore, $q_u$ satisfies $\dot q_u+a_uq_u=g_u$, $q_u(0)=0$, with
\begin{equation*}
    a_u=\frac{dc^\top x_u}{\beta u},\qquad
    g_u=-\frac{dc^\top p_u}{\beta u}y_u
        +\frac{dc^\top x_u}{\beta u^2}y_u.
\end{equation*}
Using \eqref{eq:x_vec_explicit}, we rewrite $g_u$ as
\begin{equation}\label{eq:sys6:vec:g}
    g_u
    =\frac{d\bigl(c^\top x_u-u c^\top p_u\bigr)y_u}{\beta u^2}
    =\frac{dc^\top e^{At}x_0}{\beta u^2}y_u.
\end{equation}
Since $e^{At}\ge 0$, $c\in\mathbb{R}^n_+$, $x_0\ge 0$, and $y_u\ge 0$, we conclude that $g_u\ge 0, \forall t\ge 0, \forall u > 0.$
Integrating $g_u$ over $[0,t]$, we obtain
\begin{equation}\label{eq:sys6:vec:G}
    G_u:=\int_0^t g_u(s)\,\mbox{d}s
    =\frac{d}{\beta u^2}
      \int_0^t c^\top e^{As}x_0\,y_u(s)\,\mbox{d}s
    \geq0.
\end{equation}
Since $a_u=d\,c^\top x_u/(\beta u)\geq0$, the kernel satisfies
\begin{equation}\label{eq:sys6:vec:lambda}
    \lambda_u
    =\int_t^T
      \exp\!\left(
        -\int_t^s\frac{dc^\top x_u(\tau)}{\beta u}
        \,\mbox{d}\tau
      \right)\mbox{d}s.
\end{equation}
Therefore, $\lambda_u\ge 0$ for all $t\in[0,T]$ and $u>0$.
\end{subequations}

\begin{proposition}\label{prop:sys6:vec}
    Consider the system \eqref{eq:sys6-vector}, where $A$ is Metzler
    and Hurwitz, $b,c\in\mathbb{R}^n_+$, and
    $d,\beta,\gamma\in\mathbb{R}_+$. Assume that $x_0\geq0$ and
    $y_0\geq0$ are independent of $u$. Then, for every $T>0$,
    $\mathrm{cDR}(u,T)$ is monotone nondecreasing with respect to $u>0$.
\end{proposition}

\begin{proof}
    The sensitivity $q_u(t)=\partial_u y_u(t)$ satisfies
    $\dot q_u(t)+\dfrac{dc^\top x_u(t)}{\beta u}q_u(t)=g_u(t)$,
    where $g_u(t)$ is given in \eqref{eq:sys6:vec:g}. From \eqref{eq:sys6:vec:lambda}, $\lambda_u\ge 0, \forall t\in[0,T], \forall u > 0$ and $g_u(t)\ge 0$, $\forall t\ge 0, \forall u > 0$. Thus $\lambda_u(t)g_u(t)\ge 0, \forall t\in[0,T]$. By Theorem~\ref{thm:cDR_monotone}(i), $\partial_u \mathrm{cDR}(u,T) \ge 0.$
    Hence $u\mapsto \mathrm{cDR}(u,T)$ is monotone nondecreasing.
\end{proof}

\begin{figure}[h!]
    \centering
    \subfloat[\label{v6a}]{\includegraphics[width=0.3\linewidth]{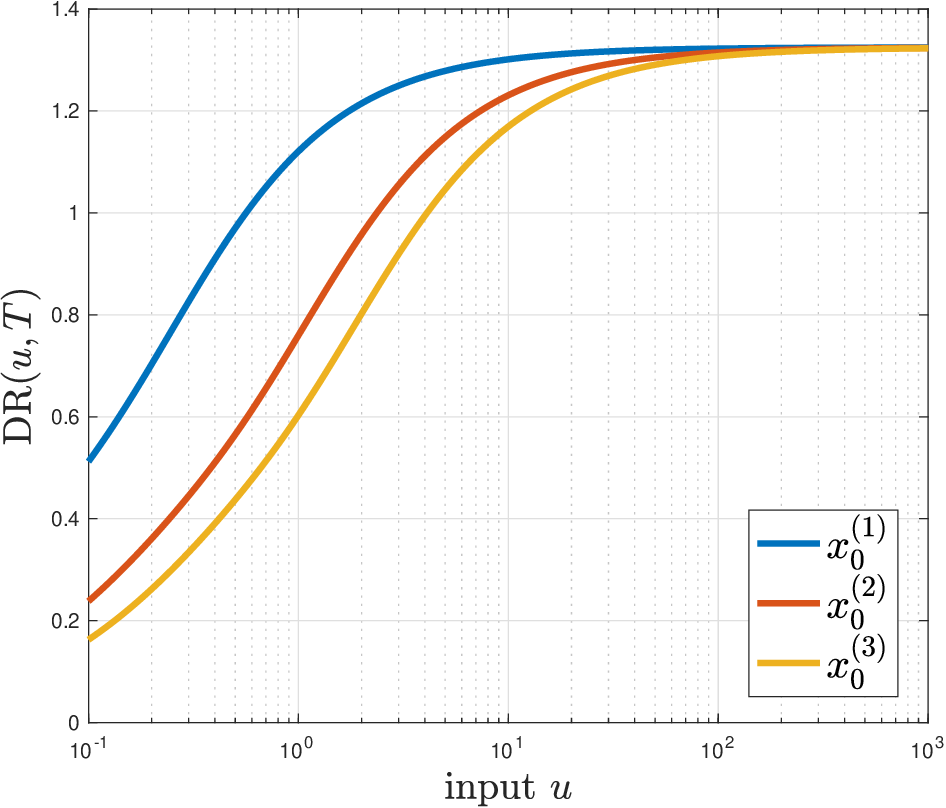}}\qquad
    \subfloat[\label{v6b}]{\includegraphics[width=0.3\linewidth]{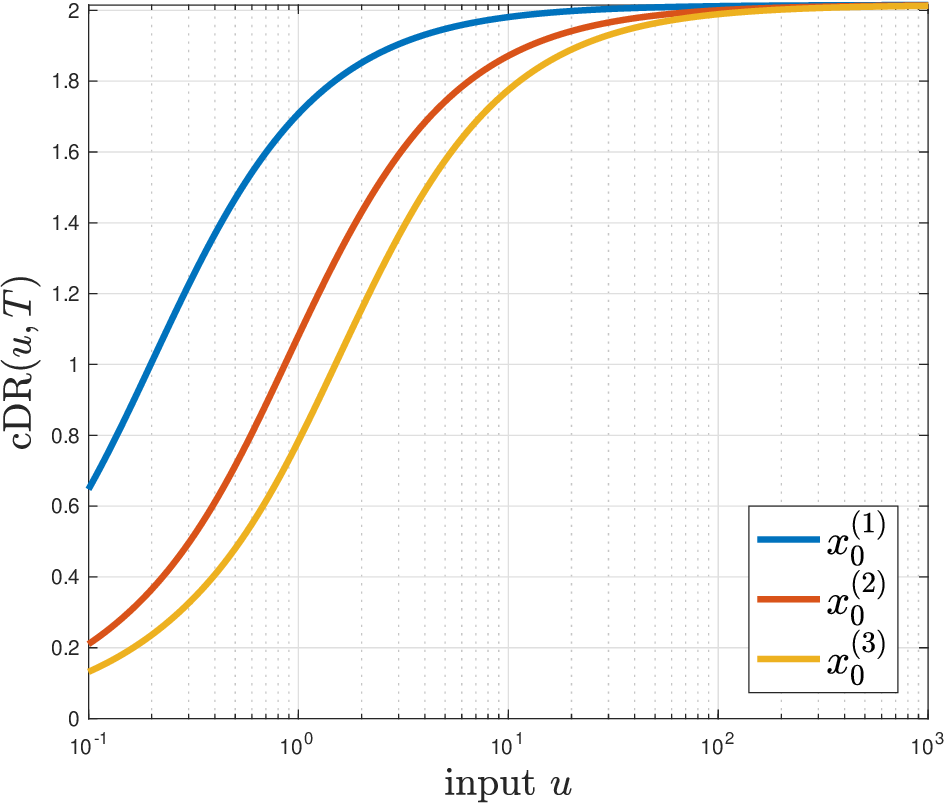}}
    \caption{System 6: Dose response $\mathrm{DR}(u,T)$ and cumulative dose response $\mathrm{cDR}(u,T)$ for $u\in[10^{-1},10^{3}]$. 
    Results are shown for three fixed initial conditions $x_0$. 
    The curves illustrate that $\mathrm{cDR}(u,T)$ is monotone nondecreasing with respect to $u$.}
    \label{fig:v_sys6}
\end{figure}

\begin{subequations}
\paragraph{System~7:} 
Consider
\begin{equation}\label{eq:sys7-vector}
    \dot x_u = A x_u + bu, \qquad
    \dot y_u = \frac{1}{\beta u} - \frac{d y_u}{K+c^\top x_u},
\end{equation}
where $A\in\mathbb{R}^{n\times n}$ is Metzler and Hurwitz, $b,c\in\mathbb{R}^n_+$, and $K,\beta,d\in\mathbb{R}_+$. 
We assume that the initial condition $x_u(0)=x_0$ is independent of $u$, with $x_0\ge 0$, and that $y_u(0)=y_0>0$ is also independent of $u$.

Let $q_u:=\partial_u y_u\in\mathbb{R}$. Differentiating the $y$-equation in \eqref{eq:sys7-vector} with respect to $u$ gives
\begin{equation*}
    \dot q_u = -\frac{1}{\beta u^2}
    +\frac{dc^\top p_u}{(K+c^\top x_u)^2}y_u
    -\frac{d}{K+c^\top x_u}q_u,
    \qquad q_u(0)=0.
\end{equation*}
Therefore, $q_u$ satisfies $\dot q_u + a_u q_u = g_u$, $q_u(0)=0$ with
\begin{equation}\label{eq:sys7:vec:ag}
    a_u=\frac{d}{K+c^\top x_u}, \qquad
    g_u=\frac{dc^\top p_u}{(K+c^\top x_u)^2}y_u-\frac{1}{\beta u^2}.
\end{equation}
Integrating $g_u$ over $[0,t]$, we obtain
\begin{equation}\label{eq:sys7:vec:G}
    G_u:=\int_0^t g_u(s)\,\mbox{d}s
    =\int_0^t\left[\frac{dc^\top p_u(s)}{(K+c^\top x_u(s))^2}y_u(s)-\frac{1}{\beta u^2}\right]ds.
\end{equation}
Since $a_u=d/(K+c^\top x_u)\ge 0$, the kernel $\lambda_u$ defined in Theorem~\ref{thm:cDR_monotone} satisfies
\begin{equation}\label{eq:sys7:vec:lambda}
    \lambda_u = \int_t^T \exp\!\left(-\int_t^s \frac{d}{K+c^\top x_u(\tau)}\,\mbox{d}\tau\right)\,\mbox{d}s.
\end{equation}
Therefore, $\lambda_u\geq0$ for all $t\in[0,T]$ and $u>0$.
\end{subequations}

We analyze the sign of $\partial_u\mathrm{cDR}(u,T)$ for small
and large $u$.

\begin{itemize}
    \item For small $u$, variation of parameters in the
    $y$--equation gives
    \[
        0\leq y_u(t)\leq y_0+\frac{T}{\beta u},
        \qquad t\in[0,T].
    \]
    Let $R_T:=\max_{t\in[0,T]}c^\top p_u(t)<\infty$.
    Since $K+c^\top x_u(t)\geq K$, \eqref{eq:sys7:vec:ag}
    implies, for every $t\in[0,T]$,
    \[
        g_u(t)
        \leq
        \frac{dR_T}{K^2}
        \left(y_0+\frac{T}{\beta u}\right)
        -\frac{1}{\beta u^2}.
    \]
    Multiplying the right-hand side by $u^2$ gives
    \[
        \frac{dR_T}{K^2}
        \left(y_0u^2+\frac{T}{\beta}u\right)
        -\frac{1}{\beta},
    \]
    which is negative for all sufficiently small $u>0$.
    Thus $g_u(t)<0$ throughout $[0,T]$. Since
    $\lambda_u(t)>0$ for $t<T$, \eqref{eq:int_by_part}
    yields
    \[
        \partial_u\mathrm{cDR}(u,T)<0
    \]
    for all sufficiently small $u>0$.

    \item For large $u$, set
    \[
        r(t):=c^\top p_u(t),\qquad
        D(t):=K+c^\top e^{At}x_0.
    \]
    Then $K+c^\top x_u(t)=D(t)+ur(t)$,
    $r(0)=0$, and $\dot r(0)=c^\top b>0$.
    Moreover, $r(t)>0$ for $t>0$ and $D(t)\geq K$.
    Consequently, for fixed $T>0$,
    \begin{equation}\label{eq:sys7:vec:h-bound}
        \int_0^T\frac{d}{D(t)+ur(t)}\,\mbox{d}t
        =
        O\!\left(\frac{\log u}{u}\right).
    \end{equation}
    Indeed, $r(t)\geq mt$ on some $[0,\delta]$,
    whereas $r(t)$ has a positive minimum on $[\delta,T]$.
    
    Variation of parameters in the $y$--equation and
    \eqref{eq:sys7:vec:h-bound} give, uniformly for $t\in[0,T]$,
    \begin{equation*}
        y_u(t)=y_0+
        O\!\left(\frac{\log u}{u}\right).
    \end{equation*}
    Similarly, \eqref{eq:sys7:vec:lambda} and
    \eqref{eq:sys7:vec:h-bound} give
    \begin{equation*}
        \lambda_u(t)=T-t+
        O\!\left(\frac{\log u}{u}\right)
    \end{equation*}
    uniformly on $[0,T]$. Also,
    \begin{equation*}
        \int_0^T
        \frac{r(t)}{[D(t)+ur(t)]^2}\,\mbox{d}t
        =
        O\!\left(\frac{\log u}{u^2}\right).
    \end{equation*}
    Thus, using \eqref{eq:int_by_part} and
    \eqref{eq:sys7:vec:ag},
    \begin{align}
        \partial_u\mathrm{cDR}(u,T)
        &=
        d y_0\int_0^T
        \frac{(T-t)r(t)}
             {[D(t)+ur(t)]^2}\,\mbox{d}t
        +O(u^{-2}).
        \label{eq:sys7:vec:leading-integral}
    \end{align}
    
    We evaluate the integral in
    \eqref{eq:sys7:vec:leading-integral}.
    Choose $\delta\in(0,T)$ small enough that
    $\dot r(t)>0$ on $[0,\delta]$. On this interval,
    set $z=r(t)$ and write $t=t(z)$. Since
    $\dot r(0)=c^\top b$ and $D(t)=D(0)+O(t)$,
    \begin{align*}
        \int_0^\delta
        \frac{(T-t)r(t)}
             {[D(t)+ur(t)]^2}\,\mbox{d}t
        &=
        \int_0^{r(\delta)}
        \frac{(T-t(z))z}
             {\dot r(t(z))[D(t(z))+uz]^2}
        \,\mbox{d}z\\
        &=
        \frac{T}{c^\top b}
        \int_0^{r(\delta)}
        \frac{z}{[D(0)+uz]^2}\,\mbox{d}z
        +O(u^{-2}).
    \end{align*}
    The remaining integral is explicit:
    \begin{align*}
        \int_0^{r(\delta)}
        \frac{z}{[D(0)+uz]^2}\,\mbox{d}z
        &=
        \frac{1}{u^2}
        \left[
            \log\!\left(1+\frac{ur(\delta)}{D(0)}\right)
            +\frac{D(0)}{D(0)+ur(\delta)}-1
        \right]\\
        &=
        \frac{\log u}{u^2}+O(u^{-2}).
    \end{align*}
    On $[\delta,T]$, $r$ is bounded away from zero,
    so that part of the integral is $O(u^{-2})$.
    Therefore, \eqref{eq:sys7:vec:leading-integral} yields
    \begin{equation}\label{eq:sys7:vec:cDR-large-u}
        \partial_u\mathrm{cDR}(u,T)
        =
        \frac{d y_0 T}{c^\top b}
        \frac{\log u}{u^2}
        +O(u^{-2})>0
    \end{equation}
    for all sufficiently large $u$, since $y_0>0$.
\end{itemize}

\begin{proposition}\label{prop:sys7:vec}
    Consider the system \eqref{eq:sys7-vector}, where $A$ is
    Metzler and Hurwitz, $b,c\in\mathbb{R}^n_+$, and
    $K,\beta,d\in\mathbb{R}_+$. Assume that $x_0\geq0$ and
    $y_0>0$ are independent of $u$. Then, for every $T>0$,
    there exist $u_-,u_+>0$ such that
    \begin{equation*}
        \partial_u\mathrm{cDR}(u_-,T)<0,
        \qquad
        \partial_u\mathrm{cDR}(u_+,T)>0.
    \end{equation*}
    In particular, $\mathrm{cDR}(u,T)$ is not monotone
    with respect to $u$.
\end{proposition}

\begin{proof}
For sufficiently small $u$, the estimate $g_u<0$ above and
\eqref{eq:int_by_part} give
$\partial_u\mathrm{cDR}(u,T)<0$. For sufficiently large $u$,
\eqref{eq:sys7:vec:cDR-large-u} gives
$\partial_u\mathrm{cDR}(u,T)>0$. Thus, choosing $u_-$ sufficiently
small and $u_+$ sufficiently large gives the stated opposite signs.
Hence, $\mathrm{cDR}(u,T)$ is not monotone with respect to $u$.
\end{proof}

\begin{figure}[h!]
    \centering
    \subfloat[\label{v7a}]{\includegraphics[width=0.3\linewidth]{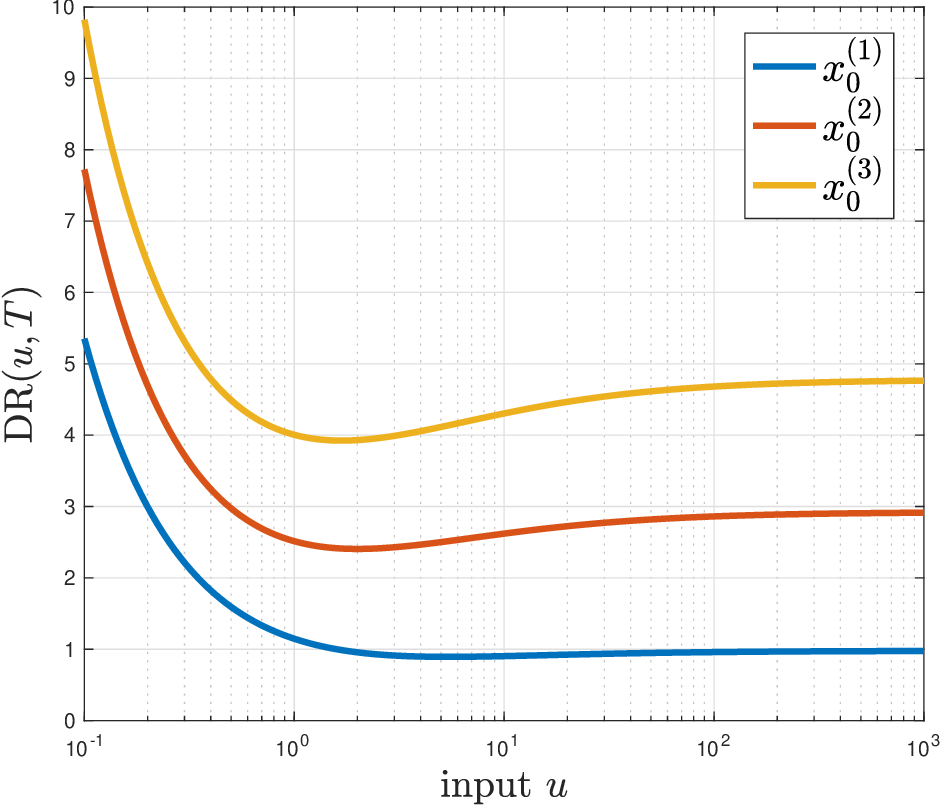}}\qquad
    \subfloat[\label{v7b}]{\includegraphics[width=0.3\linewidth]{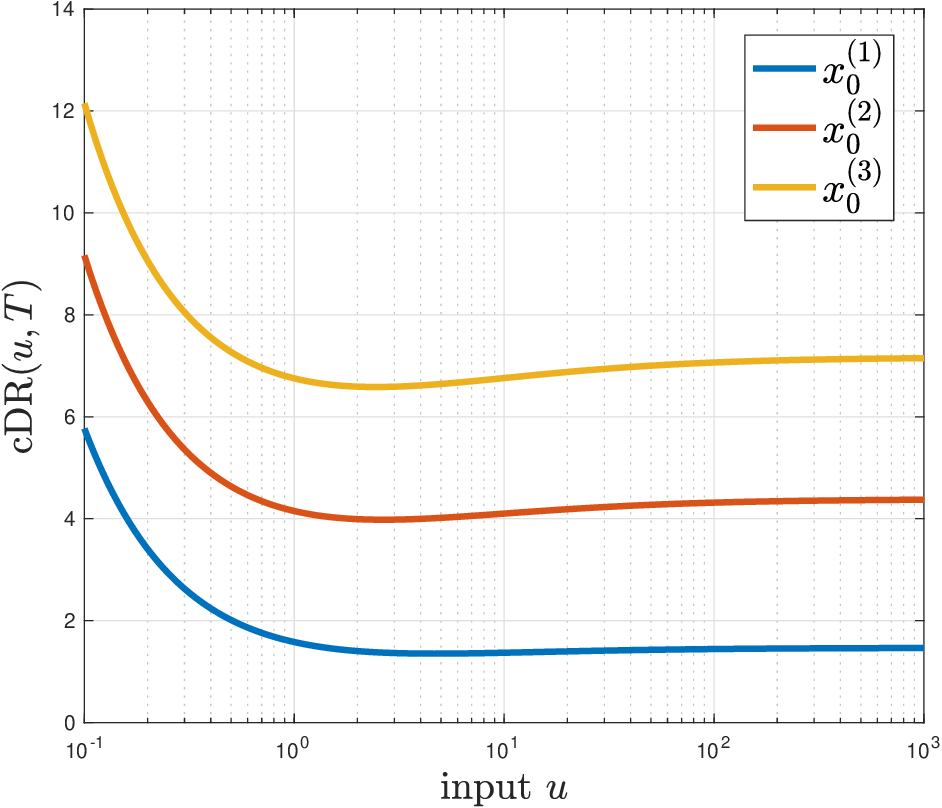}}
    \caption{System 7: Dose response $\mathrm{DR}(u,T)$ and cumulative dose response $\mathrm{cDR}(u,T)$ for $u\in[10^{-1},10^{3}]$. 
    Results are shown for three fixed initial conditions $x_0$. 
    The curves illustrate that $\mathrm{cDR}(u,T)$ is not monotone with respect to $u$, exhibiting a change in monotonicity.}
    \label{fig:v_sys7}
\end{figure}

\begin{subequations}
\paragraph{System~8:}
Consider
\begin{equation}\label{eq:sys8-vector}
    \dot x_u = A x_u + bu, \qquad
    \dot y_u = \gamma-\frac{d\beta u}{K+c^\top x_u}y_u,
\end{equation}
where $A\in\mathbb{R}^{n\times n}$ is Metzler and Hurwitz,
$b,c\in\mathbb{R}^n_+$, and $d,K,\beta,\gamma\in\mathbb{R}_+$.
We assume that the initial condition $x_u(0)=x_0$ is independent of
$u$, with $x_0\geq0$, and that $y_u(0)=y_0\geq0$ is also independent
of $u$.

Let $q_u:=\partial_u y_u\in\mathbb{R}$. Differentiating the
$y$-equation in \eqref{eq:sys8-vector} with respect to $u$ gives
\begin{equation*}
    \dot q_u
    =-\frac{d\beta}{K+c^\top x_u}y_u
    +\frac{d\beta u c^\top p_u}{(K+c^\top x_u)^2}y_u
    -\frac{d\beta u}{K+c^\top x_u}q_u,
    \qquad q_u(0)=0.
\end{equation*}
Therefore, $q_u$ satisfies $\dot q_u+a_uq_u=g_u$,
$q_u(0)=0$, with
\begin{equation*}
    a_u=\frac{d\beta u}{K+c^\top x_u}, \qquad
    g_u=-\frac{d\beta}{K+c^\top x_u}y_u
    +\frac{d\beta u c^\top p_u}{(K+c^\top x_u)^2}y_u.
\end{equation*}
Using \eqref{eq:x_vec_explicit}, we rewrite $g_u$ as
\begin{equation}\label{eq:sys8:vec:g}
    g_u
    =\frac{d\beta y_u}{(K+c^\top x_u)^2}
      \bigl(uc^\top p_u-K-c^\top x_u\bigr)
    =-\frac{d\beta y_u}{(K+c^\top x_u)^2}
      \bigl(K+c^\top e^{At}x_0\bigr).
\end{equation}
Since $K>0$, $e^{At}\geq0$, $c\in\mathbb{R}^n_+$,
$x_0\geq0$, and $y_u\geq0$, we conclude that $g_u\leq0$
for all $t\geq0$ and $u>0$. Integrating $g_u$ over $[0,t]$,
we obtain
\begin{equation}\label{eq:sys8:vec:G}
    G_u:=\int_0^t g_u(s)\,\mbox{d}s
    =-\int_0^t
      \frac{d\beta y_u(s)}
           {(K+c^\top x_u(s))^2}
      \bigl(K+c^\top e^{As}x_0\bigr)\,\mbox{d}s
    \leq0.
\end{equation}
Since $a_u=d\beta u/(K+c^\top x_u)\geq0$, the kernel
$\lambda_u$ defined in Theorem~\ref{thm:cDR_monotone} satisfies
\begin{equation}\label{eq:sys8:vec:lambda}
    \lambda_u
    =\int_t^T
      \exp\!\left(
        -\int_t^s
        \frac{d\beta u}{K+c^\top x_u(\tau)}
        \,\mbox{d}\tau
      \right)\mbox{d}s.
\end{equation}
Therefore, $\lambda_u\geq0$ for all $t\in[0,T]$ and $u>0$.
\end{subequations}

\begin{proposition}\label{prop:sys8:vec}
    Consider the system \eqref{eq:sys8-vector}, where $A$ is Metzler
    and Hurwitz, $b,c\in\mathbb{R}^n_+$, and
    $d,K,\beta,\gamma\in\mathbb{R}_+$. Assume that $x_0\geq0$
    and $y_0\geq0$ are independent of $u$. Then, for every $T>0$,
    $\mathrm{cDR}(u,T)$ is monotone nonincreasing with respect to $u>0$.
\end{proposition}

\begin{proof}
    The sensitivity $q_u(t)=\partial_u y_u(t)$ satisfies
    \[
        \dot q_u(t)
        +\frac{d\beta u}{K+c^\top x_u(t)}q_u(t)
        =g_u(t),
    \]
    where $g_u(t)$ is given in \eqref{eq:sys8:vec:g}.
    From \eqref{eq:sys8:vec:lambda}, $\lambda_u(t)\geq0$
    on $[0,T]$, while $g_u(t)\leq0$ by \eqref{eq:sys8:vec:g}.
    Thus $\lambda_u(t)g_u(t)\leq0$ on $[0,T]$.
    By Theorem~\ref{thm:cDR_monotone}(i),
    $\partial_u\mathrm{cDR}(u,T)\leq0$.
    Hence $u\mapsto\mathrm{cDR}(u,T)$ is monotone nonincreasing.
\end{proof}

\begin{figure}[t!]
    \centering
    \subfloat[\label{v8a}]{\includegraphics[width=0.3\linewidth]{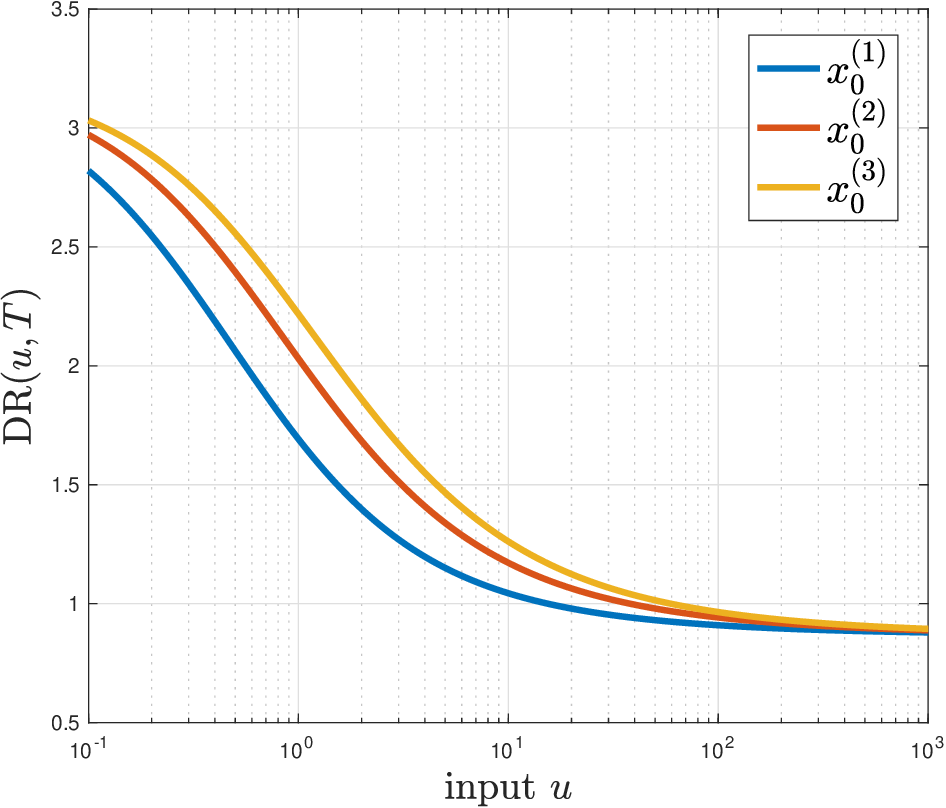}}\qquad
    \subfloat[\label{v8b}]{\includegraphics[width=0.3\linewidth]{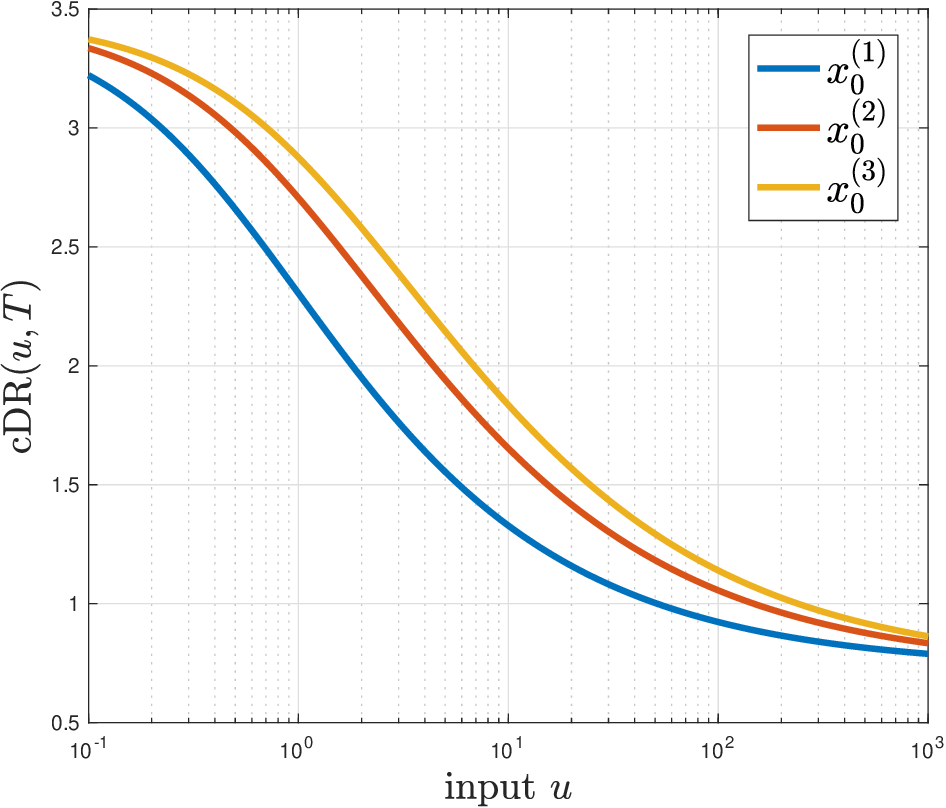}}
    \caption{System~8: Dose response $\mathrm{DR}(u,T)$ and cumulative dose response $\mathrm{cDR}(u,T)$ for $u\in[10^{-1},10^{3}]$.
    Results are shown for three fixed initial conditions $x_0$.
    The curves illustrate that $\mathrm{cDR}(u,T)$ is monotone nonincreasing with respect to $u$.}
    \label{fig:v_sys8}
\end{figure}

\end{document}